\documentclass[reqno,11pt]{amsart}
\numberwithin{equation}{section}

\usepackage[utf8]{inputenc}
\usepackage[italian,english]{babel}
\usepackage{amsmath}
\usepackage{amsfonts}
\usepackage{amssymb,esint}
\usepackage{tikz}
\usepackage{bigints}
\usepackage{comment}
\usepackage{geometry}
\usepackage{color}
\usepackage{mathrsfs}
\usepackage{cancel}
\usepackage{mathtools}

\usepackage{calrsfs}

\mathtoolsset{showonlyrefs}

\usepackage[colorlinks, citecolor=citegreen, linkcolor=refred]{hyperref}
\definecolor{citegreen}{rgb}{0,0.3,0}
\definecolor{refred}{rgb}{0.5,0,0}

\theoremstyle{plain}
\newtheorem {theorem}{Theorem}[section]
\newtheorem {lemma}[theorem]{Lemma}
\newtheorem {proposition} [theorem]{Proposition}
\newtheorem {corollary} [theorem]{Corollary}

\newtheorem{definition}[theorem]{Definition}
\newtheorem{remark}[theorem]{Remark}

\theoremstyle{remark}

\DeclarePairedDelimiter\abs{\lvert}{\rvert}
\DeclarePairedDelimiter\bigabs{\Big\lvert}{\Big\rvert}
\DeclarePairedDelimiter\norm{\lVert}{\rVert}

\newcommand\ddfrac[2]{\frac{\displaystyle #1}{\displaystyle #2}}

\newcommand{\R}{\mathbb R}
\newcommand{\N}{\mathbb N}
\newcommand{\Z}{\mathbb Z}

\newcommand{\e}{\varepsilon}
\renewcommand{\theta}{\vartheta}

\newcommand{\GR}{\mathcal{G}_{R, kR}}
\newcommand{\FR}{\mathcal{F}_{R, kR}}
\newcommand{\capo}{\text{Cap}(\Omega)}
\newcommand{\crit}{\text{Crit}(\phi)}

\newcommand{\barint}
{\rule[.036in]{.12in}{.009in}\kern-.16in \displaystyle\int}

\newcommand{\A}{\mathcal {A}}

\newcommand{\dive}{{\mathrm{div}}}

\newcommand{\ep}{\varepsilon}
\newcommand{\pa}{\partial }

\usepackage{color}

\newcommand{\numberset}{\mathbb}
\renewcommand{\N}{\numberset{N}}
\renewcommand{\R}{\numberset{R}}
\newcommand{\Sf}{\numberset{S}}

\renewcommand{\A}{\mathscr{A}}

\newcommand{\diag}{\mathop {\rm Diag}\nolimits}

\newcommand{\ric}{\mathop {\rm Ric}\nolimits}

\newcommand{\AVR}{{\rm AVR}(g)}
\newcommand{\capa}{{\rm Cap}}
\newcommand{\D}{{\rm D}}
\newcommand{\dd}{{\,\rm d}}
\newcommand{\HH}{{\rm H}}
\newcommand{\ffi}{\varphi}
\newcommand{\na}{\nabla}
\newcommand{\cg}{\tilde g}
\newcommand{\Om}{\Omega}
\renewcommand{\phi}{\varphi}
\renewcommand{\epsilon}{\varepsilon}

\newenvironment{comm}{\color{red}}{\color{black}}
\newenvironment{bcomm}{\color{blue}}{\color{black}}

\title[Sharp geometric inequalities for closed hypersurfaces]{Sharp geometric inequalities for closed hypersurfaces in manifolds
with nonnegative Ricci curvature.}

\author[V.~Agostiniani]{Virginia Agostiniani}
\address{V.~Agostiniani, Universit\`a degli Studi di Verona, strada Le Grazie 15, 37134 Verona, Italy} 
\email{virginia.agostiniani@univr.it}

\author[M.~Fogagnolo]{Mattia Fogagnolo}
\address{M.~Fogagnolo, Universit\`a degli Studi di Trento,
via Sommarive 14, 38123 Povo (TN), Italy}
\email{mattia.fogagnolo@unitn.it}

\author[L.~Mazzieri]{Lorenzo Mazzieri}
\address{L.~Mazzieri, Universit\`a degli Studi di Trento,
via Sommarive 14, 38123 Povo (TN), Italy}
\email{lorenzo.mazzieri@unitn.it}

\begin{document}

\maketitle


\begin{abstract}
In this paper we consider complete noncompact Riemannian  manifolds $(M, g)$ with nonnegative Ricci curvature and Euclidean volume growth, of dimension $n \geq 3$. For every bounded open subset $\Omega \subset M$ with smooth boundary, we prove that
\[
\int\limits_{\partial \Om} \left|\frac{\rm H}{n-1}\right|^{n-1} 
\!\!\!\!\!{\rm d}\sigma 
\,\,\geq
\,\,\AVR\,\big|\mathbb{S}^{n-1}\big| \,\,,
\]
where ${\rm H}$ is the mean curvature of $\partial \Omega$ and $\AVR$ is the asymptotic volume ratio of $(M,g)$.
Moreover, the equality holds true if and only if $(M{\setminus}\Omega, g)$ is isometric to a truncated cone
over $\partial\Omega$. An optimal version of Huisken's Isoperimetric Inequality for $3$-manifolds is obtained using this result. 
Finally, exploiting a natural extension of our techniques to the case of parabolic manifolds, we also deduce an enhanced version of Kasue's non existence result for closed minimal hypersurfaces in manifolds with nonnegative Ricci curvature.
\end{abstract}


\section{Introduction and main results}

The classical Willmore inequality \cite{Willmore}
for a bounded domain $\Om$ of $\R^3$ with smooth boundary
says that
\begin{equation}
\label{eq:will}
\int\limits_{ \pa\Om}  
{\HH}^{\,2}  
\dd\sigma
\,\geq\,
16\pi,  
\end{equation}
where $\HH$ is the mean curvature of $\pa \Omega$. Such an inequality has been extended in \cite{Chen_2} to submanifolds of any co-dimension in $\R^n$, for $n \geq 3$. In particular, for a bounded domain $\Omega$ in $\R^n$ with smooth boundary there holds
\begin{equation}
\label{will-chen}
\int\limits_{ \pa\Om}  
\left|\frac{\HH}{n-1}\right|^{n-1}  
\!\!\!\!\dd\sigma
\,\geq\,
|\Sf^{n-1}| \, ,
\end{equation}
with equality attained if and only if $\Omega$ is a ball.
Implicit in this statement is the fact
that the underlying metric by which
$\HH$ and $\dd\sigma$
are computed is the Euclidean metric $g_{\R^n}$.
Note that the above rigidity statement 
can be rephrased by saying that the equality in \eqref{eq:will}
is fulfilled if and only if 
$(\pa\Om,{g_{\pa\Omega}})$
is homothetic to
$(\Sf^{n-1},g_{\Sf^{n-1}})$, where $g_{\pa\Omega}$ is the metric induced by $g_{\R^n}$ on the submanifold $\pa\Omega$ and $g_{\Sf^{n-1}}$
is the standard round metric.

Recently, in \cite{Ago_Maz_3}, 
the Willmore-type inequality \eqref{will-chen} and the corresponding
rigidity statement 
have been deduced as a 
consequence of suitable monotonicity-rigidity
properties of the function 
\begin{equation}
\label{U_Will}
U(t)
\,:=\,
t^{-(n-1)}
\!\!\!
\int\limits_{\{u=t\}}
\!\!
|\D u|^{n-1}\dd\sigma,
\qquad\qquad 
t\in(0,1],
\end{equation}
associated with the level set
flow of the electrostatic potential $u$
generated by the uniformly
charged body $\Omega$.
In other words, 
$u$ is the unique harmonic
function in $\R^n\setminus
\overline\Omega$
which vanishes at infinity and
such that $u=1$ on $\pa\Omega$.
More precisely, what is proven in \cite{Ago_Maz_3}
is that the function $U$ is nondecreasing
and that this monotonicity is strict
unless $\Om$ is a ball.
Once this fact is established, the proof
of \eqref{will-chen} consists of a few lines.
Indeed, exploiting first the global
feature of the monotonicity 
\big(i.e. $U(1)\geq U(0^+)$\big)
and using then the asymptotic expansion
at infinity of $u$ and $|\D u|$
one gets
\begin{equation}
\label{dis_interm}
\int\limits_{ \pa\Om}  
|\D u|^{n-1}\dd\sigma
\,=:\,
U(1)
\,\geq\,
\lim_{t\to0^+}
U(t)
\,=\,
(n-2)^{n-1}
|\Sf^{n-1}|.
\end{equation}
On the other hand, 
computing the derivative at $t = 1$
\begin{equation}
\label{pass_2}
U'(1)
\,=\,
(n-2)
\int\limits_{\pa\Om} 
|\D u|^{n-2}
\left[\, \HH-
\big(\tfrac{n-1}{n-2}\big)
\, {|\D u|}
\, \right] 
\dd \sigma \, ,
\end{equation}
and using  $U'(1)\geq0$,
we deduce that
\begin{equation}
\label{pass_3}
\int\limits_{ \pa\Om}  
\left|\frac{\HH}{n-1}\right|^{n-1}
\!\!\!\dd\sigma
\,\,\geq\,\,
\int\limits_{ \pa\Om}  
\left|\frac{\D u}{n-2}\right|^{n-1}
\!\!\!\dd\sigma, 
\end{equation}
where we have also applied
the H\"older inequality.
The coupling of the 
latter inequality 
with the former \eqref{dis_interm}
yields the desired \eqref{will-chen}.
In this paper, we show that the 
strategy described above
can be adapted to a much more general setting, 
giving rise to a new Willmore-type inequalities, holding on manifolds with nonnegative Ricci curvature of dimension greater than three. 

\bigskip

\noindent {\em Throughout this paper, we systematically assume that the dimension of the underlying manifold is at least $3$.}

\bigskip

\noindent Our main result reads: 

\begin{theorem}
[Willmore-type inequality]
\label{willth}
Let $(M, g)$ be a 
complete noncompact Riemannian manifold 
with $\ric\geq 0$ and Euclidean volume growth.
If $\Omega\subset M$ is a bounded and open subset with smooth boundary, then
\begin{equation}
\label{will}
\int\limits_{\partial\Omega}\,\left\vert{\frac\HH{n-1}}\right\vert^{n-1}\!\!\!\!\!\!\dd
\sigma
\,\geq\,
\AVR\abs{\Sf^{n-1}} \, 
,
\end{equation}
where $\AVR\in(0,1]$ is the 
asymptotic volume ratio of $(M,g)$.
Moreover, the equality holds if and only if 
$(M\setminus\Omega,g)$ 
is isometric to 
\begin{equation}
\label{metric_cone}
\Big(\,\big[r_0, +\infty) \times\pa\Omega
\,,\, 
\dd r\otimes\!\dd r +(r/r_0)^2 g_{\pa\Omega}
\Big),
\qquad
\mbox{with}
\quad
r_0
\,=\,
\bigg(\frac{|\pa\Omega|}{\AVR|\Sf^{n-1}|}\bigg)^{\frac1{n-1}}.
\end{equation}
In particular, $\partial \Omega$ is a connected totally umbilic submanifold with constant mean curvature.
\end{theorem}
Let us remark that no connectedness assumption on $\Omega$ is required in the above statement. On the other hand, if the equality holds in~\eqref{will}, then we obtain that $\partial \Omega$ is connected as a by-product of the rigidity statement combined with the fact that $(M, g)$ has one end, see Proposition \ref{oneend}.

We recall that since $\ric\geq0$ in the above statement,
then, by the classical Bishop-Gromov 
Volume Comparison Theorem,
the function
\begin{equation}
\label{def:vol_ratio}
(0,+\infty)
\ni
r\longmapsto\Theta(r)=\frac{n|B(p,r)|}{r^n|\Sf^{n-1}|}
\end{equation}
is nonincreasing.
In particular, we have that
the asymptotic volume ratio
\[
\AVR =\lim_{r\to+\infty}\Theta(r)
\]
is well defined.
Moreover, we have that this limit
does not depend on the point $p\in M$,
and that 
$\lim_{r\to0^+}\Theta(r)=1$.
Hence, we have that
$0\leq\AVR\leq1$. Moreover, $\AVR = 1$ if and only if $(M, g)= (\R^n, g_{\R^n})$.
Assuming Euclidean volume growth
means assuming  $\AVR>0$.
Observe in particular that for $n = 3$, if $\ric \equiv 0$ then $(M, g)$ is isometric to $(\R^3, g_{\R^3})$ and consequently $\AVR = 1$. On the other hand, for $n \geq 4$ there exists an important class of complete noncompact Ricci flat Riemannian manifolds with $0 <\AVR <1$, that is the class of Ricci flat \emph{Asymptotically Locally Euclidean}  (ALE for short) manifolds. We refer the reader to Definition \ref{def:ALE} for the precise notion. For the time being, we just recall
that a $n$-dimensional Riemannian
manifold is ALE if it is asymptotic 
to $\big((\R^n\setminus\{0\})/\Gamma\,,\,g_{\R^n}\big)$,
where $\Gamma$  is a finite subgroup of ${\rm SO}(n)$
acting freely on $\R^n\setminus\{0\}$. This family of Riemannian manifolds is widely studied. In this regard, we first mention  that in \cite{bando-kasue-nakajima} it is proved that any Ricci flat manifold with Euclidean volume growth and strictly faster than quadratic curvature decay is actually ALE.
Moreover, we point out that $4$-dimensional Ricci flat ALE manifolds appear as important examples of \emph{gravitational instantons}, that are noncompact hyperkh\"aler $4$-manifolds with decaying curvature at infinity,  introduced by Hawking in \cite{hawking_instantons} in the framework of his Euclidean quantum gravity theory. An explicit example  is given by the famous Eguchi-Hanson metric, introduced in \cite{eguchi-hanson}, where $n = 4$, $\ric \equiv 0$ and $\Gamma = \Z_2$. We remark that ALE gravitational instantons are completely classified in \cite{kronheimer1} and \cite{kronheimer2}. Concerning the general class of gravitational instantons, let us cite, after the important works of Minerbe \cite{minerbe1, minerbe2, minerbe3},  the recent PhD thesis \cite{chen-instantons}, where gravitational instantons with strictly faster than quadratic curvature decay are classified. We refer the reader to the latter work and to the references therein for a more complete picture on this subject. The following corollary is the direct application of Theorem~\ref{willth} to the class ALE manifolds with nonnegative Ricci curvature.

\begin{corollary}
\label{cor:ALE}
Let $(M, g)$ be an ALE 
Riemannian manifold with $\ric\geq 0$. Then,
\begin{equation}
\label{will-ale}
\inf\left\{ \left.\,\,\int\limits_{\partial\Omega}\left\vert{\frac\HH{n-1}}\right\vert^{n-1}\!\!\!\!\dd
\sigma\,\, \right| \,\, \Omega \subset M\,\, \text{\emph{bounded and smooth}}\right\} \, = \,\, \frac{\vert\Sf^{n-1}\vert}{\text{\emph{card}}\,\Gamma} \,.
\end{equation}
Moreover, if the infimum is attained by some $\Omega$, then $M\setminus\Omega$ is isometric to 
\begin{equation}
\label{cone-ale}
\Big(\big[r_0, +\infty) \times\big(\Sf^{n-1}/\Gamma\big)
\,,\, 
\dd r\otimes\!\dd r + r^2 g_{\Sf^{n-1}/\Gamma}\Big),
\qquad
\mbox{with}
\quad
r_0
\,=\,
\left(\frac{\text{\emph{card}}\,\Gamma\abs{\partial \Omega}}{\vert\Sf^{n-1}\vert}\right)^{\!\frac{1}{n-1}},
\end{equation}
for some $r_0>0$ and some finite subgroup 
$\Gamma$ of ${\rm SO}(n)$.
In particular, $(\pa\Omega,g_{\pa\Omega})$
is homothetic to $\big(\Sf^{n-1}/\Gamma,g_{\Sf^{n-1}/\Gamma}\big)$.
\end{corollary}
We observe at once that on ALE manifolds the rigidity is much  stronger, being characterized by cones whose cross sections are homothetic to $\Sf^{n-1}/\Gamma$. Notice also that if $\Gamma$ is trivial one recovers the classical Willmore inequality \eqref{will-chen}.
Moreover, we point out that \eqref{will-ale} also says that, in every ALE manifold with nonnegative Ricci curvature, the lower bound we find for the Willmore-type functional is actually an infimum. This fact holds true for a larger class of manifolds. Indeed, as proved in Theorem \ref{improve}, it is sufficient to assume the hypotheses of Theorem \ref{willth} together with a quadratic curvature decay condition. Understanding metric and topological consequences of curvature decay conditions is a very interesting and widely studied problem in geometric analysis. Dropping any attempt of being complete, we refer the interested reader  to the aforementioned~\cite{bando-kasue-nakajima}, to the seminal~\cite{cheeger-gromov-taylor}, to~\cite{reiris}, where the case $n=3$  is considered, and to~\cite{yeganefar} and the references therein.

To make the picture more complete,
let us also mention that Willmore-type inequalities are 
proven in \cite{Ago_Maz_2}
for asymptotically flat (AE) static metrics 
in the framework of General Relativity, and in \cite{schulze2} for integral $2$-varifolds in Cartan-Hadamard manifolds.
\smallskip

Theorem \ref{willth} will be deduced as a consequence of 
the monotonocity-rigidity properties of the function $U$
defined as in \eqref{U_Will},
where now $u$ is the unique solution to the problem
\begin{equation}
\label{pb}
\begin{cases}
\,\,\Delta{u}= 0 & \mbox{in} \,\, M\setminus\overline\Omega 
\\
\,\,\,\,\,\,\,u=1 & \mbox{on}\,\, \partial\Omega 
\\
u(q)\to 0 &\mbox{as} \,\, d(O,q)\to+\infty,
\end{cases}
\end{equation}
with $O$ being a fixed point in $\Omega$ and 
$d$ a distance function on $M$.
Observe that the hypotheses of Theorem \ref{willth},
namely the nonnegativity of the Ricci tensor and 
the Euclidean volume growth,
guarantee the existence of the solution to problem \eqref{pb},
as explained in Section 2. 
Once that the monotonicity-rigidity of $U$ is known, 
 the proof of Theorem \ref{willth}
consists of exactly the same formal steps
outlined in \eqref{dis_interm}--\eqref{pass_3}, the main difference being the careful computation of the limit value
\[
\lim_{t\to 0^+}U(t)
\,=\,
\AVR(n-2)^{n-1}
|\Sf^{n-1}|.
\]  
We remark that
whereas in the classical Euclidean context the  limit 
was deduced from the pointwise asymptotic expansion of $u$
and  $|\D u|$ at infinity,
here it will be deduced from some quite delicate integral asymptotic expansions, in the spirit of \cite{colding3}. This is an important technical difference from the Euclidean case. Such integral asymptotics will be worked out  in Section 4. 
For completeness, we state the monotonicity-rigidity
result concerning $U$ in the wider generality
of Theorem \ref{thm:M/R} below. 
Indeed, the same monotonicity-rigidity properties
are shared by the whole family
of functions $\{U_\beta\}$, with $\beta\geq (n-2)/(n-1)$,
where $U_\beta:  (0,1] \, \longrightarrow \, \R $ 
is defined as
\begin{equation}
\label{def:U_beta}
U_\beta(t)
\,=\,\, 
t^
{-\beta\big(\frac{n-1}{n-2}\big)}
\!\!\!\!\! 
\int\limits_{ \{ u \,= \, t \}} \!\!\!\!  
|\D u|^{\beta+1} \, \dd \sigma.
\end{equation}
Note that $U_\beta$ coincides with the  function $U$ defined in \eqref{U_Will}
when $\beta=n-2$. 
Moreover, such a Monotonicity-Rigidity Theorem holds for a wider class of manifolds than the ones with Euclidean volume growth. Namely, we prove it for any complete noncompact Riemannian manifold with $\ric \geq 0$ admitting a solution to \eqref{pb}. This class of manifolds coincides with the thoroughly studied class of nonparabolic ones, as we are going to see in Section 2.

\begin{theorem}[Monotonicity-Rigidity Theorem for nonparabolic manifolds]
\label{thm:M/R}
Let $(M, g)$ be a nonparabolic Riemannian manifold
with $\ric\geq 0$.
Given a bounded and open subset $\Omega\subset M$ with smooth boundary,
let $u$ be the solution to problem~\eqref{pb} and let 
$U_\beta$ be the function defined in~\eqref{def:U_beta}. 
Then, for every 
$\beta\geq (n-2)/(n-1)$, 
the function $U_\beta$ is differentiable, with derivative
\begin{equation}
\label{derivata_di_U}
\frac{\dd U_\beta}{\dd t}(t)
\,\,=\,\,
\beta
\,\,t^{-\beta\,\big(\frac{n-1}{n-2}\big)}
\!\!\!\!\!\! \int\limits_{ \{ u = t\}} \!\!\!\!  |\D u|^\beta
\left[\, \HH-
\big(\tfrac{n-1}{n-2}\big)
\, {|\D\log u|}
\, \right] 
\dd \sigma,
\end{equation}
where $\HH$ is the mean curvature of the level set 
$\{u=t\}$ computed with respect to the unit normal vector field $\nu 
=
-{\D u}/{|\D u|}$.
The derivative of $U_\beta$ fulfills
\begin{equation}
\label{eq:monot}
\begin{split}
\frac{\dd U_\beta}{\dd t}(t)
\,\,= \,\,
\frac{\beta}{t^2} 
\int_{ \{ u < t \}} 
 u^{2-\beta\,\big(\frac{n-1}{n-2}\big) } 
\, |\D u|^{\beta-2} \, 
&\bigg\{\,\ric(\D u,\D u) \, + \\
&\,\,+\left[\,\,
\big|\D\D u\big|^2-{ \big(\textstyle\frac n{n-1}}\big)
\big|\D|\D u|\big|^2 \,\, \right]  \, + \\ 
&
\,\,+
\big({\textstyle\beta-\frac{n-2}{n-1}}\big)
\,\, \big|\D^T|\D u|\big|^2  \, + \,\,
\\
& 
\,\,+
\big({\textstyle\beta-\frac{n-2}{n-1}}\big) \,\,
|\D u|^2
\left[ \, 
\HH- 
\big({\textstyle\frac{n-1}{n-2}}\big) |\D \log u|
 \, \right]^2 
 \,\bigg\}
\,\, \dd\mu \, , 
\end{split}
\end{equation}
where $\HH$ is the mean curvature of the level sets of $u$ computed with respect to the unit normal 
vector field $\nu$.
In particular, $U_\beta$ is nondecreasing. 
Moreover, $(\dd U_\beta/\dd t)(t_0)=0$
for some $t_0\leq 1$ and some $\beta\geq (n-2)/(n-1)$
if and only if $(M, g)$ has Euclidean volume growth and
$ \big( \{u\leq t_0\}, g\big)$ is isometric to 
\begin{equation}
\label{cono}
\left(\,\big[r_0, +\infty) \times\{u=t_0\}
\,,\, 
\dd r\otimes\!\!\dd r +\left(\frac r{r_0}\right)^{\!\!2} g_{\{u=t_0\}}
\right),
\qquad
\mbox{with}
\quad
r_0
\,=\,
\bigg(\frac{|\{u=t_0\}|}
{\AVR|\Sf^{n-1}|}\bigg)^{\!\frac1{n-1}}.
\end{equation} 
In this case, in particular, $\{u = t_0\}$ is a connected totally umbilic submanifold with constant mean curvature.
\end{theorem}

Observe that the quantity on the right hand side of \eqref{eq:monot} is nonnegative because of the refined Kato's inequality for harmonic functions reading as
\begin{equation}
\label{kato-intro}
\big|\D\D u\big|^2 \geq \left(\dfrac n{n-1}\right)
\big|\D|\D u|\big|^2.
\end{equation}
The vanishing of \eqref{derivata_di_U} can in particular be interpreted as an overdetermining condition on \eqref{pb}, forcing the ambient manifold to split as a cone and the solution $u$ to be radially symmetric. We mention, in this context of PDE's and splitting results, the paper \cite{farina-mari-valdinoci}.
Although the above Theorem holds on any nonparabolic manifold with $\ric \geq 0$, the rigidity statement implies that it is  sharp just on manifolds with Euclidean volume growth. Studying monotonicity-rigidity properties of this type, but tailored on manifolds with different volume growths at infinity could be the object of future works.   
\begin{remark}
\label{optimal}
{As already observed, the Willmore-type inequalities follow from Theorem \ref{thm:M/R} applied with $\beta = n-2$. In particular, for this purpose, it is sufficient to prove it just for $\beta \geq 1$. In fact, in this paper we propose a proof suited to this range of parameters. To cover the optimal range of parameters, it is sufficient to adapt the arguments recently introduced in the updated version of \cite{Ago_Maz_3}.}
\end{remark}

As for the Euclidean case, Theorem \ref{thm:M/R} will actually be proved working in the manifold $(M\setminus \Omega, \tilde{g})$ where $\tilde{g}$ is conformally related to $g$ by
\[
\tilde{g}=u^{\frac{1}{n-2}} g,
\]
where $u$ is a solution to~\eqref{pb}. In this setting, integral identities and splitting techniques are employed to infer the monotonicity of (the conformal analogue of) $U_\beta$ and the related rigidity. We  point out that these techniques can easily produce a more general version of Theorem \ref{thm:M/R} for  nonparabolic ends with $\ric \geq 0$ of a noncompact Riemannian manifold. 
This \emph{conformal splitting method}, introduced in \cite{Ago_Maz_1}, has proved to be fruitful in various other situations. In \cite{Ago_Maz_2,bor-maz1} and \cite{bor-maz2} it has been applied to the relativistic setting, while in \cite{fogagnolo-mazzieri-pinamonti}, geometric and analytic properties of $p$-harmonic functions in exterior domains have been studied, obtaining as a by-product a new proof of the classical Minkowski inequality. Moreover, a different approach to exterior problems, but still relying on a conformal change of metric, has been introduced in~\cite{borghini-mascellani-mazzieri}.

\medskip

So far, we have considered {\em nonparabolic} Riemannian manifolds with $\ric \geq 0$. We now turn our attention to \emph{parabolic} Riemannian manifolds with nonnegative Ricci curvature. As we will see in Section~\ref{sec:ingredients}, for this class of manifolds, problem~\eqref{pb} does not admit a solution, {\em en revanche} the following problem does
\begin{equation}
\label{prob-ex-par}
\begin{cases}
\,\,\,\Delta{\psi}=0 & \mbox{in} \,\, M \setminus \overline{\Omega} \\
\,\,\,\,\,\,\,\,\psi=0 & \mbox{on}\,\, \partial\Omega \\
\psi(q) \to +\infty &\mbox{as} \,\, d(O, q) \to +\infty \, ,
\end{cases}
\end{equation}
where $\Omega \subset M$ is any bounded and open subset with smooth boundary. Inspired by the fact that problem~\eqref{prob-ex-par} presents strong formal analogies with the conformal reformulation of problem~\eqref{pb} in terms of $\tilde{g}$ (see problem~\eqref{ricciconf} below), we also provide a Monotonicity-Rigidity Theorem for parabolic manifolds with $\ric \geq 0$ involving $\psi$ in place of $u$. For $\beta \geq 0$, we define the function $\Psi_\beta: [0, \infty) \to \R$ as
\begin{equation}
\label{Psi}
\Psi_\beta (s) \,\, = \!\!\!\!\int\limits_{\{\psi = s\}} 	\!\!\abs{\D \psi}^{\beta + 1} \dd \sigma \, , 
\end{equation}
and we prove the following result:
\begin{theorem}[Monotonicity-Rigidity Theorem for parabolic manifolds]
\label{mono-par}
Let $(M, g)$ be a parabolic manifold 
with $\ric\geq 0$.
Let $\Omega \subset M$ be a bounded and open subset with smooth boundary, and let $\psi$ be a solution to problem \eqref{prob-ex-par}.
Then, for every $\beta\geq (n-2)/(n-1)$, 
the function $\Psi_\beta$ is differentiable 
with derivative
\begin{equation}
\label{derivata_di_Psi}
\frac{\dd\Psi_\beta}{\dd s}(s) 
\,= \, 
- \, \beta\!\!\!\! \int\limits_{\{\psi = s \}}\!\!\!\!  
|\D\psi|^\beta\,\HH \,\dd\sigma \, ,
\end{equation}
where $\HH$ is the mean curvature of the level set $\{\psi=s\}$
computed with respect to the unit normal vector field 
$\nu=\D \psi/|\D \psi|$.
Moreover, for every $s\geq0$,
the derivative fulfills 
\begin{equation}
\label{eq:der_fip-par}
\frac{\dd\Psi_\beta}{\dd s}(s)  \,\, 
=\,\,  -\,\beta\!\!\!\!
\int\limits_{\{\psi \geq  s\}} \!\!
{
|\D\psi|^{\beta-2} 
\Big(\ric(\D\psi,\D\psi)
+\big|\D\D \psi\big|^2
 +  \, (\beta-2) \, \big| \D |\D \psi |\big|^2\,\Big) 
}
\dd\mu \, .
\end{equation}
In particular, $\dd\Psi_\beta/\dd s$ is always nonpositive. 
Moreover, $(\dd\Psi_\beta/\dd s)(s_0)=0$
for some $s_0\geq1$ and some $\beta\geq (n-2)/(n-1)$
if and only if $ \big( \{\psi\geq s_0\} , g \big)$ is isometric  
to the Riemannian product 
$\big([s_0, +\infty) \times \{\psi = s_0\},
d\rho\otimes d\rho + {g}_{\{\psi = s_0\}}\big)$. In this case, in particular, $\partial \Omega$ is a connected totally geodesic submanifold.
\end{theorem}
Observe that the right-hand side of \eqref{eq:der_fip-par} is nonnegative again by \eqref{kato-intro}, that in fact holds for any harmonic function on any Riemannian manifold. .

Combining Theorem \ref{thm:M/R} and \ref{mono-par}, we obtain as a straightforward consequence an enhanced version of a theorem of Kasue, \cite[Theorem C (2)]{kasue_minimal}, asserting that if a smooth boundary $\partial \Omega \subset M$  has mean curvature $\HH \leq 0$, then $\HH\equiv 0$ on  $\partial \Omega$ and $M \setminus \Omega$ is isometric to a half cylinder. Our result actually gives precise lower bounds for the supremum of $\HH$ in terms of our monotone quantities and their derivatives.

\begin{theorem}[Enhanced Kasue's Theorem]
\label{enkath}
Let $(M, g)$ be a complete noncompact Riemannian manifold with $\ric \geq 0$, and let $\Omega \subset M$ be a bounded and open subset with smooth boundary. Then, for every $\beta \geq (n-2)/(n-1)$, the following assertions hold true.
\begin{enumerate}
\item[(i)] If $(M, g)$ is nonparabolic, then
\begin{equation}
\label{enka1}
\sup_{\partial \Omega} \HH \,\,\geq \,\,\frac{1}{\int_{\partial \Omega} \abs{\D u}^\beta \dd\sigma} \, \left[ {U_\beta}(0) + \frac{1}{\beta} \frac{\dd U_\beta}{\dd t}(0)\right]\,  > \, 0 \, ,
\end{equation}
where $U_\beta$ is defined in \eqref{def:U_beta} and its derivative satisfies \eqref{eq:monot}. 
\item[(ii)] If $(M, g)$ is parabolic, then
\begin{equation}
\label{enka2}
\sup_{\partial \Omega} \HH \,\, \geq \,\, - \, \frac{1}{\int_{\partial \Omega} \abs{\D \psi}^\beta \dd\sigma} \, \left[ \frac{1}{\!\beta} \frac{\dd \Psi_\beta}{\dd s} (0)\right] \, \geq \, 0\, ,
\end{equation}
where $\Psi_\beta$ is defined in \eqref{Psi} and its derivative satisfies \eqref{eq:der_fip-par}.
\end{enumerate}
\end{theorem}
Kasue's Theorem then follows as a corollary.

\begin{corollary}[Kasue's Theorem]
\label{kasue}
Let $(M, g)$ be a complete noncompact Riemannian manifold with $\ric \geq 0$ and let $\Omega \subset M$ be a bounded and open subset with smooth boundary such that $\HH_{\partial \Omega} \leq 0$ on $\partial \Omega$. Then $(M \setminus \Omega, g)$ is isometric to a Riemannian product $\left([0, +\infty) \times \partial \Omega, \dd r\otimes \dd r + {g}_{\partial\Omega}\right)$ and $\partial \Omega$ is a totally geodesic connected submanifold of $(M, g)$.
\end{corollary} 
Theorem \ref{willth} and Corollary \ref{kasue} can also be gathered in a single general statement, see Corollary \ref{general}.
\begin{remark}
It is worth pointing out that if $(M, g)$ is a Riemannian cylinder, then $M \setminus \Omega$ can have two connected components. In this situation,  problem \eqref{prob-ex-par} has to be set in one of these two ends.
The results above involving parabolic manifolds, and in particular Theorem \ref{enkath}, can still be proved with trivial modifications in this case. Similarly, one can also deal with unbounded $\Omega$, provided that $\partial \Omega$ is a compact hypersurface and $M\setminus \Omega$ is unbounded too.
\end{remark}

\bigskip

Finally, we combine our sharp Willmore-type inequality \eqref{will} with curvature flow techniques along the lines of an argument presented  by Huisken in \cite{Hui_video}. We obtain a characterization of the infimum of the Willmore functional in terms of the isoperimetric ratio of $3$-manifolds with nonnegative Ricci curvature, refining the analogous result stated in the aforementioned contribution. This is the content of the following theorem.
\begin{theorem}[$\AVR$ \& Isoperimetric Constant]
\label{cisoth}
Let $(M, g)$ be a $3$-manifold  with $\ric \geq 0$. Then,
\begin{equation}
\label{ciso}
\inf\,\ddfrac{\abs{\partial\Omega}^{3}}{36{\pi}\abs{\Omega}^2}\,=\,\inf\,\ddfrac{\int_{\partial \Omega}\!\!{\HH}^2\dd\sigma}{16{\pi}}\, = \,\AVR,
\end{equation}
where the infima are taken over bounded and open subsets $\Omega \subset M$ with smooth boundary.
In particular, the following isoperimetric inequality holds
for any bounded and open $\Omega \subset M$ with smooth boundary
\begin{equation}
\label{iso}
\frac{\abs{\partial \Omega}^{3}}{\abs{\Omega}^2} \, \geq \,36{\pi} \,\AVR.
\end{equation}
Moreover, equality is attained in \eqref{iso} if and only if $M= \R^3$ and $\Omega$ is a ball.
\end{theorem}
Beside the characterization of the isoperimetric constant in terms of the Asymptotic Volume Ratio, the novelties with respect to \cite{Hui_video} lie in the rigidity statement and in the fact that the infimum of the Willmore functional is taken over the whole class of bounded open subsets $\Omega$ with smooth boundary, and not just over \emph{outward minimizing} subsets.
 All of these improvements substantially come from our optimal Willmore-type inequality \eqref{will}.

It is worth noticing that the above theorem can be rephrased in terms of a Sobolev inequality with optimal constant. This is the content of the following corollary.
 \begin{corollary}[$\AVR$ \& {Sobolev Constant}]
\label{sobolevth}
Let $(M, g)$ be a $3$-manifold with $\ric \geq 0$. Then
\begin{equation}
\label{sobolev}
\inf_{f \in W_0^{1, 1}\!(M)}\ddfrac{
\int_M \!\abs{\D f}\, \dd\sigma
}{ \left(\int_M \abs{f}^{3/2} \dd\sigma\right)^{\!\!2/3}} \, = \,  \sqrt[3]{36\pi \, \AVR} \, .
\end{equation}
\end{corollary} 
Once Theorem \ref{cisoth} is established,  \eqref{sobolev} is obtained by very standard tools. 
We refer the reader to \cite[pages 89-90]{schoen-yau_book} for a complete proof of the well known equivalence between the isoperimetric and the Sobolev inequality.
On this regard, it is worth observing that relations  between isoperimetry and mean curvature functionals date back to Almgren~\cite{almgren}, while a first derivation of isoperimetric inequalities through a curvature flow has been obtained by Topping in the case of curves~\cite{topping}. Isoperimetric inequalities in $\R^n$ and in Cartan-Hadamard manifolds through curvature flows have been established by Schulze in \cite{schulze1} and \cite{schulze2}, while
the application to manifolds with nonnegative Ricci curvature is suggested in the already mentioned \cite{Hui_video}. The techniques lectured in \cite{Hui_video} have interesting applications also in connection with the relativistic ADM mass, see \cite{jauregui} for the details. Actually, as pointed out in the discussion following \cite[Theorem 5.13]{Hamilton_book},
a positive isoperimetric/Sobolev constant for complete noncompact Riemannian manifolds with $\ric \geq 0$ and Euclidean volume growth can be deduced via the techniques introduced by Croke in~\cite{croke1} and~\cite{croke2}. However, it is known that such constant is not optimal. Other strictly related issues about isoperimetry in noncompact manifolds with $\ric \geq 0$ are treated in \cite{mondino-spadaro} and \cite{Cho_Eic_Vol}.

\smallskip

\smallskip

This paper is organized as follows. In Section 2 we  review, for ease of the reader, the theory of harmonic functions on Riemannian manifolds with nonnegative Ricci curvature we are going to employ along this work. In Section 3 we introduce the conformal formulation of problem \eqref{pb}. In this setting, we prove (the conformal version of) Theorem \ref{thm:M/R}. In Section 4 we work out the integral asymptotic estimates for the  electrostatic potential on manifolds with nonnegative Ricci curvature. With these estimates at hand, we prove Theorem \ref{willth} and Corollary \ref{cor:ALE}. In Section 5 we prove the Monotonicity-Rigidity Theorem for parabolic manifolds and  deduce Theorem \ref{enkath} and Corollary \ref{kasue}. In Section 6, we prove Theorem \ref{cisoth}. Finally, we have included an Appendix where we describe the relations between our monotonicity formulas and some of those obtained by Colding and Colding-Minicozzi in \cite{Colding_1} and \cite{Colding_Minicozzi}.



\section{Harmonic functions in exterior domains}
\label{sec:ingredients}


In this section we are mainly concerned with characterizing Riemannian manifolds for which problems \eqref{pb} and \eqref{prob-ex-par} admit a solution. We are going to see that complete noncompact nonnegatively Ricci curved manifolds for which a solution to~\eqref{pb} exists are the \emph{nonparabolic} ones, namely, manifolds admitting a positive Green's function, while those admitting a solution to~\eqref{prob-ex-par} are the \emph{parabolic} ones. 

Nothing substantially new appears in this section. We are just collecting, re-arranging and applying classical results contained in~\cite{li_tam_green-ex,li_tam_ext,li_tam_green-volume,liyau,yau},  and~\cite{varopoulos}. 
The interested reader might also refer to the nice survey~\cite{grigoryan}, where the relation with the Brownian motion on manifolds is also explored, or, for a more general account on the vast subject of harmonic functions on manifolds, to the lecture notes~\cite{li-lecture} and the references therein. Other important works in this field will be readily cited along the paper. Before starting, let us mention that the results gathered in this preliminary section are spread in a huge literature, and frequently they do not appear exactly in the form we that need or a with a detailed proof. For this reason, we include the most relevant ones. Important gradient bounds are discussed too.

Along this section, we denote by $\D$ the Levi-Civita connection of the Riemannian manifold considered, and by $\Delta$ the related Laplacian. For any two points $p, q \in M$, we let $d(p, q)$ be their geodesic distance. Moreover, it is understood that we are always dealing with manifolds of dimension $n \geq 3$.


\subsection{Green's functions and parabolicity}

Let us begin with the definition of  Green's functions on Riemannian manifolds. 
\begin{definition}[Green's function] 
\label{green}
A smooth function 
\[
G: (M\times M) \setminus \diag(M) \to \R \, ,
\]
where $\diag(M)=\{(p, p), p\in M\}$,
 is said to be a Green's function for the Riemannian manifold $(M, g)$ if the following requirements are satisfied.
\begin{enumerate}
\item[(i)] $G(p, q)=G(q, p)$ for any $p, q \in M$, $p\neq q$.
\item[(ii)]
$\Delta G(p, \cdot)=0$ on $M\setminus\{p\}$, for any $p\in M$.
\item[(iii)]The following asymptotic expansion holds for $q\to p$:
\begin{equation}
\label{asy1}
G(p, q) \, = \, \big(1+ o(1)\big) \, d^{2-n}(p, q) \, .
\end{equation}
\end{enumerate} 
\end{definition}
It is well known that on a complete noncompact Riemannian manifold there always exists a Green's function. This result has been obtained for the first time by Malgrange in \cite{malgrange_green}, while a constructive proof, best suited for applications, has been given by Li-Tam in \cite{li_tam_green-ex}.
Complete noncompact Riemannian manifolds are then divided into two classes.

\begin{definition}[Parabolicity]
Complete noncompact Riemannian manifolds which support a \emph{positive} Green's function are called \emph{nonparabolic}. Otherwise they are called \emph{parabolic}.
\end{definition}
 
A by-product of Li-Tam's construction of Green's function gives the following very useful characterization of parabolicity, see for example~\cite[Theorem 2.3]{li-lecture} for a proof.
\begin{theorem}[Li-Tam]
\label{litam}
Let $(M, g)$ be a complete noncompact Riemannian manifold. Then, it is nonparabolic if and only if there exists a positive super-harmonic function $f$ defined on the complement of a geodesic ball $B(p, R)$ such that
\begin{equation}
\label{barrier} 
 \liminf_{d(p, q)\to\infty}f(q)<\inf_{\partial B(p, R)} f.
\end{equation}
\end{theorem}
Notice that if $(M, g)$ is a nonparabolic Riemannian manifold then a barrier function $f$ as in Theorem \ref{litam} is just the function $G_{\mid M\setminus B(p, R)}.$
A positive Green's function $G$ is called \emph{minimal} if 
\[
G(p,q)\leq \tilde{G}(p, q)
\]
for any other positive Green's function $\tilde{G}$. The construction of the Green's function in \cite{li_tam_green-ex}  actually provides the minimal one. 

The following theorem is a fundamental characterization of parabolicity for manifolds with $\ric \geq 0$ in terms of the volume growth of geodesic balls, first provided in \cite{varopoulos}. 

\begin{theorem}[Varopoulos]
\label{var-th}
Let $(M, g)$ be a complete 
noncompact Riemannian manifold with $\ric\geq 0$. 
Then $(M, g)$ is nonparabolic if and only if
\begin{equation}
\label{var}
\int\limits_1^{+\infty} \!\! \frac{r}{\abs{B(p, r )}}\, \dd r \, < \, +\infty \, ,
\end{equation}
for any $p\in M$,
where $B(p, r)$ is a geodesic ball centered at $p$ with radius $r \geq 0$.
\end{theorem}
The above characterization 
roughly says that on nonparabolic manifolds volumes are growing faster than quadratically, while on the parabolic ones they grow at most quadratically. On the other hand, on a $n$-dimensional complete noncompact Riemannian manifold with $\ric \geq 0$ Bishop-Gromov's Theorem and a result of Yau \cite{yau-growth} respectively show that the growth of volumes of geodesic balls $B(p, r)$ is controlled from above by $r^n$ and from below by $r$. 

From now on we focus our discussion on complete noncompact manifolds with nonnegative Ricci curvature. In the following two subsections we collect some basic though fundamental facts in this context, for the ease of references.

\subsection{Harmonic functions on manifolds with nonnegative Ricci curvature.} A basic tool in the study of the potential theory on Riemannian manifolds is the following celebrated gradient estimate, first provided by Yau in~\cite{yau} (see also the nice presentation given in~\cite{schoen-yau_book}).
\begin{theorem}[Yau's Gradient Estimate]
Let $(M, g)$ be a complete noncompact Riemannian manifold  with $\ric\geq 0$. Let $u$ be a positive harmonic function defined on a geodesic ball $B(p, 2R)$ of center $p$ and radius $2R$. Then, there exists a constant $C=C(n)>0$ such that
\begin{equation}
\label{cheng}
\sup_{x\in B(p, R)} \!\!\frac{\abs{\D u}}{u}\, \leq \,\, \frac{C}{R} \, .
\end{equation}
\end{theorem}
We now apply the above inequality to a harmonic function $v$ defined in a geodesic annulus $B(p, R_1)\setminus \overline{B(p, R_0)}$. We obtain a decay estimate on the gradient of $u$ that we will employ several times along this paper.
\begin{proposition}
Let $(M, g)$ be a complete noncompact Riemannian manifold with $\ric \geq 0$, and let $u$ be a positive harmonic function defined in a geodesic annulus $B(p, R_1)\setminus \overline{B(p, R_0)}$, with $R_1 > 3R_0$. Then, there exists a geometric constant $C=C(n)$ such that
\begin{equation}
\label{yau}
\abs{\D u} (q)\, \leq \, C \, \frac{u(q)}{d(p, q)},  
\end{equation}
for any point $q$ such that $2R_0 \leq d(p, q)<\dfrac{R_1 + R_0}{2}$. In particular, if $u$ is a harmonic function defined in $M \setminus \overline{B(p, R_0)}$, then 
\begin{equation}
\label{yau1}
\abs{\D u} (q)\leq C\frac{u(q)}{d(p, q)}
\end{equation}
for any point $q$ with $2R_0 \leq d(p, q)$.
\end{proposition}
\begin{proof}
Let $q$ be such that $2R_0 \leq d(p, q)<\dfrac{R_1 + R_0}{2}$. Then the ball $B\big(q,\, d(p, q)- r_0\big)$ is all contained in the annulus $B(p, R_1)\setminus \overline{B(p, R_0)}$.  In particular, by Yau's inequality \eqref{cheng} we have
\[
\abs{\D u}(q)\leq C\frac{u(q)}{d(p, q)-r_0}\leq 2 C\frac{u(q)}{d(p, q)}.
\]
Letting $R_1 \to \infty$, we get also \eqref{yau1}.
\end{proof}
Another important application of Yau's inequality is the following compactness result for sequences of harmonic functions. A complete proof can be found in \cite[Lemma 2.1]{li-lecture}. 
\begin{lemma}
\label{convergence}
Let $(M, g)$ be a complete Riemannian manifold with $\ric \geq 0$, and let $U \subset M$ be an open and connected subset. Let $\{f_i\}$ be a sequence of positive harmonic functions defined on $U$, and suppose there exists a constant $C$ such that $f_i(p) \leq C$ at some point $p \in U$ for any $i\in \N$. Then, there exists a subsequence $\{f_{i_j}\}$ converging to a positive harmonic function $f$ uniformly on any compact set $K \subset U$. 
\end{lemma}

\subsection{Ends of manifolds with nonnegative Ricci curvature}
\label{sub-ends}
It is a well-known and largely exploited fact that a complete noncompact Riemannian manifold $(M, g)$ with nonnegative Ricci curvature that is not a Riemannian cylinder has just one end. However, since a complete proof of this fact is hard to find in standard literature, we discuss the details below.

We employ the following definition of end, that is, the one used in the works by Li and Tam, see for example \cite[Definition 0.4 and discussion thereafter]{li_tam_ext}.
\begin{definition}[Ends of Riemannian manifolds]
\label{ends}
An end  of a Riemannian manifold $(M, g)$ with respect to a compact subset $K \subset M$ is an unbounded connected component of $M \setminus K$. We say that $(M, g)$ has a finite number of ends if the number of ends with respect to any compact subset $K \subset M$ is bounded by a natural number $k$ independent of $K$.
In this case, we say that $(M, g)$ has $k$ ends if it has $k$ ends with respect to a compact subset $K \subset M$ and to any other compact subsets of $M$ containing $K$.
\end{definition}

To state the result, we quickly recall some terminology. A \emph{line} in $(M,g)$ is a curve
$\gamma:\R\to M$ which is a minimal geodesic between 
any two points lying on it.
A $ray$ is half a line. We also recall that a complete Riemannian manifold $(M,g)$ is called a {\em Riemannian cylinder} if it is isometric to the Riemannian product $(\R \times N^{n-1}, \dd t \otimes \dd t + g_{N^{n-1}})$, where $N^{n-1}$ is a compact manifold.

\begin{remark}
\label{more_than_2}
Using the above definition and terminology it is clear that if a Riemannian manifold has at least two ends, 
then it contains a line. 
\end{remark}

\begin{proposition}
\label{oneend}
Let $(M,g)$ be a 
complete noncompact Riemannian manifold
with $\ric \geq 0$.
If $(M, g)$ is not a Riemannian cylinder,
then it has just one \emph{end}.
\end{proposition}


\begin{proof}
Since $(M^n,g)$  has nonnegative Ricci curvature by hypothesis,
then the Cheeger-Gromoll Splitting Theorem \cite{cheegergromoll} implies that
\begin{equation}
\label{eq:Che_Gro}
(M,g)
\,\,
\mbox{is isometric to}
\,\,
\big(\R^m\times N^{n-m},g_{\R^k}+g_{N^{n-m}}\big),
\end{equation}
for some $m\in\{0,\hdots,n\}$,
where the manifold $(N^{n-m},g_{N^{n-m}})$ has 
nonnegative Ricci curvature
and does not contain any line.
The Riemannian manifold $(M, g)$ is a Riemannian cylinder if $m=1$ and $N^{n-1}$ is compact. Let us then suppose that $(M, g)$ is not a Riemannian cylinder. We  consider the cases $m=0$, $m=1$ and $m \geq 2$. 

\emph{Case $m = 0$}.
In this case, $(M,g)$ does not contain 
any line. Then there is no more than one end, 
in view of Remark \ref{more_than_2}.

\emph{Case $m = 1$}.
Since $(M, g)$ is not a cylinder, we have that $N^{n-1}$ is a noncompact Riemannian manifold that contains no lines. Then again by Remark \ref{more_than_2}, it has at most one end. Thus, also $\R \times N^{n-1}$ has at most one end.

\emph{Case $m \geq 2$}.
We show that
\begin{equation}
\label{non_connessione}
M\setminus K\,\,\mbox{is connected for every compact}\,\,K\subset M.
\end{equation}
In view of Definition \ref{ends},
this readily implies that $M$ has at most one end.  
Now, to check \eqref{non_connessione} in view of \eqref{eq:Che_Gro},
it is sufficient to check that for every compact $Q\in\R^m$
and every compact $P\in N^{n-m}$, 
we have that
$(\R^k\times N^{n-m})\setminus(Q\times P)$
is connected. 
Let then $(x,q),\,(y,p)\in(\R^m\times N^{n-m})\setminus(Q\times P)$
and suppose for the moment
that $x,y\in Q$,
so that, in turn, $q,p\notin P$.
Choose $z\in\R^m\setminus Q$
and define the curves
\[
\alpha(t)=\big(tx+(1-t)z,q\big),
\qquad
\beta(t)=\big(ty+(1-t)z,p\big),
\]
$t\in[0,1]$, connecting $(x,q)$ to $(z,q)$
and $(y,p)$ to $(z,p)$,
respectively.
Note that $\alpha(t),\beta(t)\in(\R^m\times N^{n-m})\setminus(Q\times P)$ 
for every $t\in[0,1]$, because $q,p\notin Q$.
Now, let $\gamma(t)$ be a continuous curve 
in $N^{n-m}$ connecting $q$ and $p$.
Then the curve $\big(z,\gamma(t)\big)
\in\big((\R^m\setminus Q)\times N^{n-m}\big)$
connects $(z,q)$ to $(z,p)$.
Gluing together the curves $\alpha$, $\beta$,
and $(z,\gamma)$, we obtain a continuous path 
lying in $(\R^m\times N^{n-m})\setminus(Q\times P)$
and connecting $(x,q)$ to $(y,p)$.
Obtaining such a curve in the case
where either $x\notin Q$ or $y\notin Q$
requires a similar simpler construction.
We have thus proved that  
$(\R^m\times N^{n-m})\setminus(Q\times P)$
is path-connected, hence connected.

\smallskip

We proved that $(M,g)$
has at most one end, if it is not a cylinder. Then, it has exactly one end, because
it is noncompact.  
\end{proof}

Let us now describe separately some aspects of harmonic functions on nonparabolic and parabolic manifolds. In particular, we are going to characterize these two classes of manifolds through a couple of existence results for solutions of suitable boundary value problems in exterior domains (see Theorems~\ref{existence-nonpar} and~\ref{existence-par} below). The monotone quantities analysed in Theorems~\ref{thm:M/R} and~\ref{mono-par} are defined along the level sets of these solutions.

\subsection{The exterior problem on nonparabolic manifolds}
The following is a fundamental estimate proved by Li-Yau in \cite{liyau}.
\begin{theorem}[Li-Yau]
\label{liyau}
Let $(M, g)$ be a nonparabolic Riemannian manifold with $\ric \geq 0$. Then, its minimal Green's function $G$ satisfies
\begin{equation}
\label{liyauf}
C^{-1} \!\!\int\limits_{d(p, q)}^{+\infty} \!\!\frac{r}{\abs{B(p, r)}} \, \dd t \,\, \leq \,\, G(p, q) \,\, \leq \,\, C \!\! \int\limits_{d(p, q)}^{+\infty} \!\!\frac{r}{\abs{B(p, r)}} \, \dd t \, ,
\end{equation}
for some $C=C(n)>0$.
\end{theorem}

Combining \eqref{var} with \eqref{liyauf}, we get that the minimal Green's function goes to $0$ at infinity, i.e.  for any fixed $p$ in M
\begin{equation}
\label{greenzero}
\lim_{d(p, q)\to\infty} \!\!G(p, q) \, = \, 0 \, .
\end{equation}
An easy application of Laplace Comparison Theorem then gives the following well known fact.

\begin{lemma}
\label{laplace_comp}
Let $(M, g)$ be a nonparabolic Riemannian manifold with $\ric \geq 0$, and let $G$ be its minimal Green's function. Then, for any fixed pole $p$ we have
\begin{equation}
\label{laplacef}
d^{2-n}(p, q) \leq G (p, q).
\end{equation}
for any $q \neq p$ in $M$.
\end{lemma}
\begin{proof}
Let $r$ be the function mapping a point $q$ in $M$ to $d(p, q)$. By the Laplacian Comparison Theorem, we have 
\[
\Delta r \,\,  \leq \, \, \frac{n-1}{r}
\]
in the sense of distributions (see e.g. \cite[Theorem 1.128]{Hamilton_book}). Therefore, we have, in the sense of distributions,
\begin{equation}
\label{r^{2-n}}
\Delta r^{2-n} \,\,= \,\,  (n-2)\left[(n-1) r^{-n} - r^{1-n} \Delta r\right] \,\, \geq \,\, 0 \, ,
\end{equation}
and then the the function $r^{2-n} - G(p, \cdot)$ is sub-harmonic. By the maximum principle, for any $ \epsilon > 0$ and $R > \epsilon$
\[
\max_{\overline{B(p, R)} \setminus B(p, \epsilon)} \!\!(r^{2-n} - G(p, \cdot)) \,\, = \,\, \max_{\partial B(p, R) \cup \partial B(p, \epsilon)} \! (r^{2-n} - G(p, \cdot)) \, .
\]
We conclude by passing to the limit as $\epsilon \to 0$ and $R\to \infty$, taking into account the asymptotic behavior at the pole $p$ given by \eqref{asy1} and that $G \to 0$ at infinity, as observed in \eqref{greenzero}.
\end{proof}


Now, we characterize the existence of a solution to problem \eqref{pb} with the nonparabolicity of the ambient manifold. Let us first set up some notation that we are going to use in the rest of the paper. 
With respect to a bounded open subset $\Omega \subset M$ with smooth boundary, we denote by $O$ a generic reference point taken inside $\Omega$.

\begin{theorem}
\label{existence-nonpar}
Let $(M, g)$ be a complete noncompact Riemannian manifold with $\ric\geq 0$, and let $\Omega \subset M$ be a bounded open subset with smooth boundary.
Then, there exists a solution to problem \eqref{pb} if and only if $(M, g)$ is nonparabolic.
\end{theorem}
\begin{proof}
The proof is an easy adaptation of arguments already presented in \cite{li_tam_ext}.
By Theorem \ref{litam} the existence of a solution to problem \eqref{pb} implies nonparabolicity of $M$, since the restriction of $u$ to $M\setminus B(O, R)$ with $\Omega \subset B(O, R)$ clearly satisfies condition \eqref{barrier}.

Conversely, assume that $M$ is nonparabolic, and
consider an increasing sequence of radii $\{R_i\}_{i\in\N}$ such that $\Omega\subset B(O, R_1)$ and $R_i \to \infty$. Let, for any $i\in \N$, $u_{i}$ be the solution to the following problem:
\begin{equation}
\label{pb1}
\begin{cases}
\!\Delta{u}=0 & \mbox{in} \,\, B(O, R_i)\setminus \overline{\Omega} \\
\,\,\,\,u=1 & \mbox{on}\,\, \partial\Omega \\
\,\,\,\,u = 0 &\mbox{on} \,\, \partial B(O, R_i).
\end{cases}
\end{equation}
Let  now $G$ be the minimal positive Green's function, 
and consider the function $G(O, \cdot)$. Due to the Maximum Principle for harmonic functions and the  boundary conditions in problem \eqref{pb1} we have that
\begin{equation}
\label{control-ui}
0\leq u_{Ri}(q)\leq \frac{G(O, q)}{\min_{\partial\Omega}G(O, \cdot)} \, ,
\end{equation}
for $ q \in B(O, R_i)$.
Let then $K$ be a compact set contained in $M \setminus \overline{\Omega}$. We can clearly suppose without loss of generality that $K$ is contained in  $B(O, R_i)\setminus \overline{\Omega}$ for any $i$. Then, \eqref{control-ui} and Lemma \ref{convergence} give that $u_{i}$ converges up to a subsequence to a harmonic function $u$ on $K$. We can clearly extend by continuity $u$ to $1$ on $\partial \Omega$. Again by \eqref{control-ui}, and by uniform convergence,
\[
0\leq u(q)\leq \frac{G(O, q)}{\min_{\partial\Omega}G(O, \cdot)} \, .
\]
on $M \setminus \Omega$. Since, by \eqref{greenzero}, $G(O, q)\to 0$ as $d(O, q)\to\infty$, so does $u$, completing the proof.
\end{proof}

We conclude this section with the following easy lemma, which shows that we can control function $u$ by the minimal Green's function $G$.
\begin{lemma}
\label{control}
Let $(M, g)$ be a nonparabolic Riemannian manifold with $\ric \geq 0$, and let $G$ be its minimal Green's function.
Let $u$ be a solution to \eqref{pb} for some open and bounded set $\Omega$ with smooth boundary, and let $O \in \Omega$.
Then, there exist  constants $C_1=C_1(M, \Omega)>0$ and $C_2=C_2(M, \Omega)>0$ such that
\begin{equation}
\label{bound_u_g}
C_1\,G(O, q)\leq u(q)\leq C_2\, G(O, q)
\end{equation}
on $M \setminus \Omega$. In particular 
\begin{equation}
\label{lower}
C_1 d(O, q)^{2-n} \leq u(q).
\end{equation}
on $M \setminus \Omega$.
\end{lemma}
\begin{proof}
Just set $0<C_1<1/\max_{\partial\Omega}G(O, \cdot)$, and $C_2>1/\min_{\partial\Omega}G(O, \cdot)$. The claim follows from the Maximum principle and the observation that both $u$ and $G$ are vanishing at infinity. The inequality~\eqref{lower} is obtained combining the lower estimate on $u$ by \eqref{bound_u_g} with \eqref{laplacef}. 
\end{proof}

\subsection{The exterior problem on parabolic manifolds}
\label{exterior-parabolic}
The following inequalities, proved in \cite[Theorem 2.6]{li_tam_green-volume}, can be interpreted as a version for parabolic manifolds of the Li-Yau inequalities recalled in Theorem \ref{liyau}. We point out that we are always dealing with Green's functions obtained by the Li-Tam's construction.
\begin{theorem}
\label{litamyau}
Let $(M, g)$ be a parabolic manifold with $\ric \geq 0$, and let $p \in M$. Let $G$ be a Green's function. Then, for any fixed $r_0 > 0$ and for any $q$ with $d(p, q) > 2r_0$ there holds
\begin{equation}
\label{litamyauf-1}
- \, G(p, q) \,\, \leq \,\, C_1 \!\!
\int\limits_{r_0}^{d(p, q)}\!\!\!\frac{r}{\abs{B(p, r)}}\, \dd r\,  + \, C_2 \, ,
\end{equation}
for some constants $C_1$ and $C_2$  depending only on $n$, $r_0$ and the choice of $G$. Moreover, for any $R > r_0$ , there holds 
\begin{equation}
\label{litamyauf-2}
C_3 \int\limits_{r_0}^R \frac{r}{\abs{B(p, r)}} \, \dd r \, + \,  C_4 \,\, \leq  \sup_{\partial B(p, R)} \big(- G(p, \cdot) \big) \, ,
\end{equation}
for some constants $C_3$ and $C_4$  depending only on $n$, $r_0$ and the choice of $G$.
\end{theorem} 

When $(M, g)$ is parabolic, Li-Tam proved in \cite[Lemma 1.2]{li_tam_ext} that the exterior problem \eqref{prob-ex-par} admits a solution.
The construction of such a solution $\psi$, combined with Yau's inequality and Theorem \ref{litamyau}, readily implies a uniform gradient bound on $\psi$. 
\begin{theorem}
\label{existence-par}
Let $(M, g)$ be a parabolic Riemannian manifold with $\ric \geq 0$, and let $\Omega \subset M$ be a bounded and open subset with smooth boundary. Then, there exists a solution to problem \eqref{prob-ex-par}. Moreover, $\abs{\D \psi}$ is uniformly bounded in $M \setminus \Omega$.
\end{theorem}

\begin{remark}
Recall from Subsection~\ref{sub-ends} that if $(M, g)$ is a {\em Riemannian cylinder}, then $M \setminus \Omega$ might have two connected components. If this is the case, it will be understood that we consider problem \eqref{prob-ex-par} on a connected component of $M \setminus\Omega$. All the proofs work unchanged in this case.
\end{remark}

\begin{proof}[Proof of Theorem \ref{existence-par}]
Let $O \in \Omega$, let $U \subset \Omega$ be an open neighborhood of $O$ and let $K$ be the compact set defined by $K = \overline{\Omega} \setminus U$.  Consider, for a sequence $B(O, R_i)$ of geodesic balls with increasing radii containing $\Omega$, a corresponding sequence of positive Green's functions $G_i(O, \cdot)$ of $B(O, R_i)$ with pole in $O$ such that $G_i(O, x) = 0$ for $x \in \partial B(O, R_i)$. We then consider the sequence of functions defined in $B(O, R_i)\setminus \{O\}$ by
\[
f_i(q)\, = \, \sup_{x \in K} G_i(O, x) - G_i(O, q).
\] 
The construction in~\cite{li_tam_green-ex} implies that there exists a Green's function $G$ on $M$ such that $f_i$ converges to $-G(O, \cdot)$ uniformly on compact subsets of $M \setminus \{O\}$ (compare with the discussions around Lemma 1.2 in~\cite{li_tam_ext}).
Observing that $f_i = \sup_K G_i(O, \cdot)$ on $\partial B(O, R_i)$, we set \[
a_i \, = \, \sup_{x \in K} G_i(O, x)
\]
and consider the solution $\psi_i$ to the problem
\begin{equation}
\label{prob-ex-par-approx}
\begin{cases}
\,\,\,\Delta{\psi}=0 & \mbox{in} \,\, B(O, R_i) \setminus \overline{\Omega} \\
\,\,\,\,\,\,\,\,\psi=0 & \mbox{on}\,\, \partial\Omega \\
\,\,\,\,\,\,\,\,\psi = a_i &\mbox{on} \,\, \partial B(O, R_i).
\end{cases}
\end{equation}
Since $\sup_K G_i(O, \cdot) \geq \sup_{\partial\Omega} G_i(O, \cdot)$, the Maximum Principle immediately gives
\begin{equation}
\label{chain-parabolic}
f_i - \sup_{\partial \Omega} f_i \, \leq \, \psi_i \, \leq \, f_i
\end{equation}
on $B(O, R_i) \setminus \Omega$. Since the sequence $f_i$ is converging (uniformly on compact sets) to the Li-Tam Green's function, the second inequality in~\eqref{chain-parabolic} combined with Lemma~\ref{convergence} shows that $\psi_i$ converges uniformly on the compact subsets of $M \setminus \overline{\Omega}$ to an harmonic function $\psi$, that we can clearly extend to $0$ on $\partial \Omega$. Moreover, since for every $q \in M\setminus \{O\}$ the sequence $f_i(q)$ converges to $-G(O, q)$ and since by~\eqref{litamyauf-2} we have that $-G(O, q_j) \to +\infty$ along a sequence of points $q_j$ such that  $d(O, q_j)\to +\infty$, we use the first inequality in~\eqref{chain-parabolic}, to deduce that $\psi(q_j) \to +\infty$, as $j \to + \infty$. In particular, since by \cite[Lemma 3.40]{colding3} $\psi$ must admit a  limit at infinity, we infer that  $\psi(q) \to +\infty$, as $d(O, q) \to +\infty$. Therefore, $\psi$ is a solution to problem \eqref{prob-ex-par}.

Observe that, again by \eqref{chain-parabolic}, $\psi \leq -G(p, \cdot)$. Inequality \eqref{yau1} then yields
\begin{equation}
\label{yau-parabolic}
\abs{\D \psi}(q) \, \leq \, C \frac{\psi(q)}{d(O, q)} \, \leq \, C  \frac{-G(O, q)}{d(O, q)}
\end{equation}
for some constant $C$ and any $q$ outside some big geodesic ball $B(O, r_0)$. Combining now \eqref{litamyauf-1} with Yau's lower bound on the growth of geodesic balls, saying that $\abs{B(O, r)} \geq Cr$ for any $r \geq 1$ and for some constant $C$, we also have
\[
\frac{-G(O, q)}{d(O, q)} \leq C_1 \frac{d(O, q) + C_2}{d(O, q)} 
\] 
for $q$ with $d(O, q) > 2r_0$ and constants $C_1$ and $C_2$.
Plugging it in \eqref{yau-parabolic}, this shows that $\abs{\D \psi}$ is uniformly bounded, as claimed.
\end{proof}

\section{Proof of the Monotonicity-Rigidity Theorem for nonparabolic manifolds}



\subsection{The conformal setting}


Let $(M, g)$ be a nonparabolic Riemannian manifold with nonnegative Ricci curvature. Let $\Omega \subset M$ be a bounded and open set with smooth boundary, and let $u$ be the solution to problem \eqref{pb}. We introduce, in $M \setminus \Omega$ the metric
\begin{equation}
\label{gconf}
\cg
\,=\,
u^\frac{2}{n-2}g.
\end{equation} 
The expression for $\tilde{g}$ is formally the same as in \cite{Ago_Maz_1} and \cite{Ago_Maz_3}. Let us explain why such a conformal change of metric is natural also in the current setting. Our model geometry is that of a truncated metric cone   
\begin{equation}
\label{model-cone}
(M \setminus \Omega, g) \cong \Big(\,\big[r_0, +\infty) \times\pa\Omega
\,,\, 
\dd r\otimes\!\dd r + Cr^2 g_{\pa\Omega}
\Big) \, ,
\end{equation}
for some positive constant $r_0$ and $C$, and where $g_{\partial \Omega}$ is the metric induced by $g$ on $\partial \Omega$. We also assume that $\partial \Omega$ is a smooth closed sub-manifold with $\ric_{\partial \Omega} \geq (n-2)g_{\pa\Omega}$. Such a curvature assumption on $\partial \Omega$  is equivalent to suppose that the cone in~\eqref{model-cone} has nonnegative Ricci curvature. In this model setting, up to a suitable choice of  $C$ in \eqref{model-cone}, the solution to problem \eqref{pb} is $u(r)= r^{2-n}$. With this specific $u$, the metric $\tilde g$  becomes
\[
\tilde{g} = \dd \rho \otimes\! \dd\rho + g_{\pa\Omega} \, ,
\]
where $\rho = \log r$. In other words $\tilde g$ is a (half) {Riemannian cylinder} over $(\pa \Omega, g_{\pa \Omega})$.
In parallel, as the rigidity statement in Theorem~\ref{thm:M/R} gives a characterization of the truncated cone metrics~\eqref{model-cone}, so its conformal version in Theorem~\ref{thm:main_conf} characterizes cylindrical metrics.

Having this in mind, we are now going to describe the general features of $(M\setminus \Omega, \cg)$ in more details.
Letting
\begin{equation}
\label{def:ffi}
\ffi \, = \, -\log u \, ,
\end{equation}
we have that $\tilde{g}=e^{-\frac{2\ffi}{n-2}}g$.

As before, $\D$ is the Levi-Civita connection of $(M, g)$. Moreover, we denote by $\D\D$ the Hessian.  We denote by $\nabla$,  the Levi-Civita connection of the metric $\tilde{g}$, by $\nabla\nabla$ its Hessian, and we put the subscript $\tilde{g}$ on any other quantity induced by $\tilde g$.  We have, for a smooth function $w$ 
\begin{equation}
\label{hessianconf}
\nabla_{\alpha}\!\nabla_{\beta} w \,= \, \D_\alpha\D_{\beta} w + \frac{1}{n-2}\Big(\partial_{\alpha}w\partial_{\beta}\phi + \partial_\beta w \partial_\alpha \phi - \langle \D w, \D\phi\rangle g_{\alpha\beta}\Big),
\end{equation}
where by $\langle\cdot, \cdot\rangle$ we denote the scalar product induced by $g$. 
In particular,
\begin{equation}
\label{phiharm}
\Delta_{\tilde{g}}\phi \, = \, 0 \, .
\end{equation}
Moreover, the Ricci tensor $\ric_{\cg}$ of $\cg$ and the Ricci tensor $\ric$ of $g$ satisfy
\begin{equation}
\label{ricciconf}
\ric_{\tilde{g}} \, = \, \ric + \nabla_\alpha\!\nabla_{\beta}\phi - \dfrac{d\phi\otimes d\phi}{n-2} + \dfrac{\abs{\nabla \phi}^2_{\tilde g}}{n-2} \, \tilde{g}\, .
\end{equation}
Finally, by \eqref{phiharm} and \eqref{ricciconf} problem \eqref{pb} becomes
\begin{equation}
\label{pbconf}
\begin{cases}
\,\,\,\,\,\,\,\,\,\,\,\,\,\,\,\,\,\,\,\,\,\,\,\,\,\,\,\,\,\,\,\,\,\,\,\,\,\,\,\,\,\,\,\,\,\,\,\,\,\Delta_{\tilde{g}}\phi=0 & \mbox{in} \,\, M\setminus\overline{\Omega} \\
\ric_{\tilde{g}} - \nabla\nabla\phi + \dfrac{d\phi\otimes d\phi}{n-2}=\dfrac{\abs{\nabla\phi}_{\cg}^2}{n-2}\cg +\ric & \mbox{in} \,\, M\setminus\overline{\Omega} \\
\,\,\,\,\,\,\,\,\,\,\,\,\,\,\,\,\,\,\,\,\,\,\,\,\,\,\,\,\,\,\,\,\,\,\,\,\,\,\,\,\,\,\,\,\,\,\,\,\,\,\,\,\,\,\,\,\,\phi=0 & \mbox{on} \,\, \partial \Omega \\
\,\,\,\,\,\,\,\,\,\,\,\,\,\,\,\,\,\,\,\,\,\,\,\,\,\,\,\,\,\,\,\,\,\,\,\,\,\,\,\,\,\,\,\,\,\,\,\,\,\,\phi(q)\to + \infty & \mbox{as} \,\, d(O, q)\to+ \infty.
\end{cases}
\end{equation}
The classical Bochner identity applied to $\phi$ in $(M \setminus \Omega, \tilde{g})$ , combined with the first two equations of the above system, immediately yields the following identity
\begin{equation}
\label{boch1}
\Delta_{\tilde{g}}\abs{\nabla \phi}_{\tilde{g}}^2 - \langle \nabla \abs{\nabla \phi}_{\tilde{g}}^2, \nabla \phi\rangle_{\tilde{g}} =  2 \left[\ric(\nabla \phi, \nabla \phi) + \abs{\nabla \nabla \phi}_{\tilde{g}}^2\right],
\end{equation}
where $\ric$ is the Ricci tensor of the  background metric $g$. Such a relation is at the heart of this work.
As a first application, we have the following fundamental corrspondence between the splitting of $(M \setminus \Omega, \tilde{g})$ as a cylinder and the splitting of $(M \setminus \Omega, g)$ as a cone.
\begin{lemma}
\label{splitting-principle}
Let $(M, g)$ be a nonparabolic Riemannian manifold with $\ric \geq 0$, let $\Omega \subset M$ be a bounded and open subset with smooth boundary and let $\cg$ and $\phi$ be defined by~\eqref{gconf} and \eqref{def:ffi}. Assume that
$ \na |\na \ffi |_{\tilde g} = 0$ on $\{\phi \geq s_0\}$ for some $s_0 \in [0, +\infty)$. 

\smallskip

\begin{itemize}
\item[(i)]
Then the Riemannian manifold $(\{\phi \geq s_0\}, \cg)$ is isometric to the Riemannian product
$$\big([s_0, +\infty) \times \{\phi = s_0\}, d\rho\otimes d\rho + \tilde{g}_{\{\mid \phi = s_0\}}\big) \, .
$$ 
In particular, $\pa \Omega$ is a connected totally geodesic submanifold inside $(M\setminus \Omega, \tilde{g})$. 

\smallskip

\item[(ii)] Accordingly, for $t_0 = e^{-s_0}$, the Riemannian manifold $(\{u \leq t_0\}, \cg)$ has Euclidean volume growth and it is isometric to the truncated cone
\begin{equation}
\label{cono2}
\qquad \quad\bigg(\,\big[r_0, +\infty) \times\{u=t_0\}
\,,\, 
\dd r\otimes\!\dd r +\left(\frac r{r_0}\right)^{\!\!2} g_{\{u=t_0\}}
\bigg),
\quad
\mbox{with}
\quad
r_0
\,=\,
\bigg(\frac{|\{u=t_0\}|}
{\AVR|\Sf^{n-1}|}\bigg)^{\!\frac1{n-1}}.
\end{equation} 
In particular, $\partial \Omega$ is a connected totally umbilic submanifold inside $(M\setminus \Omega, {g})$ with constant mean curvature.
\end{itemize}

\end{lemma}
\begin{proof}
Let us first observe that plugging $\na |\na \ffi |_{\cg}= 0$ in \eqref{boch1} readily implies, since $\ric \geq 0$, that 
\begin{equation}
\label{nana_fi_0}
\na\na\ffi
\,\equiv\,0
\qquad\mbox{in}\quad\{\ffi>s_0\}=\{u<t_0\}.
\end{equation}
for $s_0 = -\log t_0$. The isometry of $(M \setminus \Omega, \tilde{g})$ with a Riemannian product then follows from \cite[Theorem 4.2 (i)]{Ago_Maz_1}, proving the first claim. Observe that $\partial \Omega$ is connected as a consequence of Proposition \ref{oneend}. Recalling  the formula
\[
\na_\alpha \!\na_{\beta} w
\,=\,        
\D_\alpha\D_{\beta}
w+
\frac1{n-2}\bigg(\pa_\alpha w\,\pa_\beta\ffi
+\pa_\beta w\,\pa_\alpha\ffi
-\langle\D w, \D\ffi\rangle\,g_{\alpha\beta}\bigg), 
\]
that holds for every ${C}^2$ function $w$,
we have in particular that condition \eqref{nana_fi_0}
translates into
\begin{equation}
\label{nana_u}
0
\,=\,
\na_\alpha \! \na_{\beta}\ffi
\,=\,        
-\frac{\D_\alpha\D_{\beta }u}{u}
+
\Big(\frac n{n-2}\Big)
\frac{\D_\alpha u\D_\beta u}{u^2}
-
\Big(\frac 1{n-2}\Big)
\left|\frac{\D u}u\right|^2g_{\alpha\beta}. 
\end{equation}
Observing now that 
\[
\D_\alpha\D_{\beta}\big(u^{-\frac2{n-2}}\big)
\,=\,        
\Big(\frac2{n-2}\Big)u^{-\frac2{n-2}}
\left[
-\frac{\D_\alpha\D_{\beta }u}u
+\Big(\frac n{n-2}\Big)
\frac{\D_{\alpha} u\D_{\beta} u}{u^2}
\right]
\]
we deduce from \eqref{nana_u} that
\[
\D_\alpha\D_{\beta}
\left(u^{-\frac2{n-2}}\right)
\,=\,
\frac2{(n-2)^2}
\left|\frac{\D u}u\right|^2
u^{-\frac2{n-2}}\,g_{\alpha\beta}.
\]
In particular, we have that $\D\D\big(u^{-2/(n-2)}\big)$
is proportional to the metric. 
By a standard result in Riemannian geometry
(see e.g. \cite[Theorem 1.1]{cat-man-maz} for a complete proof, or \cite[Section 1]{Che-Cold})
this fact implies that the potential 
$f= u^{-2/(n-2)}$, and thus $u$, depends only on the
(signed) distance $r$ from $\{u=t\}$
and that, up to a multiplicative factor in front of $f$,
the metric $g$ can be expressed as
\[
g
\,=\,
\dd r\otimes\dd r
\,+\,
\big(f'(r)\big)^{\!2}
g_\Sigma,
\qquad 
\Sigma=\{u = t\}.
\] 
Moreover, the associated Ricci tensor is given by
\begin{equation}
\label{expr_Ric}
\ric
\,=\,
-(n-1)
\frac{f'''}{f'}
\dd r\otimes\dd r
\,+\,
\ric_\Sigma
-
\big((n-2)(f'')^2+f'f'''\big)^{2}
g_\Sigma \, .
\end{equation}
The second information obtained by plugging $\nabla \abs{\nabla \phi}_{\cg}=0$ into the Bochner identity \eqref{boch1} is

\begin{equation}
\label{ricci-zero}
\ric\left(\nabla\phi, \nabla \phi\right)
\,=\,0
\qquad\mbox{in}\quad\{\ffi>s_0\}=\{u<t_0\} \, .
\end{equation}
In particular this implies that
$\ric(\D u,\D u)=0$ in the same set.
Using now expression \eqref{expr_Ric}
and recalling that $\D u$ is orthogonal to $\Sigma$,
we obtain that $f'''(r)=0$. This yields
\[
g
\,=\,
\dd r\otimes\dd r
\,+\,
\Big(\frac r{r_0}\Big)^{\!2}
g_\Sigma \, ,
\]
where $r_0>0$ is such that $\Sigma=\{r=r_0\}$.
Finally, if $O$ is the tip of the above cone, 
since
\[
|\pa B(O,R)|
\,=\,
\int\limits_{\pa B(O,R)}\!\!\!\!\!\dd\sigma
\,=\,
\int\limits_{\pa\Sigma}
\left(\frac R{r_0}\right)^{\!\!n-1}\!\!\!\!  \sqrt{\det g^\Sigma_{ij}}
\,\,\dd\theta^1\hdots\dd\theta^{n-1}
\,=\,
\left(\frac R{r_0}\right)^{\!\!n-1}
|\Sigma| \, ,
\]
we can explicitly express $r_0$ in terms
of $\AVR$ as in \eqref{cono} by 
\[
\AVR
=
\lim_{R\to+\infty}
\frac{|\pa B(O,R)|}{R^{n-1}|\Sf^{n-1}|}
\,=\,
\frac{|\Sigma|}{r_0^{n-1}|\Sf^{n-1}|} \, .
\]
The first identity is a well known characterization of the $\AVR$ in terms of areas of geodesic balls instead of volumes, and follows easily by the standard proof of the Bishop-Gromov Theorem (see e.g. \cite{Petersen_book}).  The proof of the second claim is completed.
\end{proof}

We now briefly record some of the main relations among geometric quantities induced by the two metrics. We omit the computations, since they are straightforward and completely analogous to those carried out in \cite{Ago_Maz_1} and \cite{Ago_Maz_3}. First, observe that
\begin{equation}
\label{nablaphi}
\abs{\nabla \phi}_{\tilde{g}} = \frac{\abs{\D u}}{u^{\frac{n-1}{n-2}}} \, .
\end{equation}
Let $\HH$ and $\HH_{\tilde{g}}$ be  the mean curvatures of the level sets of $u$, that coincide with those of $\phi$, respectively in the Riemannian manifold $(M\setminus\Omega, g)$ and in $(M\setminus\Omega, \tilde{g})$. They are computed using the unit normal vectors $-\D u/\abs{\D u}$ and $\nabla \phi / \abs{\nabla \phi}_{\tilde{g}}$, respectively. Exploiting the $g$-harmonicity and $\tilde{g}$-harmonicity of $u$ and $\phi$, we obtain that
\begin{equation}
\label{meancurv}
\HH \,= \,\frac{\D\D u (\D u, \D u)}{\abs{\D u}^3} \, , \qquad \HH_{\tilde{g}} \, = \, -\frac{\nabla\nabla\phi(\nabla \phi, \nabla \phi)}{\abs{\nabla\phi}_{\tilde{g}}^3} \, .
\end{equation}
These quantities are related as follows
\begin{equation}
\label{meancurv1}
\HH_{\cg} \, = \, u^{-\frac{1}{n-2}}\left[\HH- \left(\frac{n-1}{n-2}\right)\frac{\abs{\D u}}{u}\right] \, .
\end{equation}
Letting $\dd \sigma_{\tilde{g}}$ and $\dd\mu_{\tilde{g}}$ denote respectively the surface and the volume measure naturally induced by $ \tilde{g}$ on $M \setminus \Omega$, we have
\begin{equation}
\label{measures}
\dd\sigma_{\tilde{g}} \, = \, u^{\frac{n-1}{n-2}} \dd \sigma \, , \qquad \dd\mu_{\tilde{g}} \, = \, u^{\frac{n}{n-2}} \dd \mu \, . 
\end{equation}
Finally, for every $\beta\geq0$, define the conformal analogue of
the $U_\beta$ as the function  
\mbox{$\Phi_\beta: [0, +\infty) \longrightarrow \R$}
mapping 
\begin{equation}
\label{eq:fip}
\Phi_\beta(s) 
\,\,=\!\!\!
\int\limits_{\{\ffi = s\}}\!\!\!
|\na \ffi|_{\cg}^{\beta+1} \dd\sigma_{\cg}.
\end{equation}
The functions $U_\beta$ and $\Phi_\beta$,
and their derivatives are related to each other
as follows
\begin{align}
U_\beta \,&= \,\Phi_\beta(-\log t), \\
-tU_\beta'(t)\,&= \,\Phi_\beta'(-\log t) \label{uder-phider} \, ,
\end{align} 
for $0<t\leq 1$. The following theorem is the conformal version of the Monotonicity-Rigidity Theorem~\ref{thm:M/R}.

\begin{theorem}
\label{thm:main_conf}
Let $(M, g)$ be a nonparabolic Riemmanian manifold 
with $\ric\geq 0$.
Let $\Omega\subset M$ be a bounded and open subset with smooth boundary, and let $\tilde{g}, \phi$ and $\Phi_\beta$ be defined respectively as in \eqref{gconf}, \eqref{derphi} and \eqref{eq:fip}. 
Then, for every $\beta\geq (n-2)/(n-1)$, 
the function $\Phi_\beta$ is differentiable 
with derivative
\begin{equation}
\label{derivata_di_Phi}
\frac{\dd\Phi_\beta}{\dd s}(s) 
\,\,=\,\, 
- \, \beta\!\!\!\! \int\limits_{\{\ffi = s \}}\!\!\!\!  
|\na\ffi|_{\cg}^\beta\,\HH_{\cg} \,\dd\sigma_g \, ,
\end{equation}
where $\HH_g$ is the mean curvature of the level set $\{\ffi=s\}$
computed with respect to the unit normal vector field 
$\nu_{\cg}=\na \ffi/|\na \ffi|_{\cg}$.
Moreover, for every $s\geq0$,
the derivative fulfills 
\begin{equation}
\label{eq:der_fip}
\frac{\dd\Phi_\beta}{\dd s}(s)  \,\, 
= -\,\,\beta\,\,{\rm e}^s\!\!\!
\int\limits_{\{\ffi \geq  s\}} \!\!
\frac{
|\na\ffi|_{\cg}^{\beta-2} 
\Big(\ric(\na\ffi,\na\ffi)
+\big|\na\na \ffi\big|_{\cg}^2
 +  \, (\beta-2) \, \big| \na |\na \ffi |_{\cg}\big|_{\cg}^2\,\Big) 
}
 {{\rm e}^\ffi}
\,\,\dd\mu_{\cg}  \, .
\end{equation}
In particular, $\dd\Phi_\beta/\dd s$ is always nonpositive. 
Moreover, $(\dd\Phi_\beta/\dd s)(s_0)=0$
for some $s_0\geq1$ and some $\beta\geq(n-2)(n-1)$
if and only if $ \big(\{\ffi\geq s_0\}, \tilde{g} \big)$ is isometric  
to the Riemannian product  $\left([s_0, +\infty) \times \{\phi = s_0\}, d\rho\otimes d\rho + \tilde{g}_{\{\mid \phi = s_0\}}\right)$. In particular, $\{\phi = s_0\}$ is a connected totally geodesic submanifold. 
\end{theorem}
Observe that the right hand side in~\eqref{eq:der_fip} is nonpositive because of the refined Kato's inequality for harmonic functions
\begin{equation}
\label{kato}
\big\vert{\nabla\nabla \phi}\big\vert_{\cg}^2 \,\, \geq \,\, \left(\frac{n}{n-1}\right) \big| \na |\na \ffi |_{\cg}\big|_{\cg}^2 \, .
\end{equation}
Notice also the striking analogy between the above statement and Theorem \ref{mono-par}.

We are now going to show how to recover the Monotonicity-Rigidity Theorem~\ref{thm:M/R} for nonparabolic manifolds from its conformal version.

\begin{proof}[Proof of Theorem~\ref{thm:M/R} after Theorem \ref{thm:main_conf}]

Deducing formulas \eqref{derivata_di_U} and \eqref{eq:monot} from formulas \eqref{derivata_di_Phi} and \eqref{eq:der_fip} is just a matter of lengthy but straightforward computations carried out using the relations between $g$ and $\cg$ recalled above. We sketch the main steps. First, compute
\begin{equation}
\label{hessphi-u}
\abs{\nabla \nabla \phi}_{\cg}^2 \,  = \, u^{-\frac{4}{n-2}}\left\{\left\vert\frac{\D\D u}{u}\right\vert^2 + \,  \frac{n(n-1)}{(n-2)^2} \left\vert \frac{\D u}{u}\right\vert^4 - \,  \frac{2n}{n-2}\left\vert\frac{\D u}{u}\right\vert^3 \HH \right\},
\end{equation}
where $\HH$ is defined as in \eqref{meancurv},
and
\begin{equation}
\label{nablaphi-u}
\big\vert\nabla \abs{\nabla \phi}_{\cg}\big\vert_{\cg}^2 \,  = \,  u^{-\frac{4}{n-2}}\left\{\left\vert\frac{\D\abs{\D u}}{u}\right\vert^2 \! + \left(\frac{n-1}{n-2}\right)^{\!2} \left\vert \frac{\D u}{u}\right\vert^4 \!- 2\, \left(\frac{n-1}{n-2} \right)\left\vert\frac{\D u}{u}\right\vert^3 \HH \right\}.
\end{equation}
By \eqref{hessphi-u} and \eqref{nablaphi-u}, we can write
\begin{equation}
\label{quasi-quadratoni-u}
\begin{split}
\abs{\nabla \nabla \phi}_{\cg}^2 + (\beta -2)\big\vert\nabla \abs{\nabla \phi}_{\cg}\big\vert_{\cg}^2\,\, & = \,\,  u^{-\frac{4}{n-2}}\,\Bigg\{\frac{\abs{\D\D u}^2 - (\frac{n}{n-1})\big\vert \D \abs{\D u}\big\vert^2}{u^2} \,\, + \\
+ \left(\beta -\frac{n-1}{n-2}\right)&\bigg[\left\vert\frac{\D\abs{\D u}}{u}\right\vert^2 \! + \left(\frac{n-1}{n-2}\right)^{\!2} \left\vert \frac{\D u}{u}\right\vert^4 \!  - \, 2 \left(\frac{n-1}{n-2} \right) \left\vert\frac{\D u}{u}\right\vert^3 \HH \bigg] \Bigg\} \, .
\end{split}
\end{equation}
Now, considering a orthonormal frame as $\{e_1, \dots, e_{n-1}, \D u / \abs{\D u}\}$, where the first $n-1$ vectors are tangent to the level sets of $u$, we can decompose
\begin{equation}
\label{decomposition}
\left\vert\frac{\D\abs{\D u}}{u}\right\vert^2 \! = \, \left\vert\frac{{\D u}}{u}\right\vert^2 \HH^2 + \sum_{j=1}^{n-1}\left\langle \frac{\D\abs{\D u}}{u} , e_j \right\rangle^{\!2} .
\end{equation}
Plugging the above decomposition into~\eqref{quasi-quadratoni-u}, we obtain, with the aid of some algebra,
\begin{equation}
\label{quadartoni-ok}
\begin{split}
\abs{\nabla \nabla \phi}_{\cg}^2 + (\beta -2)\big\vert\nabla \abs{\nabla \phi}_{\cg}\big\vert_{\cg}^2 \,= \,\, u^{-\frac{4}{n-2}}
&\,\,\bigg\{\left[\,\,
\big|\D\D u\big|^2-{ \big(\textstyle\frac n{n-1}}\big)
\big|\D|\D u|\big|^2 \,\, \right] \\ 
&
\,\,\,+
\big({\textstyle\beta-\frac{n-2}{n-1}}\big)
\,\, \big|\D^T|\D u|\big|^2 \,\,
\\
& 
\,\,\,+
\big({\textstyle\beta-\frac{n-2}{n-1}}\big) \,\,
|\D u|^2
\left[ \, 
\HH- 
\big({\textstyle\frac{n-1}{n-2}}\big) |\D \log u|
 \, \right]^2 \bigg\} \, ,
\end{split}
\end{equation}
where 
\[
\big|\D^T|\D u|\big|^2\,  =\,\, \sum_{j=1}^{n-1}\left\langle \frac{\D\abs{\D u}}{u} , e_j \right\rangle^{\!2}.
\]
The monotonicity formula \eqref{eq:monot} now follows easily. On the other hand, \eqref{derivata_di_U}  follows from \eqref{derivata_di_Phi} by a direct computation.
\smallskip

Assume now that $U'_\beta (t_0) = 0$ for some $t_0 \in (0, 1]$ and some $\beta \geq 1$ (recall Remark \ref{optimal}). Then, by \eqref{uder-phider}, $\Phi_\beta'(-\log t_0) = 0$.  Then, through Kato's inequality~\eqref{kato}, \eqref{eq:der_fip} implies that $\big\vert \nabla\abs{\nabla \phi}_{\cg} \big\vert_{\cg} = 0$ in $\{\phi \geq - \log t_0\}$. The rigidity statement in Theorem~\ref{thm:M/R} then follows from Lemma~\ref{splitting-principle}.
\end{proof}
\subsection{Proof of Theorem \ref{thm:main_conf}}
As already pointed out in Remark~\ref{optimal},  we restrict ourself to the range $\beta \geq 1$, since it  is sufficient for the main geometric corollaries of Theorem~\ref{thm:M/R}. The arguments are considerably easier in this regime, however in Remark~\ref{difficulties}  we will comment on the main difficulties that arise in the case $(n-2)/(n-1) \leq \beta < 1$ and about how they can be overtaken, following the strategy described in~\cite{Ago_Maz_3}.

In what follows, we are always referring to a background nonparabolic Riemannian manifold $(M, g)$ with $\ric \geq 0$. Outside a bounded and open subset $\Omega \subset M$ with smooth boundary, we define the conformal metric $\tilde{g}$ as in~\eqref{gconf}. In particular, with the same notations as in the previous subsection, $(M \setminus \Omega, \tilde{g}, \phi)$ is a solution  to~\eqref{pbconf}.

A first fundamental ingredient for the proof of the (conformal) Monotonicity-Rigidity Theorem~\ref{thm:main_conf} is the upper bound for the function $\Phi_\beta$'s. We are going to prove it by first showing that $\abs{\nabla \phi}_{\cg}$ is uniformly bounded in $M \setminus \Omega$. Although not necessary for our aim, we actually provide a sharp upper bound for such a function, together with a rigidity statement when this bound is attained. The proof is substantially the one proposed in~\cite[Lemma 3.5]{Ago_Maz_1} for the Euclidean case.
\begin{proposition}[Sharp Gradient Estimate]
The inequality
\begin{equation}
\label{bound-nablaphi}
\abs{\nabla \phi}_{\cg}(q) \,\, \leq \,\, \sup_{\partial \Omega} \, \abs{\nabla \phi}_{\cg}
\end{equation}
holds for every $q \in M \setminus \Omega$. Moreover, if the equality is attained for some $q \in M \setminus \Omega$, then the Riemannian manifold $\left(M \setminus \Omega, \cg\right)$ is isometric to $\big([0, +\infty) \times \partial \Omega, d\rho\otimes d\rho + \tilde{g}_{\partial\Omega}\big)$.
\end{proposition}
\begin{proof}
We first show that there exists a constant $C = C(M, \Omega) >0 $ such that $\abs{\nabla \phi}_{\cg}\leq C$ in $M \setminus \Omega$. Fixing a reference point $O$ inside $\Omega$, and letting $d(O, \cdot)$ be the distance from this point with respect to the metric $g$, we have, by \eqref{yau1}, that
\[
\abs{\D u}(q) \,\leq \,  C \, \frac{u(q)}{d(O, q)}
\]
outside some ball containing $\Omega$. Then, 
\[
\abs{\nabla \phi}_{\tilde{g}} (q) \, = \,  \frac{\abs{\D u}}{u^{\frac{n-1}{n-2}}}(q) \, \leq \,  C \, \frac{u^{-\frac{1}{n-2}}(q)}{d(O, q)} \, \leq \,  C \,  C_1^{-\frac{1}{n-2}} \, ,
\]
where in the last inequality we used \eqref{lower}.
Consider now, for a given constant $\alpha > 0$, the auxiliary function 
$$
w_\alpha = \abs{\nabla \phi}^2_{\cg}{\rm e}^{-\alpha \phi} \, .
$$  
Observe that by the just proved upper bound on $\abs{\nabla \phi}_{\cg}$ we have $w_\alpha(q) \to 0$ as $d(O, q) \to +\infty$ for any $\alpha > 0$. Moreover, a direct computation combined with \eqref{boch1} shows that $w_\alpha$ satisfies the  identity
\begin{equation}
\label{eq-walpha}
\Delta_{\cg} w_\alpha - (1-2\alpha) \left\langle \nabla w_\alpha, \nabla \phi\right\rangle_{\cg} - \alpha(1-\alpha) \abs{\nabla \phi}_{\cg}^2 w_\alpha \,\,  = \,\, 2\, {\rm e}^{-\alpha \phi}\left[\ric (\nabla \phi, \nabla \phi) + \abs{\nabla \nabla \phi}_{\cg}^2 \right].
\end{equation}
In particular, using the Maximum Principle, one has that for every $q \in M \setminus \Omega$ it holds 
$$
w_\alpha (q) \, \leq \,  \sup_{\partial \Omega} \, \abs{\nabla \phi}_{\cg} \, .
$$
Letting $\alpha \to 0$ in the above inequality yields \eqref{bound-nablaphi} for every $q \in M \setminus \Omega$. Assume now that the maximum value of $\abs{\nabla \phi}_{\cg}$ is attained at some interior point $q \in M \setminus \overline{\Omega}$. Then, by the Strong Maximum Principle $\abs{\nabla \phi}_{\cg}^2$ must be constant, since it satisfies 
\[
\Delta_{\tilde{g}}\abs{\nabla \phi}_{\tilde{g}}^2 - \langle \nabla \abs{\nabla \phi}_{\tilde{g}}^2 \, , \nabla \phi\rangle_{\tilde{g}} \, \geq \,0 \,.
\]
Lemma \ref{splitting-principle} then yields the desired conclusion.
\end{proof}

\begin{remark}
The above results, in terms of $(M \setminus \Omega, g)$ and $u$, says that for every $q \in M \setminus \Omega$
\[
\frac{\abs{\D u}}{u^{\frac{n-1}{n-2}}}(q) \, \leq \, \sup_{\partial \Omega} \frac{\abs{\D u}}{u^{\frac{n-1}{n-2}}} \, ,
\]
with equality attained only if $(M \setminus \Omega, g)$ is isometric to a truncated cone. This is exactly the content of Colding's~\cite[Theorem 3.2]{Colding_1}.
\end{remark}

A direct consequence is that the functions $\Phi_\beta$'s are bounded.
\begin{corollary}
\label{cor-bound}
For every $\beta \geq 0$ the function $\Phi_\beta: [0, +\infty) \to \R$ defined in \eqref{eq:fip} is bounded as
\[
\Phi_\beta(s) \,\, \leq \,\,  \sup_{\partial \Omega} \, \abs{\nabla \phi}_{\cg}^\beta \,\, \int\limits_{\partial \Omega}  \abs{\nabla \phi}_{\cg} \dd\sigma_{\cg} \, .
\]
for every $s \in [0, +\infty)$.
\end{corollary}
\begin{proof}
Just observe that a simple application of the Divergence Theorem combined with the $\tilde{g}$-harmonicity of $\phi$ gives the constancy in $s$ of the function 
\[
s \longmapsto \int\limits_{\{\phi = s\}} \!\abs{\nabla \phi}_{\cg} \, \dd \sigma_{\cg} \, .
\]
Combining this observation with the bound~\eqref{bound-nablaphi}, we have
\[
\Phi_\beta (s) \,\, = \!\!\!\! \int\limits_{\{\phi = s\}}\!\!\! \abs{\nabla \phi}_{\cg}^\beta \, \abs{\nabla \phi}_{\cg} \, \dd\sigma_{\cg} \,\, \leq \,\, \sup_{\partial \Omega} \, \abs{\nabla \phi}_{\cg}^\beta \int\limits_{\partial \Omega} \abs{\nabla \phi}_{\cg} \, \dd\sigma_{\cg} \, ,
\]
for every $s \in [0, +\infty)$, as claimed.
\end{proof}
A fundamental tool in the forthcoming computations, leading to the expression of the derivative of $\Phi_\beta$ in terms of a nonpositive integral (as in~\eqref{eq:der_fip}), is the following Bochner-type identity.

\begin{lemma}[Bochner-type identity]
\label{boch}
At every point where $\abs{\nabla \phi}_{\tilde{g}} \neq 0$, the following identity holds for every $\beta \geq 0$
\begin{equation}
\label{bochf}
\Delta_{\tilde{g}} \abs{\nabla \phi}^\beta_{\tilde{g}} - \left\langle \nabla \abs{\nabla \phi}_{\tilde{g}}^\beta, \nabla \phi\right\rangle_{\!\tilde{g}}  \,  = \,\,  \beta \, \abs{\nabla \phi}_{\tilde{g}}^{\beta - 2} \left[ \, \ric(\nabla \phi, \nabla \phi) + {\abs{\nabla \nabla \phi}_{\tilde{g}}^2} + (\beta - 2) \big\vert \nabla \abs{\nabla \phi}_{\tilde{g}} \big\vert^2_{\tilde{g}}\, \right] \!,
\end{equation}
where the $\ric$ is the Ricci tensor of the background metric $g$.
\begin{proof}
By a direct computation one gets
\[
\Delta_{\tilde{g}}\abs{\nabla \phi}_{\tilde{g}}^\beta \,\, = \,\,  \abs{\nabla \phi}_{\tilde{g}}^{\beta - 2}\left[ \, \frac{\beta}{2}  \Delta_{\tilde{g}}\abs{\nabla \phi}_{\tilde{g}}^2 \,  + \,  \beta(\beta - 2) \big\vert \nabla \abs{\nabla \phi}_{\tilde{g}}\big\vert_{\tilde{g}}^2 \, \right],
\]
that, combined with \eqref{boch1}, leads to \eqref{bochf}.
\end{proof}
\end{lemma}

We prove now an integral identity that will enable us to link the derivative of $\Phi_\beta$ to the Bochner-type formula above. This is where the assumption $\beta \geq 1$ enters the game, compare with Remark~\ref{optimal} and Remark~\ref{difficulties}.

\begin{lemma}[Fundamental Integral Identity]
\label{integral_id}
Let $0\leq s<S<+\infty$, and let $\beta\geq 1$. Then the following identity holds
\begin{equation}
\label{identity}
\begin{split}
\bigintssss\limits_{\{\phi = S\}}\!\!\! \frac{\left\langle \nabla \abs{\nabla \phi}_{\tilde{g}}^\beta, \frac{\nabla \phi}{\abs{\nabla \phi}_{\tilde{g}}}\right\rangle}{e^S} \dd\sigma_{\tilde{g}} \,\, &- \!\! \bigintssss\limits_{\{\phi = s\}}\!\!\! \frac{\left\langle \nabla \abs{\nabla \phi}_{\tilde{g}}^\beta, \frac{\nabla \phi}{\abs{\nabla \phi}_{\tilde{g}}}\right\rangle}{e^s}\dd \sigma_{\tilde{g}} \,\, =   \\
&= \,\, \beta \!\!\!\!\!\!\!\bigintssss\limits_{\{s \leq \phi \leq S\}}\!\!\!\! \!\!\!\frac{\abs{\nabla\phi}_{\tilde{g}}^{\beta -2} \left[\ric({\nabla \phi, \nabla \phi}) + \abs{\nabla \nabla \phi}_{\tilde{g}}^2 + (\beta - 2)\big\vert\nabla\abs{\nabla \phi}_{\tilde{g}}\big\vert_{\tilde{g}}^2\right]}{e^\phi}\dd \mu_{\tilde{g}} \, ,
\end{split}
\end{equation}
where the Ricci tensor is referred to the background metric $g$.
\end{lemma}

\begin{proof}
Since all the quantities that appear in this proof are referred to the metric $\cg$, except for $\ric$, which is the Ricci tensor of the background metric $g$, we shorten the notation, dropping the subscript $\tilde{g}$. We consider the vector field
\[
X = \frac{\nabla \abs{\nabla \phi}^\beta}{e^\phi},
\]
that is well defined at every point where $\abs{\nabla \phi} \neq 0$, and thus $\sigma_{\cg}$-almost everywhere, in virtue of the regularity theory for elliptic equations recalled above. Moreover, by the Bochner-type identity \eqref{bochf}, we have that wherever $\abs{\nabla \phi} \neq 0$ it holds
\begin{equation}
\label{divx}
\dive X \, = \,  \frac{\Delta\abs{\nabla \phi}^\beta - \langle \nabla \abs{\nabla \phi}^\beta, \nabla \phi \rangle}{e^\phi}\, = \, \beta \, \frac{\abs{\nabla \phi}^{\beta -2}\left[\ric({\nabla \phi, \nabla \phi}) + \abs{\nabla \nabla \phi}^2 + (\beta -2)\big\vert\nabla\abs{\nabla \phi}\big\vert^2 \right]}{e^\phi}.
\end{equation}
Let us consider, in $\{s \leq \phi \leq S\}$, a tubular neighborhood of the critical set 
$$
B_\epsilon [\crit] \,\,\, = \!\! \bigcup_{x \in \crit} \!\!\!B_\ep(x) \, ,
$$
It is well known that, for almost every $\epsilon > 0$, the boundary of $B_\epsilon [\crit]$ is $(n-1)$-Lipschitz rectifiable with respect to $\sigma_{\tilde{g}}$ (see e.g. \cite[Theorem 2.5]{alberti} for a much more general result). Then, the Divergence Theorem applies and yields
\begin{equation}
\label{div2}
\begin{split}
\int\limits_{\{s \leq \phi \leq S\} \setminus B_{\epsilon_j}[\crit]}\!\!\!\!\!\!\!\!\! \!\!\!\!\!\!\!\!\!\!\dive X \,\dd \mu \,\,\, \,\,\,  = \!\!\!\!\!\!\!\!\!\! \bigintssss\limits_{\{\phi = S\}\setminus B_{\epsilon_j} (\crit)} \!\!\!\!\!\!\!\!\!\!\!\!\!\!\!\!\frac{\left\langle \nabla \abs{\nabla \phi}^\beta, \frac{\nabla \phi}{\abs{\nabla \phi}}\right\rangle}{e^S} \, \dd\sigma \,\,\,\,\,  &- \!\!\!\!\!\! \!\!\!\!\!\! \bigintssss\limits_{\{\phi = s\}\setminus B_{\epsilon_j} [\crit]}  \!\!\!\!\!\!\!\!\!\!\!\!\!\!
 \frac{\left\langle \nabla \abs{\nabla \phi}^\beta, \frac{\nabla \phi}{\abs{\nabla \phi}}\right\rangle}{e^s} \, \dd \sigma \\ 
&+ \!\!\! \!\!\!\int\limits_{\partial B_{\epsilon_j}[\crit]}\!\!\!\!\!\!\! \frac{\left\langle \nabla \abs{\nabla \phi}^\beta, \nu_{\epsilon_j} \right\rangle}{e^\phi} \, \dd\sigma,
\end{split}
\end{equation}
for a sequence $\epsilon_j \to 0^+$,
where $\nu_{\epsilon_j}$ is the exterior normal to $\partial B_{\epsilon_j}[\crit]$, that is well defined almost everwhere on such a set due to rectifiability.
By  Kato's inequality 
\begin{equation}
\label{kato1}
\abs{\nabla \nabla \phi}^2 \geq \big\vert\nabla \abs{\nabla \phi} \big\vert^2,
\end{equation}
and identity \eqref{divx}, $\dive X$ is nonnegative. 
Thus,
by the Monotone Convergence Theorem, we have
\[
\int\limits_{\{s \leq \phi \leq S\}}\!\!\!\!\!\! \dive X \,\dd \mu \,\,\,\,   = \,\,\,\, \lim_{\epsilon \to 0^+}  \!\!\!\!\!\!\!\!\!\!\!\!\!\int\limits_{\{s \leq \phi \leq S\} \setminus B_\epsilon[\crit]}\!\!\!\!\!\!\!\!\! \!\!\!\!\! \!\!\!\dive X \,\dd \mu \, .
\]  
Moreover, observe that the vector field $\big\vert{\nabla \abs{\nabla \phi}^\beta}\big\vert$ is bounded in $\{s \leq \phi \leq S\}$ for $\beta \geq 1$ again by  Kato's inequality
\[
\nabla \abs{\nabla \phi}^\beta = \beta \abs{\nabla \phi}^{\beta -1} \big\vert \nabla \abs{\nabla \phi} \big\vert \leq \beta \abs{\nabla \phi}^{\beta -1}\abs{\nabla\nabla \phi},
\]
and thus the Dominated Convergence Theorem gives
\[
\lim_{\epsilon_j \to 0^+} \!\!\!\!\!\!\!\!\!\!\!\!\bigintssss\limits_{\{\phi = s\}\setminus B_{\epsilon_j} [\crit]}\!\!\!\!\!\!\! \!\!\!\!\frac{\left\langle \nabla \abs{\nabla \phi}^\beta, \frac{\nabla \phi}{\abs{\nabla \phi}}\right\rangle}{e^s} \, \dd \sigma \,\,\,\, = \!\! \bigintssss\limits_{\{\phi = s\}}\!\!\frac{\left\langle \nabla \abs{\nabla \phi}^\beta, \frac{\nabla \phi}{\abs{\nabla \phi}}\right\rangle}{e^s} \, \dd \sigma \, , 
\]
and analogously for $\{\phi = S\}$.
We are then left to show that 
\begin{equation}
\label{tub}
\lim_{\epsilon_j \to 0^+} \!\!\! \!\!\!\int\limits_{\partial B_{\epsilon_j}[\crit]}\!\!\!\!\!\! \frac{\left\langle \nabla \abs{\nabla \phi}^\beta, \nu_\epsilon \right\rangle}{e^\phi} \dd\sigma = 0.
\end{equation}
To see this, we observe that the integrand os bounded and we use the estimate on the tubular neighbourhoods of the critical set given in \cite[Theorem 1.17]{Che_Nab_Val}.
Namely, for every $\eta > 0$ and every $\epsilon > 0$ there exists a constant $C_\eta>0$ such that
\begin{equation}
\label{mink}
\abs{B_\epsilon(\crit)} \, \leq\,  C_\eta \, \epsilon^{2-\eta}.
\end{equation}
Let $\rho$ be the distance function from the closed set $\crit$. Then, $\abs{\nabla \rho} = 1$ almost everywhere outside $\crit$ (see \cite[Theorem 3.1]{Mant_Menn} for much deeper insights on distance functions). Therefore, by the co-area formula and \eqref{mink},
we get
\begin{equation}
\label{co-mink}
\int\limits_0^\epsilon \dd\rho \!\!\!\! \!\!\int\limits_{\partial B_\rho[\crit]}\!\!\!\!\!\!\!\!\!\dd\sigma  \,\, \leq \,\,  C_\eta \, \epsilon^{2-\eta} \, .
\end{equation}
By the Mean Value Theorem, we can find a sequence $\epsilon_j \to 0^+$ such that
\[
\int\limits_{\partial B_{\epsilon_j}[\crit]}\!\!\!\!\!\!\!\!\!\dd\sigma \,  \leq  \, C_\eta \, \epsilon_j^{1-\eta},
\]
and then, up to relabeling, the limit in \eqref{tub} holds true, completing the proof.
\end{proof}

\begin{remark}
\label{difficulties}
As already observed, the refined Kato inequality for harmonic function~\eqref{kato} implies that, whenever $\nabla \phi \neq 0$, the integrand on the right hand side of~\eqref{identity} is nonnegative also for $(n-2)/(n-1) \leq \beta <1$. However, for $\beta$ in this range, the vector field $X = e^{-\phi} \nabla \abs{\nabla \phi}_{\cg}^\beta$ is no longer uniformly bounded in a neighborhood of the critical set of $\phi$, and thus the above argument is not sufficient to prove~\eqref{tub}. This difficulty has been recently overcome in~\cite{Ago_Maz_3}, introducing a new family of tubular neighborhoods 
$U_\epsilon = \{\abs{\nabla \phi}_{\cg}^2 \leq \epsilon\}$ of $ \crit$, in place of the standard ones, and then modifying the vector field $ X = e^{-\phi} \nabla \abs{\nabla \phi}^\beta$ through a suitable cut-off function that vanishes inside $U_\ep$. The key point is that the divergence of the modified vector field always remains nonnegative. This scheme is shown to converge when $\ep \to 0$, leading to the validity of Lemma~\ref{integral_id} for every $(n-2)/(n-1) \leq \beta <1$. We refer the reader to~\cite{Ago_Maz_3} for the details of this construction.
\end{remark}

Building on the Fundamental Integral Identity proved in Lemma~\ref{integral_id}, we are now ready to deduce the following monotonicity result. Again, the case $(n-2)/(n-1) \leq \beta <1$ can be treated adapting the argument presented in~\cite{Ago_Maz_3}.

\begin{theorem}[Monotonicity of $e^{-s}\Phi'_{\beta}(s)$]
\label{mono1} 
Let $\beta \geq 1$. The function $\Phi_\beta$ defined in \eqref{eq:fip} is differentiable for any $s \geq 0$, and its derivative satisfies  
\begin{equation}
\label{derphi}
\Phi_\beta' (s) \,\, =\!\! \int\limits_{\{\phi = s\}}\!\!\! {\left\langle \nabla \abs{\nabla \phi}_{\tilde{g}}^\beta, \frac{\nabla \phi}{\abs{\nabla \phi}_{\tilde{g}}}\right\rangle} \dd\sigma_{\tilde{g}} \,\, = \,\, -\beta\!\! \int\limits_{\{\phi =s\}} \abs{\nabla \phi}_{\tilde{g}} \HH_{\tilde{g}} \, \dd\sigma_{\tilde{g}} \, .
\end{equation}
In particular,
 for any $0\leq s<S<\infty$, we have
\begin{equation}
\label{mono1f}
e^{-S}\Phi'_{\beta}(S) - e^{-s}\Phi'_\beta(s) \,\, = \,\, \beta\!\!\!\!\!\!\! \int\limits_{\{s \leq \phi \leq S\}}\!\!\!\!\!\!\abs{\nabla \phi}_{\tilde{g}}^{\beta-2}\left[\ric(\nabla \phi, \nabla \phi) + \left\vert \nabla \nabla \phi\right\vert_{\tilde{g}}^2 + (\beta -2)\left\vert \nabla \abs{\nabla \phi}_{\tilde{g}}\right\vert_{\tilde{g}}^2 \right] e^{-\phi} d\mu_{\tilde{g}} \, ,
\end{equation} 
where the Ricci tensor is referred to the background metric $g$.

\begin{proof}
Let us drop the subscript $\tilde{g}$. 
Let $s_0 \in [0, +\infty)$, and $s\geq s_0$. Consider
\[
\Phi_\beta (s) - \Phi_\beta(s_0)  \,\,=\!\! \int\limits_{\{\phi =s\}} \left\langle \abs{\nabla \phi}^\beta \nabla \phi, \frac{\nabla \phi}{\abs{\nabla \phi}}\right\rangle \dd\sigma\,\, - \!\! \! \!\int\limits_{\{\phi =s_0\}} \left\langle \abs{\nabla \phi}^\beta \nabla \phi, \frac{\nabla \phi}{\abs{\nabla \phi}}\right\rangle \dd\sigma \, .
\]
Ruling out the critical set of $\phi$ as shown in the proof of the above Lemma \ref{integral_id}, we can apply the Divergence Theorem to get
\[
\Phi_\beta(s) - \Phi_\beta(s_0)\,\,\, = \!\!\!\!\!\int\limits_{\{s_0 \leq \phi \leq s\}} \!\!\!\! \dive (\abs{\nabla \phi}^\beta \nabla \phi) \dd\mu \,\,\,= \!\!\!\!\!\int\limits_{\{s_0 \leq \phi \leq s\}}\!\!\!\! \left\langle \nabla \abs{\nabla \phi}^\beta, \nabla \phi\right\rangle \dd \mu \, ,
\]
where the last identity is due to harmonicity of $\phi$.
By co-area formula, the above quantity can also be written as
\begin{equation}
\label{div+coarea}
\Phi_\beta(s) - \Phi_\beta(s_0)\,\, = \int\limits_{s_0}^s \! \dd\tau \!\!\!\int\limits_{\{\phi = \tau\}} \!\!\left\langle \nabla \abs{\nabla \phi}^\beta, \frac{\nabla \phi}{\abs{\nabla \phi}}\right\rangle \dd\sigma \, .
\end{equation}
By the Fundamental Theorem of Calculus, the continuity of the mapping
\[
\tau \longmapsto  I(\tau) \, = \!\!\int\limits_{\{\phi = \tau\}} \!\!\left\langle \nabla \abs{\nabla \phi}^\beta, \frac{\nabla \phi}{\abs{\nabla \phi}}\right\rangle \dd\sigma
\]
implies the differentiability of $\Phi_\beta$, together with the first identity in~\eqref{derphi}. The second one follows from the first just by a direct computation involving \eqref{meancurv}.
In order to see the continuity of $I$ at a time $\tau_0$, consider, for $\tau > \tau_0$, the difference
\[
e^{-\tau}I(\tau) - e^{-\tau_0} I(\tau_0)\,\,\, = \,\,\, e^{-\tau}\!\!\!\!\int\limits_{\{\phi = \tau\}} \!\!\left\langle \nabla \abs{\nabla \phi}^\beta, \frac{\nabla \phi}{\abs{\nabla \phi}}\right\rangle \dd\sigma \,\, - \,\, e^{-\tau_0}  \!\!\!\!\! \int\limits_{\{\phi = \tau_0\}} \!\!\left\langle \nabla \abs{\nabla \phi}^\beta, \frac{\nabla \phi}{\abs{\nabla \phi}}\right\rangle \dd\sigma \, .
\]
Then, by \eqref{identity}, we have
\[
e^{-\tau}I(\tau) - e^{-\tau_0} I(\tau_0)\,\, = \,\, \beta \!\!\!\!\! \bigintssss\limits_{\{\tau_0 \leq \phi \leq \tau\}}\!\!\!\! \!\!\!\frac{\abs{\nabla\phi}^{\beta -2} \left[\ric({\nabla \phi, \nabla \phi}) + \abs{\nabla \nabla \phi}^2 + (\beta - 2)\big\vert\nabla\abs{\nabla \phi}\big\vert^2\right]}{e^\phi}\dd \mu\, .
\]
In particular, by the Dominated Convergence Theorem the above integral vanishes as $\tau \to \tau_0^{+}$, yielding the right continuity of $e^{-\tau} I(\tau)$ at $\tau_0$. By a completely analogous argument with $\tau < \tau_0$ we get that $e^{-\tau} I(\tau)$ is actually continuous at $\tau_0$, and clearly so is $I$, ending the proof.    
\end{proof}
\end{theorem}

We are now in a position to complete the proof of Theorem \ref{thm:main_conf}.

\begin{proof}[Conclusion of the proof of Theorem \ref{thm:main_conf}]
We are going to pass to the limit as $S \to +\infty$ in \eqref{mono1f}. The following argument is due to Colding and Minicozzi, \cite{Colding_Minicozzi}. We first prove that the derivative of $\Phi_\beta$ has a sign.
By \eqref{mono1f} we have that, for every $\beta \geq 1$ and every $0\leq s < S <+\infty$,
\[
\Phi_\beta'(S) \,\geq \, e^{(S-s)}\Phi_\beta'(s)\, .
\]
Integrating the above differential inequality, we get
\begin{equation}
\label{trick}
\Phi_\beta (S)\, \geq \, \left(e^{(S-s)} - 1\right)\Phi_\beta'(s) \, +  \, \Phi_\beta (s) \, .
\end{equation}
for every $0 \leq s < S < +\infty$.
Assume now, by contradiction, that $\Phi_\beta'(s) > 0$ for some $s\in [0, +\infty)$. Then, passing to the limit as $S \to +\infty$ in \eqref{trick}, we would get $\Phi_\beta(S) \to +\infty$, against the boundedness of $\Phi_\beta$ provided in Corollary \ref{cor-bound}. Thus, $\Phi_\beta'(s) \geq 0$ for every $s \in [0, +\infty)$. Therefore $\Phi_\beta$ is a nonincreasing, differentiable bounded function on $[0, +\infty)$, and in particular, $\Phi_\beta '(S) \to 0$ as $S \to +\infty$. Passing to the limit as $S \to +\infty$ in \eqref{mono1f} finally gives 
the monotonicity formula \eqref{eq:der_fip}.

For the rigidity statement, assume now that $\Phi_\beta'(s_0) = 0$ for some $s_0 \in [0, +\infty)$. Then, by \eqref{eq:der_fip}, and  Kato's inequality, we get $\big\vert \nabla \abs{\nabla \phi}\big\vert = 0$. We then conclude by Lemma~\ref{splitting-principle}.
\end{proof} 

\begin{remark}
\label{difficulties-rigidity}
For what it concerns  the rigidity statement when $\Phi'_\beta (s_0) = 0$ for some $s_0 \geq 0$ and $\beta <1$, observe that the only case to be discussed differently from $\beta \geq 1$ is $\beta = (n-2)/ (n-1)$. In fact, for this specific parameter, we just obtain from \eqref{eq:der_fip} that $\abs{\nabla \nabla \phi}_{\cg}= n/(n-1) \abs{\nabla \abs{\nabla \phi}_{\cg}}_{\cg}$. However, in this case, the isometry with a Riemannian cylinder substantially follows from \cite[Proposition 5.1]{Bou_Car}.  We refer the reader to \cite{Ago_Maz_3} for the details of this limit case.
\end{remark}

\section{Proof of the Willmore-type inequalities}
In this section, we are going to prove the Willmore inequality on manifolds with nonnegative Ricci curvature and Euclidean volume growth, using the geometric features of the capacitary potential $u$ of a given bounded domain with smooth boundary $\Omega$. For the reader's convenience we recall that $u$ is a solution to the following problem
\begin{equation*}
\begin{cases}
\,\,\Delta{u}= 0 & \mbox{in} \,\, M\setminus\overline\Omega  
\\
\,\,\,\,\,\,\,u=1 & \mbox{on}\,\, \partial\Omega 
\\
u(q)\to 0 &\mbox{as} \,\, d(O,q)\to+\infty \, ,
\end{cases}
\end{equation*}
whose existence in the present context follows from Theorem~\ref{existence-nonpar}. As sketched in the Introduction, the proof makes use of the global features of the Monotonicity-Rigidity Theorem, that is, it compares the behavior of $U_\beta$ in the large with $U_\beta(1)$. 
The well known asymptotics of $u$ and of its gradient in the Euclidean case were the crucial tool to compute the limits of $U_\beta$ in~\cite{Ago_Maz_1}, that consequently gave the sharp lower bound on the Willmore energy. On manifolds with nonnegative Ricci curvature and Euclidean volume growth, the work of Colding-Minicozzi~\cite{colding3} actually implies that the asymptotic behavior of the potential is completely analogous to that in $\R^n$. However, as we will clarify in Remark~\ref{no-hope}, there is no hope to get an Euclidean-like pointwise behavior of $\D u$ in the general case. Nevertheless, using techniques developed in the celebrated~\cite{Che-Cold}, we are able to achieve asymptotic integral estimates for the gradient that in turn will let us conclude the proof of the Willmore-type inequalities~\eqref{will}.

As in Section~\ref{sec:ingredients}, we denote by $\D$ the Levi-Civita connection on the manifold considered, by $\D\D$ the Hessian and by $\Delta$ the Laplacian. Moreover, we let $r(x)=d(O, x)$, where $O \in \Omega$.
To keep the notation shorter it is also useful to recall the following characterisation of the electrostatic capacity of $\Omega$
\[
\capa(\Omega)\, = \, \ddfrac{\int_{\partial \Omega} \abs{\D u} \dd\sigma}{(n-2)\abs{\Sf^{n-1}}} \, ,
\]
in terms of the capacitary potential $u$ of $\Omega$.

\subsection{Manifolds with Euclidean volume growth}

With respect to the value assumed by $\AVR$, we define manifolds with Euclidean volume growth and sub-Euclidean volume growth as follows.
\begin{definition}[Volume growth]
Let $(M, g)$ be a complete noncompact Riemannian manifold with $\ric\geq 0$. Then we say that it has \emph{Euclidean volume growth} if $\AVR > 0$, \emph{sub-Euclidean volume growth} otherwise.
\end{definition}

If follows at once that manifolds with Euclidean volume growth satisfy \eqref{var}, and thus from Varopoulos' characterization  \emph{manifolds with Euclidean volume growth are nonparabolic}.

Let us also recall the well known fact that in a noncompact complete Riemannian manifold $(M, g)$ with $\ric \geq 0$ also the function 
\[
(0, +\infty) \ni r \longmapsto \theta(r)= \frac{\abs{\partial B(O, r)}}{r^{n-1} \abs{\Sf^{n-1}}}
\]
is monotone nonincreasing for any $O \in M$, and the asymptotic volume ratio satisfies
\begin{equation}
\label{avr1}
\AVR \, = \lim_{r \to + \infty} \frac{\abs{\partial B(O, r)}}{r^{n-1} \abs{\Sf^{n-1}}} \, .
\end{equation}
In the remainder of this section we are repeatedly going to use both the area and the volume formulations of the Bishop-Gromov Theorem without always mentioning them.
\smallskip

Let us now derive some rough estimates for the electrostatic potential $u$ and its gradient $\D u$ holding on manifolds with nonnegative Ricci curvature and Euclidean volume growth, to be refined later.

\begin{proposition}
\label{bound_eucl}
Let $(M, g)$ be a complete noncompact Riemannian manifold with $\ric\geq 0$ and Euclidean volume growth. Then, 
it is nonparabolic and the solution $u$ to problem \eqref{pb} for some bounded and open $\Omega$ with smooth boundary satisfies
\begin{equation}
\label{bound_u1}
C_1\, r^{2-n}(x)\, \,\leq \,\, u(x) \,\, \leq \,\, C_2\, r^{2-n}(x) 
\end{equation}
on $M\setminus \Omega$ for some positive constants $C_1$ and $C_2$ depending on $M$ and $\Omega$. Moreover, if $\Omega \subset B(O, R)$, it holds
\begin{equation}
\label{bound_du}
\abs{\D u}(x) \,\,\leq \,\, C_3 \,  r^{1-n}(x) \, ,
\end{equation}
on $M \setminus B(O, 2R)$ with $C_3=C_3(M, \Omega)>0$.
\end{proposition}

\begin{proof}
The first inequality in \eqref{bound_u1} is just \eqref{lower}.
To obtain the second one, observe first that the monotonicity in the Bishop-Gromov Theorem implies, for any $O\in M$ and any $r \in (0, +\infty)$, that
\[
\abs{B(O, r)}\geq \left(n\abs{\Sf^{n-1}} \AVR\right) r^n .							
\]
Then, the second inequality in the Li-Yau estimate \eqref{liyau}, combined with the second inequality in \eqref{bound_u_g} completes the proof of \eqref{bound_u1}.
Finally, inequality \eqref{bound_du} is achieved just by plugging the upper estimate on $u$ given by \eqref{bound_u1} into \eqref{yau}.
\end{proof}

\subsection{Asymptotics for $u$}

The behavior at infinity of $u$ can be deduced along the path indicated in~\cite{colding3}. However, we prefer to present here a simplified version of that proof, taking advantage of some of the refinements provided in~\cite{li_tam_wang}. Let us first recall the following asymptotic behaviour of $G$ (see \cite[Theorem 0.1]{colding3}, or \cite[Theorem 1.1]{li_tam_wang} for a completely different proof),
\begin{equation}
\label{asyg}
\lim_{r(x) \to +\infty} \frac{G(O, x)}{r(x)^{2-n}}\, = \, \frac{1}{\AVR} \, .
\end{equation}
Building on this fact, we deduce precise asymptotics for the capacitary potential of $\Om$.
\begin{lemma}
\label{asympt_u}
Let $(M, g)$ be a complete noncompact  Riemannian manifold with $\ric \geq 0$ and Euclidean volume growth, and let $u$ be a solution 
to problem \eqref{pb}. Then
\begin{equation}
\label{limu}
\lim_{r(x) \to +\infty}\frac{u(x)}{r(x)^{2-n}}\, = \, \frac{\capa( \Omega)}{\AVR} \, .
\end{equation}
\end{lemma}
\begin{proof}

By \cite[Theorem 1.2]{li_tam_wang}, that slightly extends \cite[Theorem 0.3]{colding3}, we have that outside some large ball $B(O, R)$ containing $\Omega$
\begin{equation}
\label{litamwang}
u \,\,\ = \, - \,\, \frac{G}{(n-2)\abs{\Sf^{n-1}}} \int\limits_{B(O, R)}\!\! \frac{\partial u}{\partial \nu}\,  \dd\sigma  \, + \, v \, ,
\end{equation}
where $G$ is the Green's function with pole in $O$, $\nu$ is the exterior unit normal to the boundary of $B(O ,R)$ and $v$ is an harmonic function defined in $M \setminus B(O, R)$ satisfying
\begin{equation}
\label{vpiccola}
\abs{v} \,\,\leq \,C \, \frac{G}{r}
\end{equation}
for some constant $C > 0$. We point out that the Green's function considered in \cite{li_tam_wang} is, in our notation, $G /\big((n-2)\abs{\Sf^{n-1}}\big)$.
By the Divergence Theorem and the harmonicity of $u$, we infer that
\[
\int\limits_{\partial B(O, R)}\!\!\! \frac{\partial u}{\partial \nu} \dd\sigma \, + \int\limits_{\partial \Omega} \frac{\partial u}{\partial \nu}\dd \sigma \,\,\, = \!\!\!\!\!\int\limits_{B(O, R) \setminus \Omega} \!\!\!\!\!\!\! \Delta u \dd \mu \, = \,  0 \, ,
\]
where we denote by $\nu$ the exterior unit normal to the boundary of $B(O, R) \setminus \overline{\Omega}$. Since, on $\partial \Omega$, $\nu = \D u /\abs{\D u}$, we get, by the above identity,  that
\begin{equation}
\label{c1}
- \,\, \ddfrac{\int_{B(O, R)} \frac{\partial u}{\partial \nu} \dd\sigma}{(n-2)\abs{\Sf^{n-1}}}  \,\, = \,\, \ddfrac{\int_{\partial \Omega} \abs{\D u} \dd\sigma}{(n-2)\abs{\Sf^{n-1}}}\,\, = \,\,  \capa(\Omega) \, .
\end{equation}
Dividing both sides of \eqref{litamwang} by $r^{2-n}$ and passing to the limit as $r \to +\infty$ taking into account \eqref{asyg}, \eqref{c1} and \eqref{vpiccola}, we get the claim.
\end{proof}

As a straightforward corollary of the above lemma, we compute the rescaled area of large geodesic spheres in $M$.

\begin{corollary}
\label{asy-area}
Let $(M, g)$ be a complete noncompact Riemannian manifold with $\ric \geq 0$ and Euclidean volume growth, and let $u$ be a solution to \eqref{pb}. Then
\begin{equation}
\label{asy-areaf}
\lim_{R_i \to +\infty} \!\! \int\limits_{\partial B(O, R_i)} \!\!\!\! u^{\frac{n-1}{n-2}} d\sigma \,\, = \,\, \left\vert\Sf^{n-1}\right\vert \, \AVR \, \left( \frac{\capa(\Omega)}{\AVR}\right)^{\!\frac{n-1}{n-2}}
\end{equation}
\begin{proof}
It follows from Lemma \ref{asympt_u} and  \eqref{avr1}.
\end{proof}
\end{corollary}

Building on the above formula~\eqref{limu} for the asymptotics of $u$, we now derive integral asymptotics for ${\D u}$. They are achieved through an adaptation of the methods used in~\cite[Section 4]{Che-Cold}. Similar estimates have been widely considered in literature (see e.g. \cite{Colding_1} \cite{colding3}, \cite{colding_harmonic_poly}, \cite{colding_ricciflat}). We are providing here a complete and self-contained proof. 

In the computations below, we are going to consider first and second derivatives of the distance function $x \mapsto d(O, x)$. This can be readily justified by approximating the distance function by convolution, by performing the computations below for the approximating sequence and finally passing to the limit. A scheme like this, often implicit in  modern literature, has been carried out in details for example in \cite{li_tam_wang}. However, in order to keep the presentation as transparent as possible, we are going to work directly with the distance function.    

\begin{proposition}[Integral Asymptotics for $\D u$]
Let $(M, g)$ be a complete noncompact  Riemannian manifold with $\ric \geq 0$ and Euclidean volume growth, and let $u$ be a solution to problem \eqref{pb}. Then, for every $k > 1$, it holds
\begin{equation}
\label{asyintdu}
\lim_{R \to + \infty}  \,\,
\frac{R^{2n-2}}{\abs{A_{R, kR}}}
{\int\limits_{A_{R, kR}}\left\vert \,  \D u - \dfrac{\capa(\Omega)}{\AVR} \D r^{2-n} \,\right\vert^2\! \dd\mu}
\,\,=\,\, 0 \, ,
\end{equation}
where, for $R > 0$, we set $A_{R, kR} = B(O, kR) \setminus \overline{B(O, R)}$.
\end{proposition}
\begin{proof}
A simple integration by parts combined with the harmonicity of $u$ leads to the following identity
\begin{equation}
\label{parts}
\begin{split}
\small
{\int\limits_{A_{R, kR}}\left\vert \D u - \dfrac{\capa(\Omega)}{\AVR}r^{2-n}\right\vert^2 \!\!\dd \mu} \,\,= \,  &-\!\!\!\int\limits_{A_{R, kR}}\!\!\! \left(\frac{\capa(\Omega)}{\AVR} \Delta r^{2-n}\right)\left(u-\frac{\capa(\Omega)}{\AVR}r^{2-n}\right)  \dd\mu  \,\,\, + \\ 
&+ \!\! \!\int\limits_{\partial A_{R, kR}}\!\!\!\!\!\!\!\left(u-\frac{\capa(\Omega)}{\AVR}r^{2-n}\right)\left\langle \D u- \frac{\capa(\Omega)}{\AVR}\D r^{2-n},\nu\right\rangle \dd\sigma \,.
\end{split}
\end{equation}
Let us estimate separately the integrals on the right hand side of the above identity. 
Let $\epsilon > 0$. Then, by \eqref{limu},
\[
\left\vert \frac{u}{r^{2-n}} - \frac{\capa (\Omega)}{\AVR}\right\vert < \epsilon.
\]
for $r$ big enough.
We have, for $R$ big enough
\begin{equation}
\label{estimate1}
\begin{split}
\left\vert  \,\, \int\limits_{A_{R, kR}}\!\!\!\!\!\!\left(\frac{\capa(\Omega)}{\AVR} \Delta r^{2-n}\right)\left(u-\frac{\capa(\Omega)}{\AVR}r^{2-n}\right) \dd\mu \, \right\vert & \, \leq \!\!\! \int\limits_{A_{R, kR}}\! \!\left\vert\frac{\capa(\Omega)}{\AVR} \, \Delta r^{2-n}\right\vert \,\, \left\vert \frac{u}{r^{2-n}} - \frac{\capa(\Omega)}{\AVR}\right\vert r^{2-n} \dd \mu \\
&\leq \, \epsilon \, R^{2-n} \!\!\! \int\limits_{A_{R, kR}}\left\vert \frac{\capa(\Omega)}{\AVR} \,  \Delta r^{2-n} \right\vert \dd\mu \\
&= \,\epsilon \, R^{2-n}\int\limits_{A_{R, kR}}\frac{\capa(\Omega)}{\AVR} \, \Delta r^{2-n} \dd\mu \,, \end{split}
\end{equation}
where in the last identity we used $\Delta r^{2-n} \geq 0$
in the sense of distributions, as already shown along the proof of Lemma~\ref{laplace_comp}. 
Integrating by parts $\Delta r^{2-n}$ we obtain
\[
\int\limits_{A_{R, kR}} \!\!\! \Delta r^{2-n} \dd \mu  \,\, = \,\, (2-n) \left[(kR)^{1-n}\abs{\{r = kR\}} - R^{1-n}\abs{\{r = R\}}\right] \,.
\]
In particular, by the assumption on the Euclidean volume growth, the above quantity is uniformly bounded in $R$. We have thus proved that the first summand on the right hand side of \eqref{parts}, for $R$ large enough, is bounded as follows
\begin{equation}
\label{estimate2}
\left\vert \, \int\limits_{A_{R, kR}} \left(\frac{\capa(\Omega)}{\AVR} \Delta r^{2-n}\right)\left(u-\frac{\capa(\Omega)}{\AVR}r^{2-n}\right) \dd\mu \, \right\vert \,\, \leq \,\, C_1\epsilon R^{2-n} \, .
\end{equation}
Let us turn our attention to the second integral in the right hand side of \eqref{parts}. Proceeding as in~\eqref{estimate1}, we have
\begin{equation}
\label{secondoadd}
\begin{split}
\int\limits_{\partial A_{R, kR}}\bigg\vert\left(u-\frac{\capa(\Omega)}{\AVR}r^{2-n}\right) \bigg\langle \D u \, - \,  &\frac{\capa(\Omega)}{\AVR}\D r^{2-n}, \nu\bigg\rangle \bigg\vert \dd\sigma   \, \leq \\
 & \leq \,\, \epsilon R^{2-n} \!\!\int\limits_{\partial A_{R, kR}} \!\!\!\!\left[ \, \abs{\D u} + (n-2)\frac{\capa(\Omega)}{\AVR} r^{1-n} \, \right] \dd\sigma \, .
\end{split}
\end{equation}
Recall now that in Proposition \ref{bound_eucl} we proved that
\[
\abs{\D u} \,\, \leq \,\, C_3\,  r^{1-n} \, ,
\]
for some positive constant $C_3$ independent of $r$. Thus, by the Euclidean volume growth, the integral on the right hand side of~\eqref{secondoadd} is uniformly bounded in $R$ and so
\begin{equation}
\label{estimate3}
\left\vert \,\,  \int\limits_{\partial A_{R, kR}}\left(u-\frac{\capa(\Omega)}{\AVR}r^{2-n}\right)\left\langle \D u- \frac{\capa(\Omega)}{\AVR}\D r^{2-n},\nu\right\rangle \dd\sigma \, \right\vert \,\, \leq \,\, C_3 \,\epsilon \, R^{2-n}
\end{equation}
for some $C_3$ independent of $R$. 
Finally, by \eqref{parts}, \eqref{estimate2} and \eqref{estimate3}, we obtain, for $R$ big enough, the estimate
\[ 
\frac{1}{{\abs{A_{R, kR}}}}
{\int\limits_{A_{R, kR}}\!\!\left\vert \D u - \dfrac{\capa(\Omega)}{\AVR}\D r^{2-n}\right\vert^2 \!\dd \mu} \,\, \leq \,\, C_4 \,\epsilon \, \frac{R^{2-n}}{\abs{A_{R, kR}}} \,\, \leq \,\, C_5 \, \epsilon \, R^{2-2n} ,
\]
for some positive constants $C_4$ and $C_5$ independent of $R$. In the last inequality we 
used the Euclidean volume growth assumption. Our claim~\eqref{asyintdu} is thus proved.
\end{proof}

From the above Proposition we easily deduce the integral asymptotic behaviour of $\abs{\D u}$ on geodesic spheres. 
\begin{corollary}
\label{lim1l}
Let $(M, g)$ be a complete noncompact Riemannian manifold with $\ric \geq 0$ and Euclidean volume growth, and let $u$ be a solution to problem \eqref{pb}. Then, there exists a sequence of positive real numbers $R_k$ with $R_k \to +\infty$ such that
\begin{equation}
\label{lim1}
\lim\limits_{R_k \to +\infty} \!\!\!{\int\limits_{\{r=R_k\}}\!\!\!\bigabs{\abs{\D u} - (n-2) \frac{\capa(\Omega)}{\AVR} r^{1-n}} \dd\sigma} \,\, = \,\, 0 \,.
\end{equation}
\end{corollary}

\begin{proof}
Let us first observe that, by means of H\"older inequality, we can deduce  from the $L^2$ asymptotics \eqref{asyintdu} an  analogous $L^1$ behaviour. Namely, for any $\epsilon > 0$ we have
\begin{equation}
\label{l1}
\ddfrac{\int_{A_{R, kR}}\left\vert \D u - \dfrac{\capa(\Omega)}{\AVR}\D r^{2-n}\right\vert \dd\mu}{R^{1-n}\abs{A_{R, kR}}} \,\,\,\, \leq \,\,\, \left(\ddfrac{\int_{A_{R, kR}}\left\vert \D u - \dfrac{\capa(\Omega)}{\AVR}\D r^{2-n}\right\vert^2 \dd\mu}{R^{2-2n}\abs{A_{R, kR}}}\right)^{\!\!{1}/{2}}  \!\!\!\!\leq \, \,\,  \epsilon
\end{equation} 
for any $R$ large enough. By the Coarea Formula, the above $L^1$ estimate gives, for $R$ large enough,
\[
\ddfrac{\int\limits_R^{kR}  \dd s \!\!\int\limits_{\{r=s\}} \!\left|  {\D u - \dfrac{\capa(\Omega)}{\AVR}\D r^{2-n}} \right| \dd\sigma}{R^{1-n} \abs{A_{R, kR}}} \,\, \leq \, \epsilon.
\]
Thus, by the Mean Value Theorem, there exists $R_k \in (R, kR)$ such that
\[
\ddfrac{\int\limits_{\{r=R_k\}} \!\!\left| {\D u - \dfrac{\capa(\Omega)}{\AVR}\D r^{2-n}} \right| \dd\sigma}{R^{-n} \abs{A_{R, kR}}} \,\, \leq C \, \epsilon \, ,
\]
for some constant $C$ independent of $R$.
The Euclidean volume growth of the annulus $A_{R, kR}$ as $R$ increases then implies the existence of a sequence ${R_k} \to +\infty$ as in the statement.
\end{proof}

\begin{remark} 
\label{no-hope}
The integral asymptotic for $\abs{\D u}$ given by Corollary \ref{lim1l}
cannot, in general, be improved to a pointwise asymptotic expansion at infinity on a manifold with nonnegative Ricci curvature and Euclidean volume growth. Indeed, the validity of such a formula would imply $\abs{\D u} \neq 0$ outside some big ball $B(O, R)$, and, in turn, $M \setminus B(O, R)$ would be diffeomorphic to $\partial B(O, R) \times [R, +\infty)$. This would imply that $M$ has finite topological type. However, Menguy provided in \cite{menguy} examples of manifolds of dimension $n \geq 4$  with $\ric \geq 0$ and $\AVR > 0$ with infinite topological type.
\end{remark}
\subsection{Final steps of the proof.} Recalling that 
\begin{equation*}
U_{\beta}(t) \, = \,  t^{-\beta \left(\frac{n-1}{n-2}\right)} \int\limits_{\{u=t\}} \abs{\D u}^{1 + \beta} \dd \sigma \, ,
\end{equation*}
it is easy to realise that we have now at hand all the elements to compute the limit as $t \to 0^+$.

\begin{proposition}
\label{prop:limite_u}
Let $(M, g)$ be a complete noncompact Riemannian manifold with $\ric \geq 0$ and Euclidean volume growth. Then, for every $\beta \geq 0$, we have that
\begin{equation}
\label{limit}
\lim_{t \to 0^+} U_{\beta} (t) \,\, = \,\, \capa(\Omega)^{1 - \frac{\beta}{n-2}}\, \AVR^{\frac{\beta}{n-2}}\,(n-2)^{\beta + 1}\, \abs{\Sf^{n-1}} \, .
\end{equation}

\end{proposition}

\begin{proof}
We multiply and divide inside the integral in \eqref{lim1} by $u^{(n-1)/(n-2)}$. Using the asymptotics of $u$ obtained in Lemma~\ref{asympt_u}, we obtain 
\begin{equation}
\label{limtrasf}
\lim_{R_i \to +\infty} {\int\limits_{\{r=R_i\}}\left|\left\vert{\frac{\D u}{u^{\frac{n-1}{n-2}}}}\right\vert - (n-2)\left(\frac{\AVR}{\capo}\right)^{\frac{1}{n-2}} \right| } \, u^{\frac{n-1}{n-2}}\dd\sigma \,\, = \,\, 0 \, .
\end{equation}
Let us now recall the following basic interpolation inequality
\begin{equation}
\label{inter}
\norm{f}_{L^p(X)} \,\, \leq \,\, \norm{f}_{L^1(X)}^{\theta}\norm{f}_{L^q(X)}^{1-\theta},
\end{equation}
holding for any $f \in L^1(X) \cap L^q(X)$, where $X$ is a measure space and the numbers $p$ and $q$ satisfy $1 < p < q < +\infty$ and $1/p =\theta + (1-\theta)/q$.
We apply such an estimate with 
\begin{equation*}
f \,= \, \left|\, \left\vert{\frac{\D u}{u^{\frac{n-1}{n-2}}}}\right\vert - (n-2)\left(\frac{\AVR}{\capo}\right)^{\frac{1}{n-2}} \, \right|
\end{equation*}
$p= 1 + \beta$, $q > 1 + \beta$, and with respect to the measure $u^{(n-1)/(n-2)} \dd \sigma$. We get
\begin{align}
\label{interpophi}
\left(\,\,{\int\limits_{\{r=R_i\}}\Bigg|\left\vert{\frac{\D u}{u^{\frac{n-1}{n-2}}}}\right\vert - (n-2)\left(\frac{\AVR}{\capo}\right)^{\frac{1}{n-2}} \Bigg|^{1+\beta} } u^{\frac{n-1}{n-2}}\dd\sigma \right)^{\!\!{1}/{(1 + \beta)}}\!\!\!\!\!\!\! \leq \hspace{5cm} \nonumber \\ 
\leq \,\,\, \left(\,\,{\int\limits_{\{r=R_i\}}\left| \, \left\vert{\frac{\D u}{u^{\frac{n-1}{n-2}}}}\right\vert - (n-2)\left(\frac{\AVR}{\capo}\right)^{\frac{1}{n-2}} \right|}u^{\frac{n-1}{n-2}}\dd\sigma  \right)^{\!\!\theta} \times\hspace{2cm}\\
 \times 
\left(\,\, {\int\limits_{\{r=R_i\}}\left| \, \left\vert{\frac{\D u}{u^{\frac{n-1}{n-2}}}}\right\vert - (n-2)\left(\frac{\AVR}{\capo}\right)^{\frac{1}{n-2}} \right|^{q} }u^{\frac{n-1}{n-2}}\dd\sigma\right)^{\!{(1-\theta)}/{q}}.\hspace{1cm}
\end{align}
Due to both the uniform bounds on $\abs{\D u}/u^{(n-1)/(n-2)} = \abs{\nabla \phi}_{\tilde{g}}$ given in \eqref{bound-nablaphi} and the bound
on $\int_{\{r =R_i\}} u^{\frac{n-1}{n-2}}d\sigma$, that follows from Corollary \ref{asy-area}, it is easy to see that the second integral on the right hand side of~\eqref{interpophi} is bounded in $R_i$. Thus, by \eqref{interpophi} and \eqref{lim1}, we deduce that
\begin{equation}
\lim_{R_i \to \infty} {\int\limits_{\{r=R_i\}}\left|\, \left\vert{\frac{\D u}{u^{\frac{n-1}{n-2}}}}\right\vert - (n-2)\left(\frac{\AVR}{\capo}\right)^{\frac{1}{n-2}} \right|^{1+\beta} } u^{\frac{n-1}{n-2}}\dd\sigma \,\, = \,\, 0 \, .
\end{equation}
The above limit in particular implies that
\[
\lim_{R_i \to +\infty} \int\limits_{\{r = R_i\}} \left\vert{\frac{\D u}{u^{\frac{n-1}{n-2}}}}\right\vert^{1 + \beta} u^{\frac{n-1}{n-2}} \dd \sigma \,\, = \,\, (n-2)^{1 + \beta} \left( \frac{\AVR}{\capo}\right)^{\frac{1+ \beta}{n-2}}\lim_{R_i \to \infty} \int\limits_{\{r= R_i\}} \!\!\!u^{\frac{n-1}{n-2}} \dd \sigma \, ,
\]
that, combined with \eqref{asy-areaf}, gives
\[
\lim_{R_i \to +\infty} \int\limits_{\{r = R_i\}} \left\vert{\frac{\D u}{u^{\frac{n-1}{n-2}}}}\right\vert^{1 + \beta} u^{\frac{n-1}{n-2}} \dd \sigma \,\, = \,\, (n-2)^{1 + \beta}\, \text{Cap}(\Omega)^{1 - \frac{\beta}{n-2}}\, \AVR^{\frac{\beta}{n-2}}\, \abs{\Sf^{n-1}} \, .
\]
Moreover, by the asymptotic behavior of $u$, we have
\[
\lim_{R_i \to +\infty} \int\limits_{\{r = R_i\}} \left\vert{\frac{\D u}{u^{\frac{n-1}{n-2}}}}\right\vert^{1 + \beta} u^{\frac{n-1}{n-2}} \dd \sigma = \lim_{t_i \to 0^+} \int\limits_{\{u = t_i\}} \bigabs{\frac{\D u}{u^{\frac{n-1}{n-2}}}}^{1 + \beta} u^{\frac{n-1}{n-2}} \dd \sigma
\,= \lim_{t_i \to 0^+} U_{\beta}(t_i) \, ,
\] 
and we have thus proved our claim for some sequence $t_i \to 0^+$.
However, by the boundedness and monotonicity of $U_\beta$ the whole limit as $t \to 0^+$ exists, and it coincides with the just computed one.
\end{proof}
Let us briefly discuss the case of sub-Euclidean volume growth in the following remark.
\begin{remark}
If $(M, g)$ has sub-Euclidean volume growth, that is,
if
\[
\lim_{r \to +\infty} \frac{\abs{B(p, r)}}{r^n}\,\, = \,\, 0
\]
for any $p \in M$, it is easy to realize that $\lim_{t \to 0^+} U_\beta = 0$ for any $\beta \geq 0$. Indeed, by \eqref{bound_u_g}, \eqref{liyau} and \eqref{yau1}, one easily obtains
\[
\frac{\abs{\D u}}{u^{\frac{n-1}{n-2}}}\,\, \leq C \,\, \left[ \frac{\abs{B(p, r)}}{r^n} \right]^{\frac{1}{n-2}}
\]
outside some ball containing $\Omega$ for some $C = C(M, \Omega)$. This implies that 
\[
\lim_{t \to 0^+} U_\beta (t)\,= \lim_{t\to 0^+} \int\limits_{\{u=t\}}\left\vert\frac{\D u}{u^{\frac{n-1}{n-2}}}\right\vert^\beta \!\!\abs{\D u} \dd \sigma \,\, = \,\, 0, 
\]
where we used the constancy of $t \mapsto \int_{\{u=t\}} \abs{\D u} \dd\sigma$. This computation clearly shows that $U_\beta$ cannot be employed to deduce a Willmore inequality on manifolds with sub-Euclidean volume growth, and it supports the perception that the infimum of the Willmore-type functional is zero on these manifolds. This actually happens for example on noncompact Riemannian manifolds with a metric which is asymptotic to a warped product metric of the following type
\[
g \,\, = \,\, d\rho \otimes d\rho + C \rho^{2\alpha }g_{\mid \Sigma} \, ,
\]
where $\Sigma$ is a compact hypersurface, $C > 0$ and $0<\alpha <1$. Indeed, this is readily checked by computing the Willmore-type functional on large level sets $\{\rho = r\}$.
On the other hand, the behavior of the infimum of the Willmore-type functional is completely understood in the 
three-dimensional case, in virtue of Theorem~\ref{cisoth}. 
\end{remark}

\begin{proof}
[Proof of Theorem \ref{willth}]
With Theorem \ref{thm:M/R} and Proposition \ref{prop:limite_u} at hand, Theorem \ref{willth} follows exactly as in the Euclidean proof recalled in the Introduction. Precisely, let $\beta = n-2$. Then, \eqref{limit} reads
\[
\lim_{t \to 0^+}U_{n-2} (t)\,\, = \,\,\AVR \,(n-2)^{n-1}\, \abs{\Sf^{n-1}}.
\]
Moreover, by the nonnegativity of expression \eqref{derivata_di_U} proved in Theorem \ref{thm:M/R}, and the H\"older inequality,
\[
\frac{(n-1)}{(n-2)}\int_{\partial \Omega} \abs{\D u}^{n-1}\dd \sigma \,\, \leq \,\, \int_{\partial \Omega} \abs{\D u}^{n-2} \,\HH \dd\sigma \,\, \leq \,\, \left(\int_{\partial \Omega} \abs{\D u}^{n-1}\dd\sigma\right)^{\!\!(n-2)/(n-1)} \left(\int_{\partial \Omega} \HH^{n-1} \dd\sigma\right)^{\!\!1/(n-1)} \!\!\!\!\!\!\!\!\!\!\!\!,
\]
that gives
\[
\int_{\partial \Omega} \left\vert{{{\D u}}}\right\vert^{n-1}\dd \sigma \,\, \leq \,\, (n-2)^{n-1} \int_{\partial \Omega}\left\vert\frac{\HH}{n-1}\right\vert^{n-1} \!\!\!\!\!\dd \sigma.
\]
Finally, one has
\[
\begin{split}
\AVR \, (n-2)^{n-1}\, \abs{\Sf^{n-1}} \,= \lim_{t \to 0^+}U_{n-2}(t) \leq U_{n-2}(1)&= \int_{\partial \Omega} \abs{Du}^{n-1} \dd\sigma \, \leq \\ 
&\leq (n-2)^{n-1} \int_{\partial \Omega}\left\vert\frac{\HH}{n-1}\right\vert^{n-1} \!\!\!\!\!\dd \sigma \, .
\end{split}
\] 
This completes the proof of the Willmore-type inequality. The rigidity statement when equality is attained follows straightforwardly from the rigidity part of Theorem \ref{thm:M/R}.
\end{proof}
\subsection{Application to ALE manifolds}
We can improve our Willmore-type inequality if $(M, g)$ satisfies a \emph{quadratic curvature decay} condition, showing that, in this case, the lower bound $\AVR\big\vert\abs{\Sf^{n-1}}$ on the Willmore functional is actually an infimum. Let us recall the following well known definition.

\begin{definition}[Quadratic curvature decay]
A complete noncompact Riemannian manifold $(M, g)$ has \emph{quadratic curvature decay} if there exists a point $p \in M$ and a constant $C = C(M, p)$ such that 
\[
\left\vert\text{\emph{Riem}}\right\vert \!(q)\,\, \leq \,\, C \, d(p, q)^2,
\]
where by \emph{Riem} we denote the Riemann curvature tensor of $(M, g)$.
\end{definition}  
When this assumption is added on a  Riemannian manifold with $\ric \geq 0$ and Euclidean volume growth $(M, g)$, \cite[Proposition 4.1]{colding3} gives the following asymptotic behaviour of the gradient and the Hessian of the minimal Green's function $G$.
\begin{theorem}
\label{pointasy}
Let $(M, g)$ be a complete noncompact Riemannian manifold with $\ric \geq 0$, Euclidean volume growth and quadratic curvature decay, and let $G$ be its minimal Green's function. Then
\begin{gather}
\lim_{d(p, x) \to +\infty}\frac{\abs{\D_x G}(p,x)}{ d^{1-n}(p,x)} \, = \, \frac{(n-2)}{\AVR} \label{DG-pointasy} \\ 
\lim_{d(p, x) \to +\infty} \left\vert \D\D_x \left(G^{2/(2-n)}\right)(p, x) - 2\left(\frac{1}{\AVR}\right)^{\frac{2}{2-n}} g(x)\right\vert \, = \, 0 \, . \label{hessG-pointasy} 
\end{gather}
\end{theorem}
Observe that arguing as in Remark \ref{no-hope}, one realizes that \eqref{DG-pointasy} actually implies that Riemannian manifolds satisfying the assumptions of the above Theorem have finite topological type. 

Theorem \ref{pointasy} enables us to prove that the Willmore-type functional of large level sets of the Green's function approach $\AVR \abs{\Sf^{n-1}}$. This fact, combined with our Willmore-type inequality \eqref{will}, easily yields the following refinement.
\begin{theorem}
\label{improve}
Let $(M, g)$ be a complete noncompact Riemannian manifold with $\ric \geq 0$, Euclidean volume growth and quadratic curvature decay. Then,
\begin{equation}
\label{inf}
\inf\left\{ \left.\,\,\int\limits_{\partial\Omega}\left\vert{\frac\HH{n-1}}\right\vert^{n-1}\!\!\!\!\dd
\sigma\,\, \right| \,\, \Omega \subset M\,\, \text{\emph{bounded and smooth}}\right\} \, = \, 
\AVR \, \abs{\Sf^{n-1}} \, .
\end{equation}
Moreover, the infimum is attained at some bounded and smooth $\Omega \subset M$ if and only if $(M \setminus \Omega, g)$ is isometric to 
\begin{equation}
\label{metric_cone1}
\Big(\,\big[r_0, +\infty) \times\pa\Omega
\,,\, 
\dd r\otimes\dd r +(r/r_0)^2 g_{\pa\Omega}
\Big),
\qquad
\mbox{with}
\quad
r_0
\,=\,
\bigg(\frac{|\pa\Omega|}{\AVR|\Sf^{n-1}|}\bigg)^{\frac1{n-1}}.
\end{equation}
\end{theorem}
\begin{proof}
Let $p \in M$ be fixed, and let $G$ be the minimal Green's function of $(M, g)$. Let us denote again by $G$ the function $q \mapsto G(p, q)$. In light of Theorem \ref{willth}, it suffices to prove that 
\begin{equation}
\label{claim}
\lim_{t \to +\infty}\!\int\limits_{\{G^{2/(2-n)}=t\}}\bigabs{\frac\HH{n-1}}^{n-1}\dd
\sigma \, = \, \AVR \abs{\Sf^{n-1}} \, .
\end{equation}
To see this, consider, at a point $x$, orthonormal vectors $\{e_1, \dots, e_{n-1}\}$ tangent to the level set $\{G = G(x)\}$. Then, letting $r(x)=d(p, x)$, we have, by \eqref{hessG-pointasy}, 
\begin{equation}
\label{asy-tr}
\lim_{r(x) \to \infty} \sum_{i=1}^{n-1} \D\D \left(G^{\frac{2}{2-n}}\right)(e_i, e_i) (x) \, = \, 2(n-1) \left(\frac{1}{\AVR}\right)^{\frac{1}{2-n}}.
\end{equation}
The mean curvature of the level sets of $G^{2/(2-n)}$ is clearly computed
as
\[
\HH_{G^{2/(2-n)}} \, = \, \frac{\sum_{i=1}^{n-1} \D\D \left(G^{\frac{2}{2-n}}\right)(e_i, e_i)}{\abs{\D G^{2/(2-n)}}} \, ,
\]
that, combined with \eqref{asy-tr}, \eqref{DG-pointasy} and \eqref{asyg} gives
\[
\lim_{r(x) \to \infty}r(x)\,{\HH_{G^{2/(2-n)}}(x)}\,=\, (n-1) \, .
\]
Combining it again with \eqref{asyg} and Euclidean volume growth gives \eqref{claim}, in fact completing the proof.
\end{proof}

We now particularize Theorem \ref{improve} to \emph{Asymptotically Locally Euclidean} (ALE) manifolds, proving Corollary \ref{cor:ALE}. We adopt the following definition, that is a sort of extension of the one considered in the celebrated \cite{bando-kasue-nakajima} -- where striking relations between curvature decay conditions and behaviour at infinity of manifolds are drawn -- and sensibly weaker than the one used by Joyce in the classical reference \cite{Joyce_book}.  
\begin{definition}[ALE manifolds]
\label{def:ALE} 
We say that a complete noncompact Riemannian manifold $(M, g)$ is ALE (of order $\tau$)  if there exist a compact set $K \subset M$, a ball $B \subset \R^n$, a diffemorphism $\Psi: M\setminus K \to \R^n\setminus B$, a subgroup $\Gamma < SO(n)$ acting freely on $\R^n\setminus B$ and a number $\tau > 0$ such that
\begin{gather}
{(\Psi^{-1}\circ \pi)^*g}(z) \, = \, g_{\R^n} + O(\abs{z})^{-\tau} \label{ale1} \\
\bigabs{\partial_{i}[(\Psi^{-1}\circ \pi)^*g]}(z) \, = \,  O(\abs{z})^{-\tau - 1} \label{ale2}  \\
\bigabs{\partial_{i}\partial_j[(\Psi^{-1}\circ \pi)^*g]}(z) \, = \,  O(\abs{z})^{-\tau - 2} \label{ale3},
\end{gather}
where $\pi$ is the natural projection $\R^n \to \R^n/\Gamma$, $z \in \R^n \setminus B$ and $i, j = 1, \dots, n$.
\end{definition}

\begin{proof}
[Proof of Corollary \ref{cor:ALE}]

Conditions~\eqref{ale1},~\eqref{ale2},~\eqref{ale3} readily imply that ALE manifolds have Euclidean volume growth and quadratic curvature decay. Moreover, condition \eqref{ale1} and a direct computation give that 
\begin{equation}
\label{chain-ale}
\AVR \,= \, \frac{\abs{\Sf^{n-1}/\Gamma}}{\abs{\Sf^{n-1}}} \, = \, \frac{1}{\text{card}\, \Gamma} \, .
\end{equation}
The characterization~\eqref{will-ale} then follows immediately from~\eqref{inf}. 

Assume now that the infimum of the Willmore functional is attained at some $\Omega \subset M$. Then, by the rigidity part in Theorem \ref{improve}, $M\setminus \Omega$ is isometric to a truncated cone over $\partial \Omega$. However, by \eqref{ale1}, $(M, g)$ is also $C^0$-close at infinity to a metric cone with link $\Sf^{n-1}/\Gamma$. Since the cross sections of a cone are all homothetic to each other, $\partial \Omega$ is homothetic to $\Sf^{n-1}/\Gamma$, that is, they are diffeomorphic and $g_{\partial \Omega} = \lambda^2 g_{\Sf^{n-1}/ \Gamma}$ for some positive constant $\lambda$. This fact, together with \eqref{chain-ale} in the rigidity part of Theorem \ref{willth} imply that 
\begin{equation}
\label{metric_cone-ale}
\Big(\,\big[r_0, +\infty) \times\pa\Omega
\,,\, 
\dd r\otimes\!\dd r +(\lambda r/r_0)^2 g_{\Sf^{n-1}/\Gamma}
\Big),
\qquad
\mbox{with}
\quad
r_0
\,=\,
\bigg(\frac{|\pa\Omega|}{|\Sf^{n-1}/\Gamma|}\bigg)^{\frac1{n-1}}.
\end{equation}
In particular, one has
\[
\abs{\partial B(O, R)} \, = \, \left(\frac{R\lambda}{r_0}\right)^{\!n-1} \abs{\Sf^{n-1}/\Gamma}.
\]
Combining this with~\eqref{chain-ale} we conclude that $\lambda = r_0$, proving the isometry with~\eqref{cone-ale} and completing the proof.
\end{proof}

\section{Proof of the Enhanced Kasue's Theorem}
The proof of Theorem~\ref{mono-par} is completely analogous to the proof of Theorem~\ref{thm:main_conf}, for this reason most of the details can be easily adapted from the previous section and will be left to the interested reader.
For $\beta \geq 0$, we recall the definition of the function $\Psi_\beta : [0, +\infty) \to \R$ satisfying 
\begin{equation}
\label{Psi1}
\Psi_\beta (s) \,\, = \!\!\!\!\int\limits_{\{\psi = s\}} \!\!\abs{\D \psi}^{\beta + 1} \dd \sigma \, , 
\end{equation}
where $\psi$ is a solution to problem \eqref{prob-ex-par} for some bounded $\Omega \subset M$ with smooth boundary and $\beta \geq 0$. Combining the uniform bound on $\abs{\D \psi}$ given in Theorem \ref{existence-par} with the constancy of $\Psi_0$, we obtain, as in Corollary \ref{cor-bound}, that $\Psi_\beta$ is uniformly bounded in $s$ for any $\beta \geq 0$. We record this fact for future reference.
\begin{lemma}
\label{bound-lemma-par}
Let $(M, g)$ be a parabolic manifold with $\ric \geq 0$, and let $\psi$ be a solution to problem \eqref{prob-ex-par}. Then, the function $\Psi_\beta$ defined in \eqref{Psi} is uniformly bounded for every $\beta \geq 0$.
\end{lemma}

\begin{proof}[Proof of Theorem \ref{mono-par}]
As for Theorem \ref{thm:main_conf}, we only prove Theorem~\ref{mono-par} for $\beta \geq 1$. To include the remaining cases it is sufficient to follow the strategy suggested in Remark~\ref{optimal}.
Let us begin with the computation of a Bochner-type identity for $\abs{\D\psi}^\beta$, in the same spirit as in Lemma~\ref{boch}
\begin{equation}
\label{bochner-par}
\Delta \abs{\D\psi}^\beta \,\,= \,\, \beta \abs{\D \psi}^{\beta -2} \left[\abs{\D\D\psi}^2 + (\beta-2)\big\vert\D\abs{\D\psi}\big\vert^2 + \ric(\D\psi, \D\psi) \right].
\end{equation}
Observe that the right hand side of \eqref{bochner-par} is nonnegative if $\beta \geq (n-2)/(n-1)$, by means of the refined Kato's inequality \eqref{kato} for harmonic functions.
Applying the Divergence Theorem and the  co-area formula in the same fashion as in the proof of~\eqref{div+coarea}, we get
\begin{equation}
\label{coarea-par}
\Psi_\beta(s) - \Psi_\beta(s_0) \,\, = \,\, \int\limits_{s_0}^s \! \dd \tau \!\!\! \int\limits_{\{\psi = \tau\}} \!\! \left\langle \D \abs{\D \psi}^\beta, \frac{\D \psi}{\abs{\D \psi}}\right\rangle \dd\sigma \, .
\end{equation}
Exactly as in Lemma \ref{integral_id}, we can use the Divergence Theorem and \eqref{bochner-par} to obtain that for every $0 \leq \tau_0 < \tau$ it holds
\[
\begin{split}
\int\limits_{\{\psi = \tau\}} \!\! \left\langle \D \abs{\D \psi}^\beta, \frac{\D \psi}{\abs{\D \psi}}\right\rangle &\, \dd\sigma \,\,\, - \!\!\! \int\limits_{\{\psi = \tau_0\}}\!\!\! \left\langle \D \abs{\D \psi}^\beta, \frac{\D \psi}{\abs{\D \psi}}\right\rangle \dd\sigma  \,\,\, = \\ 
&\!\!\!= \,\,\, \beta\int\limits_{\{\tau_0 \leq \psi \leq \tau\}}  \!\!\!\!\abs{\D \psi}^{\beta -2} \left[\abs{\D\D\psi}^2 + (\beta-2)\big\vert\D\abs{\D\psi}\big\vert^2 + \ric(\D\psi, \D\psi) \right] \dd \mu \, . 
\end{split}
\]
Combining this integral identity with~\eqref{coarea-par}, we deduce, as in the proof of Theorem~\ref{mono1}, that $\Psi_\beta$ is differentiable, that
\[
\Psi_\beta'(s) \,\, = \int\limits_{\{\psi = s\}} \!\!\!\left\langle \D \abs{\D \psi}^\beta, \frac{\D \psi}{\abs{\D \psi}}\right\rangle \dd\sigma \,\, = \,\, -\beta \!\!\int\limits_{\{\psi = s \}}\!\!\!
|\D\psi|^\beta\,\HH \,\dd\sigma \, ,
\]
and that, for every $S \geq s \geq 0$, it holds
\begin{equation}
\label{mono1-par}
\Psi_\beta'(S) - \Psi_\beta'(s) \,\, = \,\, \beta \!\!\!\!\!\!
\int\limits_{\{s \leq \psi \leq  S\}} \!\!\!\!\!
{
|\D\psi|^{\beta-2} 
\Big(\ric(\D\psi,\D\psi)
+\big|\D\D \psi\big|^2
 +  \, (\beta-2) \, \big| \D |\D \psi |\big|^2\,\Big) 
}
\dd\mu \, \geq \, 0 \, .
\end{equation}
Since $\Psi_\beta$ is bounded by Lemma~\ref{bound-lemma-par}, we can argue as in the conclusion of the proof of Theorem~\ref{thm:main_conf} to pass to the limit as $S \to +\infty$ in \eqref{mono1-par} and obtain the monotonicity formula \eqref{eq:der_fip-par}. The rigidity part of the statement is obtained exactly as for that of Theorem \ref{thm:main_conf}. 
\end{proof}
The Enhanced Kasue's Theorem \ref{enkath} now follows at once.
\begin{proof}[Proof of Theorem \ref{enkath} and Corollary \ref{kasue}]
Assume first that $(M, g)$ is nonparabolic. Then, it is sufficient to use~\eqref{derivata_di_U} and~\eqref{eq:monot} at $t = 1$ to get
\[
\left(\frac{n-1}{n-2}\right)  U_\beta(0) \, + \,  \frac{1}{\beta} \, \frac{\dd U_\beta}{\dd t}(0)\,\, = \,\, \int\limits_{\partial \Omega} \HH \, \abs{\D u}^\beta \dd\sigma\,\, \leq \,\, \sup_{\partial \Omega} \HH_{\partial \Omega}  \int\limits_{\partial \Omega} \abs{\D u}^\beta \dd\sigma \, ,
\] 
that is \eqref{enka1}. If $(M, g)$ is parabolic, one can prove inequality~\eqref{enka2} in a completely analogous fashion. 
Since $U_\beta > 0$, it is easy to deduce from~\eqref{enka1} and~\eqref{enka2} that $\HH_{\partial \Omega}\leq 0$ on $\partial \Omega$ if and only if $(M, g)$ is parabolic and ${\dd\Psi}/{\dd s}_\beta (0) =0$. The rigidity statement in Theorem~\ref{mono-par} gives then Corollary~\ref{kasue}.
\end{proof}

Corollary \ref{kasue} can also be interpreted as a rigidity statement when the equality is attained in \eqref{will} if $\AVR = 0$. The following is then a direct  consequence of Theorem~\ref{willth} and Corollary~\ref{kasue}. 

\begin{corollary}
\label{general}
Let $(M, g)$ be a 
complete noncompact Riemannian manifold 
with $\ric\geq 0$.
If $\Omega\subset M$ is a bounded subset with smooth boundary, then
\begin{equation}
\label{will-general}
\int\limits_{\partial\Omega}\,\left\vert{\frac\HH{n-1}}\right\vert^{n-1}\!\!\!\!\!\!\dd
\sigma
\,\,\geq\,\,
\AVR\abs{\Sf^{n-1}} \, .
\end{equation}
If $\AVR > 0$, then equality  in \eqref{will-general} holds if and only if 
$(M\setminus\Omega,g)$ 
is isometric to 
\begin{equation}
\label{metric_cone-general}
\Big(\,\big[r_0, +\infty) \times\pa\Omega
\,,\, 
\dd r\otimes\!\dd r +(r/r_0)^2 g_{\pa\Omega}
\Big),
\qquad
\mbox{with}
\quad
r_0
\,=\,
\bigg(\frac{|\pa\Omega|}{\AVR \,|\Sf^{n-1}|}\bigg)^{\frac1{n-1}}.
\end{equation}
In particular, $\partial \Omega$ is a connected submanifold with constant mean curvature.
If $\AVR = 0$ , equality holds in \eqref{will-general} if and only if $(M \setminus \Omega, g)$ is isometric to a Riemannian product $\left([0, +\infty) \times \partial \Omega, \dd r\otimes \dd r + {g}_{\partial\Omega}\right)$. In particular, $\partial \Omega$ is a connected totally geodesic submanifold of $(M, g)$.

\end{corollary}

\section{The isoperimetric inequality for $3$-manifolds}
As already discussed in the Introduction, we show here how to use our Willmore inequality~\eqref{will} to improve a result stated by Huisken in~\cite{Hui_video}, in which
the infimum of the Willmore energy is characterized in terms of the infimum of the isoperimetric ratio on $3$-manifolds with nonnegative Ricci curvature. 
In concrete, we are going to prove  Theorem~\ref{cisoth}, whose statement is recalled hereafter for the reader's convenience.

\begin{theorem}
Let $(M, g)$ be a $3$-manifold  with $\ric \geq 0$. Then,
\begin{equation*}
\inf\,\ddfrac{\abs{\partial\Omega}^{3}}{36{\pi}\abs{\Omega}^2}\,=\,\inf\,\ddfrac{\int_{\partial \Omega}\!\!{\HH}^2\dd\sigma}{16{\pi}}\, = \,\AVR,
\end{equation*}
where the infima are taken over bounded and open subsets $\Omega \subset M$ with smooth boundary.
In particular, the following isoperimetric inequality holds
for any bounded and open $\Omega \subset M$ with smooth boundary
\begin{equation*}
\frac{\abs{\partial \Omega}^{3}}{\abs{\Omega}^2} \, \geq \,36{\pi} \,\AVR.
\end{equation*}
Moreover, equality is attained in \eqref{iso} if and only if $M= \R^3$ and $\Omega$ is a ball.
\end{theorem}

\begin{remark}
\label{cylindrical-ends}
Observe that Theorem \ref{cisoth} is obvious if $(M, g)$ has cylindrical ends, that is, if there exists a bounded subset $\Omega \subset M$ with smooth boundary such that $(M \setminus \Omega, g)$ is isometric to half a cylinder, as in the rigidity statement of Theorems \ref{thm:main_conf} and \ref{mono-par}. Indeed, since on such an end any cross-section is minimal, the infimum of the Willmore functional is clearly zero, and, since any cross-section has the same surface area, considering an increasing sequence of regions enclosed by cross-sections shows that the infimum of the isoperimetric ratio is zero too.
\end{remark}

\begin{remark}
We point out that in the $3$-dimensional case Theorem~\ref{cisoth} extends Theorem~\ref{improve} to any complete manifold with $\ric \geq 0$, with no restrictions on the volume growth and no curvature assumptions at infinity. Such enhancement is implicitly due to the fact that in dimension $n=3$ the topology of nonnegatively Ricci curved manifolds is completely understood (see~\cite{liu-structure}). 
\end{remark}

\subsection{Huisken's argument.} 
\label{sub:hui}
Let us briefly present Huisken's heuristic argument to deduce an isoperimetric inequality from Willmore's through the \emph{mean curvature flow}. 
We first recall that a sequence of orientable hypersurfaces $F_t(p): \Sigma \to M$ immersed in a Riemann manifold $(M, g)$, evolves through the Mean Curvature Flow if 
\[
\frac{\dd}{\dd t} F_t(p) \,\, = \,\, - \, \HH(p, t) \, \nu (p, t) \, ,
\]
where $\HH$ is the mean curvature of $\Sigma_t = F_t (\Sigma)$ and $\nu$ is its (exterior, in the case where $\Sigma_t$ is the boundary of a domain) unit normal. Accordingly, we say that $\{\Omega_t\}$ is a mean curvature flow if the boundaries are evolving through mean curvature flow in the sense explained above.
Let then $\Omega$ be an open bounded set with smooth boundary and let $\{\Omega_t\}$, with $t \in [0, T)$, be a smooth mean curvature flow starting from $\Omega$. Suppose, in addition, that
\begin{equation}
\label{lim-mcf}
\lim_{t \to T^-} \,\abs{\Omega_t} \, = \, 0.
\end{equation}
Consider, for some constant $C> 0$ to be defined later, the \emph{isoperimetric difference}
\begin{equation}
\label{iso-diff}
D(t) \,\, = \,\, \abs{\partial\Omega_t}^{3/2} - \,\,C \, \abs{\Omega_t} \, .
\end{equation}
Taking derivatives in $t$, and using standard formulas (see for example \cite[Theorem 3.2]{Huis_Pold}), one finds
\[
\frac{\dd}{\dd t} D(t) \,\, = \,\, - \, \frac{3}{2} \,  \abs{\partial \Omega_t}^{1/2} \!\!\int\limits_{\partial \Omega_t} \!\HH^{2} \dd\sigma \, +\,  C \!\!\int\limits_{\partial \Omega_t} \!\HH \dd\sigma \, , 
\]
that, through H\"older inequality, gives
\[
\frac{\dd}{\dd t} D(t) \, \leq \,  \left(\left\vert\partial \Omega_t\right\vert\int\limits_{\partial \Omega_t} \HH^2\dd\sigma\right)^{\!\!1/2} \, \left[\,  C \, - \,  \frac{3}{2}\left(\,\, \int\limits_{\partial\Omega_t}{\HH}^2 \dd\sigma\right)^{\!\!1/2}\right].
\]
Thus, if we choose $C$ such that
\begin{equation}
\label{willinf}
C \,\, \leq \,\, \frac{3}{2}\left(\,\, \int\limits_{\partial\Omega}{\HH}^2 \dd\sigma\right)^{\!\!1/2}
\end{equation}
for any bounded and smooth $\Omega \subset M$, $t \mapsto D(t)$ is nonincreasing. This implies that
\[
D(0) \,\, = \,\, \abs{\partial\Omega}^{3/2} - C\abs{\Omega} \,\,\geq \,\, \lim_{t\to T^-} D(t) \,\, \geq \,\, 0 \, ,
\]
where we have also used~\eqref{lim-mcf}. The above comparison in particular gives the (possibly non sharp) isoperimetric inequality
\[
\frac{\abs{\partial \Omega}^{3/2}}{\abs{\Omega}} \,\, \geq \,\, C \, .
\]
In~\cite{Hui_video}, the constant $C$ is chosen to be the infimum of the right hand side of~\eqref{willinf}, when $\Omega$ varies in the class of outward minimizing domains.

\subsection{Tools from the Mean Curvature Flow of mean-convex domains.}
We are first concerned with the accurate justification of the above computations. This will be accomplished with the help of a couple of important results due to Schulze and White, respectively. In the first part of our treatment we assume that the boundary $\partial \Omega$ of the bounded set $\Omega$ is smooth and mean-convex, that we understand as $\HH>0$. We will see later how to deal with the general cases.

Since the Mean Curvature Flow (MCF for short) is likely to develop singularities, one needs to consider an appropriate weak notion in order to state the following useful results. In particular, we consider the Weak Mean Curvature Flow in the sense defined in~\cite{evans-spruck}. A special case of the regularity theorem~\cite[Theorem 1.1]{white-size} gives
\begin{theorem}[White's Regularity Theorem]
\label{whiteth}
Let $(M, g)$ be a $3$-dimensional Riemannian manifold, let $\Omega\subset M$ be a bounded set with smooth mean-convex boundary and let $\{\Omega_t\}_{t \in [0, T)}$ be its Weak Mean Curvature Flow. 
Then, the boundary of $\Omega_t$ is smooth for almost every $t \in [0, T)$.
\end{theorem}
We point out that the maximal time $T$ might {\em a priori} be infinite on a general Riemannian manifold.
We are going to combine the above regularity result with the following special case of~\cite[Proposition 7.2]{schulze1}, that is a weak version of the monotonicity of the isoperimetric ratio. It can be checked, indeed, that the computations performed  to obtain such a result do not involve the geometry of the underlying manifold.
\begin{theorem}[Schulze]
\label{schulzeth}
Let $(M, g)$ be a $3$-dimensional Riemannian manifold, let $\Omega\subset M$ be a bounded set with smooth mean-convex boundary and let $\{\Omega_t\}_{t \in [0, T)}$ be its Weak Mean Curvature Flow.
Assume there exists a universal constant $C\geq 0$ such that
\begin{equation}
\label{condc}
C \,\, \leq \,\, \frac{3}{2}\left(\,\, \int\limits_{\partial\Omega_t}\!{\HH}^{2} \dd\sigma\right)^{\!\!1/2}
\end{equation}
for almost every $t \in [0, T)$. Then, the isoperimetric difference $t \mapsto D(t)$ defined as in~\eqref{iso-diff} using the constant $C$, is nonincreasing for every $t \in [0, T)$.
\end{theorem}
\begin{remark}
\label{surgery}
A different tool that might be used to deal with the singularities would be the theory of the~\emph{Mean Curvature Flow with surgery}, recently developed by Brendle and Huisken in~\cite{brendle-huisken1} and~\cite{brendle-huisken2}. On this regard, one should first to make clear whether the monotonicity of the isoperimetric difference survives the surgeries.
\end{remark}
The following theorem provides a complete description of the long time behaviour of the Weak MCF of a surface moving inside a $3$-dimensional Riemannian manifolds, and it can be readily deduced from~\cite[Theorem~11.1]{white-size}.
\begin{theorem}[Long time behaviour of MCF]
\label{white-large}
Let $(M, g)$ be a $3$-dimensional Riemannian manifold, let $\Omega\subset M$ be a bounded set with smooth mean-convex boundary and let $\{\Omega_t\}_{t \in [0, T)}$ be its Weak Mean Curvature Flow.
If $\abs{\Omega_t}$ and $\abs{\partial \Omega_t}$ do not vanish at finite time as $t \to T^{-}$, then $\Omega_t$ converges smoothly to a subset $K$, and the boundary of any connected component of $K$ is a stable minimal submanifold.
\end{theorem}
As a consequence, if $(M, g)$ contains no bounded subsets with minimal boundary, the Weak MCF of a bounded set with mean-convex boundary is going to vanish. In particular, combining Corollary~\ref{kasue} with Theorem~\ref{white-large}, one gets the following corollary.
\begin{corollary}
\label{large-mcf}
Let $(M, g)$ be a complete, noncompact, $3$-dimensional Riemannian manifold with $\ric \geq 0$ and no cylindrical ends, let $\Omega\subset M$ be a bounded set with smooth mean-convex boundary and let $\{\Omega_t\}_{t \in [0, T)}$ be its Weak Mean Curvature Flow.
Then, $T$ is finite and $\abs{\Omega_t}$ and $ \abs{\partial \Omega_t}$ tend to $0$  as $t\to T^-$. 
\end{corollary}

So far we have at hand all the ingredients that allow to completely justify the computations of Subsection~\ref{sub:hui} and in turns to prove the Isoperimetric Inequality for mean-convex domains.

\subsection{Proof of Theorem~\ref{cisoth}.}
In order to prove the isoperimetric inequality for any possibly non mean-convex domain $\Omega$, we are going to consider the minimizing hull $\Omega^*$ (see \cite[Section 1]{Hui_Ilm}). By \cite{sternberg-williams}, $\partial \Omega$ enjoys $C^{1, 1}$ regularity, and by the minimizing property, its weak mean curvature $\HH_{\partial \Omega^*}$ is nonnegative. We will actually actually flow  $\Omega^*$ by mean curvature. To this aim, we will invoke
\cite[Lemma 2.6]{huisken-ilm_higher}, where the authors show that $C^1$ bounded hypersurfaces  with nonnegative variational mean curvature can be approximated in $C^{1, \beta} \cap W^{2, p}$, for any $\beta \in (0, 1)$ and $p \in [1, \infty)$ by smooth submanifolds with strictly positive mean-curvature. Interestingly, the approximating sequence is built through an appropriate notion of mean curvature flow starting from such a $C^1$ hypersurface. Although presented in $\R^n$, the proof given in \cite{huisken-ilm_higher} can be easily adapted to go through the case of a general ambient Riemannian manifold, and we refer the reader to \cite[Lemma 4.4]{zhou-mean} for the details of this extension. Such a result has also been pointed out in \cite[Lemma 4.2]{wei_kottler} in the ambient setting of a Kottler-Schwarzschild
manifold, and used in the proof of \cite[Corollary 1.2]{schulze1}, where the ambient manifold was a Cartan-Hadamard $3$-manifold. For our aim, where we are interested in approximating $C^{1, 1}$ hypersurfaces, it actually suffices to argue as in \cite[Lemma 5.6]{Hui_Ilm}. We include here the precise and general statement.

\begin{lemma}[Huisken-Ilmanen's approximation lemma]
\label{approx-lemma}
Let $(M, g)$ be a Riemannian manifold, and let $F: \Sigma \hookrightarrow M$ be a $C^1$ closed immersed hypersurface. Assume $\Sigma$ has nonnegative weak mean curvature, that is, there exists a  nonnegative function $\HH$ defined almost everywhere on $\Sigma$ such that
\[
\int_{\Sigma} \dive_\Sigma X \dd\sigma \,\, = \,\, \int_\Sigma \HH \left\langle X, \nu \right\rangle \dd\sigma
\]
for any compactly supported vector field $X$ of $M$. Assume also that $\Sigma$ is not minimal, that is, there exists a subset $K \subset \Sigma$ of positive measure such that $\HH > 0$ on $K$. Then, $F$ is of class $C^{1, \beta} \cap W^{2, p}$ and there exists a family of smooth immersions $F(\cdot, \epsilon) : \Sigma \hookrightarrow M$, with $\epsilon \in (0, \epsilon_0]$  such that
\[
\frac{\dd}{\dd \epsilon} F(p, \epsilon)\,\, = - \HH_{\Sigma_\epsilon}(p, \epsilon)\nu (p, \epsilon)
\]
for any $\epsilon  \in (0, \epsilon_0]$, where $\HH_{\Sigma_\epsilon}$ is the mean curvature of $\Sigma_\epsilon$ and $\nu$ its outer unit normal, and
\[
\lim_{\epsilon \to  0^+} F(p, \epsilon) \, = \, F(p) 
\]
locally uniformly in $C^{1, \beta} \cap W^{2, p}$. Moreover, $\HH_{\Sigma_\epsilon} > 0$ for any $\epsilon \in (0, \epsilon_0]$.
\end{lemma}
\begin{remark}
Observe that if $\Sigma$ is a (non minimal) $C^2$ hypersurface with $\HH_\Sigma \geq 0$, then the approximation of $\Sigma$ by means of a family of smooth mean-convex hypersurfaces $\{\Sigma_\epsilon\}_{\ep >0}$ is a straightforward procedure.
Indeed, it is sufficient to run the MCF starting at $\Sigma$ for short time (see~\cite{mantegazza_libro} for an account about the classical existence theory), say until some time $\ep_0 >0$. This provides a family of hypersurfaces $\{\Sigma_\ep \}_{\ep \in (0, \ep_0]}$,
%
%
%
whose mean curvatures satisfy (see e.g. \cite[Theorem 3.2]{Huis_Pold}) the following reaction-diffusion equation,
\[
\frac{\partial}{\partial \epsilon} \HH \, = \, \Delta \HH + \HH\left( \abs{\rm h}^2 + \ric (\nu, \nu)\right),
\]
where $h$ is the second fundamental form of the evolving hypersurface and $\ric$ is the Ricci tensor of the ambient manifold. Then, a standard maximum principle for parabolic equations (see e.g. Theorem 7 and subsequent remarks in \cite{weinberger-protter}) shows that $\HH_{\Sigma_\epsilon} > 0$ for every $\epsilon \in (0, \epsilon_0]$, unless $H_\Sigma$ is constantly null. The latter case is excluded by the non minimality of $\Sigma$.
\end{remark}

We can finally prove Theorem~\ref{cisoth}.
\begin{proof}[Proof of Theorem~\ref{cisoth}]
Observe first, that we can suppose that $(M, g)$ has no cylindrical ends, otherwise there is nothing to prove, in light of Remark \ref{cylindrical-ends}. We argue as in~\cite[proof of Corollary 1.2]{schulze1}.
Let us first suppose that the boundary of $\partial \Omega$ is strictly mean-convex, that is, $\HH_{\partial \Omega} >0$. Let $\{\Omega_t\}_{t\in[0, T)}$ be a mean curvature flow starting from $\Omega$. Then, by Theorem \ref{whiteth}, for almost any $t\in [0, T)$ the boundary $\partial \Omega_t$ is a smooth submanifold.
Let $C$ be defined as 
\begin{equation}
\label{def:C_inf}
C \,= \, \inf\left\{\frac{3}{2}\left(\int_{\partial\Omega}{\HH}^2 \dd\sigma\right)^{\!\!1/2} \Bigg\vert \,\Omega \subset M \,\,\text{bounded set with smooth boundary}\right\}.
\end{equation}
Observe that $C$ is possibly zero, so far. However, Theorem~\ref{schulzeth} guarantees that,
with the above choice of $C$, the isoperimetric difference $t \mapsto D(t)$ defined in~\eqref{iso-diff} is nonincreasing for $t \in [0, T)$. Moreover, by Corollary~\ref{large-mcf} $D(t)$ tends to $0$ as $t \to T^-$. This implies the inequality
\begin{equation}
\label{iso-c}
\frac{\abs{\partial \Omega}^{3/2}}{\abs{\Omega}} \, \geq \, C 
\end{equation}
for any $\Omega$ with sooth mean-convex boundary. If this is not the case, take  the minimizing hull $\Omega^*$ of $\Omega$ (see \cite[Section 1]{Hui_Ilm} for details). By \cite{sternberg-williams} (compare also with the comprehensive \cite[Theorem 1.3]{Hui_Ilm}) $\partial \Omega^*$ is a $C^{1,1}$ hypersurface. Observe that, by the minimizing property, $\abs{\partial \Omega^*} \leq \abs{\partial \Omega}$, while trivially $\abs{\Omega^*} \geq \abs{\Omega}$. Hence, proving a lower bound on the isoperimetric ratio for $\Omega^*$ readily implies that the same lower bound holds for $\Omega$. Moreover, again by the minimizing property, we have that $\HH_{\partial \Omega^*}\geq 0$ (see also \cite[(1.15)]{Hui_Ilm}). Also notice that $\pa \Omega^*$ cannot be minimal, for otherwise $(M,g)$ will have a cylindrical end, in virtue of Corollary~\ref{kasue}. By Lemma \ref{approx-lemma}, we find a sequence of smooth hypersurfaces $\Sigma_\epsilon$ with $\HH_{\Sigma_\epsilon}>0$ approximating $\partial \Omega^*$ locally uniformly in $C^{1, \beta}$ for any $\beta \in (0, 1)$.  Arguing as above, we thus obtain the isoperimetric inequality
\[
\frac{\abs{\partial \Sigma_\epsilon}^{3/2}}{\abs{\Sigma_\epsilon}} \, \geq \, C,
\]
that, through letting $\epsilon \to 0^+$, gives the isoperimetric inequality for $\Omega^*$, and, in turn, for any bounded $\Omega$ with smooth boundary. Combining it with our Willmore inequality \eqref{will}, we get
\begin{equation}
\label{chain-iso}
\inf\ddfrac{\abs{\partial\Omega}^{3}}{36{\pi}\abs{\Omega}^2}\, \geq \,\inf\ddfrac{\left(\int_{\partial \Omega}{\HH}^2\dd\sigma\right)}{16{\pi}} \, \geq \, \AVR,
\end{equation}
where the infima are taken over any bounded $\Omega$ with smooth boundary. 
We now want to prove that the equality sign hold in both the above inequalities, as stated in~\eqref{ciso}. To do so, we fix a point $O \in M$ and we observe that by the Bishop-Gromov Theorem, we can find, for every $\delta > 0$, a radius $R_\delta$ such that
\[
\ddfrac{\abs{\partial B(O,{R_\delta})}^{3}}{36{\pi}\abs{B(O, {R_\delta})}^2} \, \leq \, \AVR + \delta.
\]
Observe that we can suppose $\partial B(O, R_\delta)$ to be smooth. Otherwise, it suffices to consider in place of $B(O, R_\delta)$ a smooth set whose perimeter and volume approximate $\abs{\partial B(O, R_\delta)}$ and $\abs{B(O, R_\delta)}$, respectively (this can be done by standard tools, see e.g.~\cite[Remark 13.2]{maggi}).
This proves that
\[
\inf\left\{\ddfrac{\abs{\partial\Omega}^{3}}{36{\pi}\abs{\Omega}^2}\,\,\,\Bigg\vert \,\Omega \subset M \,\,\text{bounded and smooth}\right\} \, \leq \, \AVR . 
\]
Combining the above inequality with~\eqref{chain-iso}, gives~\eqref{ciso}.

\smallskip

To prove the rigidity statement, we assume now that~\eqref{iso} holds with the equality sign for a smooth and bounded $\Omega \subset M$.
In virtue of~\eqref{def:C_inf} and of~\eqref{ciso},
we have that
\[
C
\,=\,
\sqrt{36\pi\AVR}\,,
\]  
Moreover, we can clearly suppose that $\AVR > 0$. By the minimizing property, $\Omega^*$ satisfies the same equality (recall that we actually proved the isoperimetric inequality for minimizing hulls). We claim that it also holds for the region enclosed by any approximating $\Sigma_\epsilon$ as above. 
Indeed, by Lemma \ref{approx-lemma}, this family is a smooth mean curvature flow, and then, by the monotonicity of the isoperimetric difference, for any fixed $\epsilon_1 \in (0, \epsilon_0]$ we have
\[
D(\epsilon) \, \geq \, D(\epsilon_1) \, \geq \, 0
\]
for every $\epsilon \in (0, \epsilon_1)$.
Since $\Sigma_\epsilon$ converges  to $\partial \Omega^*$ as $\epsilon \to 0^+$, and on $\Omega^*$ the isoperimetric difference is $0$, $D(\epsilon) \to 0^+$ as $\epsilon \to 0^+$, and thus $D(\epsilon_1)= 0$ as well. Since $\epsilon_1$ was arbitrarily chosen, the region enclosed by any of the $\Sigma_\epsilon$ with 
$\ep \in (0 ,\ep_0]$ 
satisfies the equality in the isoperimetric inequality, as claimed. 
In particular, from the same computations as in 
Subsection \ref{sub:hui}, for any fixed $\ep\in(0,\ep_0]$, 
we have that
\begin{equation}
\label{will-mcf}
0 \,\,= \,\, \frac{\dd D}{\dd\ep}(\ep) \,\, \leq \,\,  \left(\left\vert\Sigma_{\epsilon}\right\vert
\int_{\Sigma_{\epsilon}} \HH^2\dd\sigma\right)^{1/2}
\left[ 
\sqrt{36\pi\AVR}\,
 \, - \, \frac{3}{2}\left(\int_{\Sigma_{\epsilon}}{\HH}^2 \dd\sigma\right)^{\!\!1/2}\right] \,\, \leq \,\, 0 \, .
\end{equation}
This implies that the equality sign holds in the Willmore inequality for $\Sigma_\ep$, and thus, by the rigidity statement in Theorem~\ref{willth}, the exterior of $\Sigma_\ep$ in $(M, g)$ is isometric to a cone. More precisely, if we call $\Omega_\ep$ the (open and bounded) region enclosed by $\Sigma_\ep$, we have that $(M \setminus \Omega_\ep)$ is isometric to
\begin{equation}
\label{eq:cono_ep}
\Big(\,\big[r_\ep, +\infty) \times\Sigma_\epsilon
\,,\, 
\dd r\otimes\!\dd r +(r/r_\ep)^2 g_{\Sigma_\epsilon}
\Big),
\qquad
\mbox{where}
\quad
r_\ep
\,=\,
\bigg(\frac{|\Sigma_\epsilon|}{4\pi \AVR}\bigg)^{1/{2}}.
\end{equation}
Hence, it is easily seen that the MCF $\{ \Sigma_\ep\}_{\ep \in (0, \ep_0]}$ is given by totally umbilic hypersurfaces coinciding with the cross sections of the cone 
\begin{equation}
\label{eq:omega_ep}
(M \setminus \Omega_\ep, g) \cong \Big(\,\big[r_\ep, +\infty) \times\pa \Om^*
\,,\, 
\dd r\otimes\!\dd r + \frac{4\pi\AVR }{\abs{\pa \Om^*}}r^2 g_{\pa\Om^*}
\Big) \, .
\end{equation}


We now claim that the MCF $\{ \Sigma_\ep\}_{\ep>0}$ does not develop singularities before the extinction time~$\ep^*$. Letting $(0,\ep_*)$ be the maximal interval of existence of the smooth MCF, we claim that $\e_*=\ep^*$. In fact, from~\eqref{eq:omega_ep}
one can easily see that the mean curvature of $\Sigma_\ep$ is given by $(n-1)/r_\ep$, 
and in turn the squared norm of its second fundamental form is equal to $(n-1)/r^2_\ep$. It follows then by~\cite[Theorem 7.1]{Huisken1986} that $\ep_*$ is such that $r_{\ep_*} \!=0$, and thus coincides with the extinction time of the flow, i.e., $\ep^* = \ep_*$.
We have hence deduced the isometry
\[
(M \setminus{\{O\}}, g) \cong \Big(\,\big(0, +\infty) \times\pa \Om^*
\,,\, 
\dd r\otimes\!\dd r + \frac{4 \pi \AVR}{\abs{\pa \Om^*}}r^2 g_{\pa\Om^*}
\Big),
\]
for some $O \in M$. In particular, the surface area of the geodesic balls centered at $O$ decays as 
$4\pi r^{2} \AVR$, and, since $g$ is smooth at $O$, this implies that $\AVR =1$. By Bishop-Gromov, we infer that  $(M, g)$ is isometric to $(\R^3, g_{\R^3})$ and $\pa\Om^*$ is isometric to a sphere. This implies that $\Omega=\Omega^*$, since, otherwise, the mean curvature of $\partial \Omega^*$ would be null on the points not belonging to $\partial \Omega$ (compare with \cite[(1.15)]{Hui_Ilm}), leading to a contradiction. We have thus shown that $\Omega$ is a ball, completing the proof. \end{proof}

\section*{Appendix: 
comparison with the monotonicity formulas by Colding and Minicozzi}  
\label{sec:comparison}


In this section, we provide a comparison 
between our monotonicity formulas and those obtained
by Colding and by Colding-Minicozzi
in \cite{Colding_1} and \cite{Colding_Minicozzi_2},
respectively.
To start with, let $u$ be a solution of \eqref{pb} in a nonparabolic Riemannian manifold $(M, g)$ with $\ric \geq 0$, for a bounded subset $\Omega\subset M$ with smooth boundary, 
and set 
\begin{equation}
\label{def_b}
b
\,=\,
u^{-\frac1{n-2}}.
\end{equation}
Note that $b=1$ on $\pa\Om$ and that $b\to+\infty$ at infinity.
Associated with the level sets of $b$,
consider the family of functions
$\{A_\beta\}$, where 
$A_\beta:[1,+\infty)\to[0,+\infty)$
is defined for every $\beta\geq0$ as
\begin{equation*}
A_\beta(r)
\,=\,
\frac1{r^{n-1}}
\!\!\!
\int\limits_{\{b=r\}}
\!\!\!
|\D b|^{\beta+1}\dd\sigma.
\end{equation*}
Now, replacing $u$ be a minimal Green's function $G(O, \cdot)$, for some pole $O \in M$,
the above defined $A_\beta$ is exactly the 
quantity considered in \cite[formula (1.1)]{Colding_Minicozzi_2}.
Note that in Colding's setting, the level sets $\{b=r\}$
are considered for every $r>0$, since $b (q) \to 0$ as $d(O, q) \to 0$.

Our aim is to see how the monotonicity of our family
of functions $\{\Phi_\beta\}$ translates in terms of the 
family $\{A_\beta\}$. First of all, it is 
straightforward from \eqref{gconf}--\eqref{def:ffi} 
and \eqref{def_b} that
\begin{equation}
\label{b-phi}
b
\,=\,
{\rm e}^{\frac\ffi{n-2}},
\qquad\qquad
|\D b|
\,=\,
\frac{|\na\ffi|_{\cg}}{n-2},
\qquad\qquad
\dd\sigma
\,=\,
b^{n-1}
\dd\sigma_{\cg},
\end{equation}
and in turn that
\begin{equation}
\label{identita_Col_noi}
\Phi_\beta(s)
\,=\,
(n-2)^{\beta+1}
A_\beta\big({\rm e}^{\frac s{n-2}}\big),
\qquad\qquad\mbox{for every }s\geq0.
\end{equation}
We look at the derivative
\eqref{derivata_di_Phi}
of $\Phi_\beta$ and at its equivalent expression
\eqref{eq:der_fip}.
In particular, the volume integral \eqref{eq:der_fip}
contains the following terms.
\begin{equation}
\label{ric_b}
\ric(\na\ffi,\na\ffi)
\,=\,
\left(\dfrac{n-2}{2}\right)^2\,
\ric(\D b^2,\D b^2),
\end{equation}
and
\begin{align}
\label{b_quadratone}
\big|\na\na\ffi\big|_{\cg}^2
+
(\beta-2)\big|\na|\na\ffi|_{\cg}\big|_{\cg}^2
\,=\,
\left(\textstyle{\dfrac{n-2}2}\right)^{\!2}
&
\bigg\{
\Big|\D\D\,b^2-\frac{\textstyle{\Delta b^2}}n\,g\Big|^2
\nonumber\\
&+
(\beta-2)\,\big|\D^T|\D b|\big|^2
\nonumber\\
&+
(\beta-2)\,\big|\D b^2\big|^2\Big[\HH-(n-1)\big|\D\log b\big|\Big]^2
\bigg\}
\end{align}
which have been expressed in terms of 
the function $b$ and of the metric $g$
via some computations (compare with the proof of \eqref{eq:monot}) .
Differentiating both sides of
\eqref{identita_Col_noi} and writing expression \eqref{eq:der_fip} in terms of $b$ and $g$ through 
formulas \eqref{b-phi}, \eqref{ric_b} and \eqref{b_quadratone}, 
we obtain
\begin{align}
\label{cold-noi}
\frac{\dd A_\beta}{\dd r}(r)
\,&=\,
\frac{(n-2)^{-\beta}}r
\frac{\dd\Phi_\beta}{\dd s}\big((n-2)\log r\big)\\
&=\,
-\frac\beta4\,r^{n-3}
\int_{\{b>r\}}
\!\!\!
|\D b|^{\beta-2}
\!
\bigg\{\ric (\D b^2, \D b^2) + 
\Big|\D\D b^2-\frac{\Delta b^2}n g\Big|^2
\\
&\phantom2\hspace{4cm}+
(\beta-2)\,\big|\D^T|\D b|\big|^2
\\
&\phantom2\hspace{4cm}+
(\beta-2)\,\big|\D b^2\big|^2\Big[\HH-(n-1)\big|\D\log b\big|\Big]^2
\bigg\}
\,b^{2-2n}
\,\dd\mu
\,\leq\,0\,.
\end{align}
Setting $b = 2$ in the above formula, we obtain exactly the integrand of the right hand side of \cite[(2.106)]{Colding_1}, that, arguing as  in the conclusion of the present Theorem \ref{thm:main_conf}, leads to the monotonicity of $A_2$.
For a general $\beta \geq (n-2)/(n-1)$, in \cite[Theorem 1.3]{Colding_Minicozzi} the monotonicity of $A_\beta$ is inferred grouping the terms in \eqref{cold-noi} in a different way. Observe indeed that for $\beta < 2$ the volume integral in \eqref{cold-noi} does not evidently carry a sign. On the other hand, \eqref{b_quadratone} combined with Kato's inequality immediately show the nonnegativity of the expression.
\bigskip

We close this appendix by showing how our methods can be applied also to obtain the Monotonicity-Rigidity Theorem for the Green's function, obtaining a new (conformal) proof of Colding-Minicozzi's \cite[Theorem 1.3]{Colding_Minicozzi}. 
Indeed, let $(M, g)$ be a nonparabolic Riemannian manifold with $\ric \geq 0$, let $G$ be its minimal Green's function and consider the new metric on $M \setminus \{O\}$
\[
\tilde{g}\, = \, G(O, \cdot)^{\frac{2}{n-2}} g,
\]
for some point $O \in M$.
Set 
\[
\phi \,= \, - \log G(O, \cdot).
\]
Then, we have that the triple $M, \tilde{g}, \phi$ satisfies
the system
\begin{equation}
\label{pb-green}
\begin{cases}
\,\,\,\,\,\,\,\,\,\,\,\,\,\,\,\,\,\,\,\,\,\,\,\,\,\,\,\,\,\,\,\,\,\,\,\,\,\,\,\,\,\,\,\,\,\,\,\,\,\Delta_{\tilde{g}}\phi \, = \, 0 & \mbox{in} \,\, M\setminus\{ O \} \\
\ric_{\tilde{g}} - \nabla\nabla\phi + \dfrac{d\phi\otimes d\phi}{n-2} \, = \, \dfrac{\abs{\nabla\phi}_{\cg}^2}{n-2}\cg +\ric & \mbox{in} \,\, M\setminus\{O\} \\
\,\,\,\,\,\,\,\,\,\,\,\,\,\,\,\,\,\,\,\,\,\,\,\,\,\,\,\,\,\,\,\,\,\,\,\,\,\,\,\,\,\,\,\,\,\,\,\,\,\,\phi(q)\to + \infty & \mbox{as} \,\, d(O, q) \, \to \, + \infty \\
\,\,\,\,\,\,\,\,\,\,\,\,\,\,\,\,\,\,\,\,\,\,\,\,\,\,\,\,\,\,\,\,\,\,\,\,\,\,\,\,\,\,\,\,\,\,\,\,\,\,\phi(q)\,\to\, - \infty & \mbox{as} \,\, d(O, q)\to 0.
\end{cases}
\end{equation}
We denote by $d$ the distance with respect to $g$. Define the function $ \Phi_\beta : \R \mapsto \R$ given by
\begin{equation}
\label{eq:fip-green}
\Phi_\beta(s) 
\,\,=\!\!\!
\int\limits_{\{\ffi = s\}}\!\!\!
|\na \ffi|_{\cg}^{\beta+1} \dd\sigma_{\cg}.
\end{equation}

All the theory developed in Section 3 holds with trivial modification for $\Phi_\beta$  as above, and immediately yields a conformal Monotonicity-Rigidity Theorem for the Green's function.

\begin{theorem}
Let $(M, g)$ be a nonparabolic Riemannian manifold with $\ric \geq 0$. Let $G$ be its minimal Green's function. Then, with the notations above, we have
\begin{equation}
\label{eq:der_fip-green}
\frac{\dd\Phi_\beta}{\dd s}(s)  \,\, 
= -\,\,\beta\,\,{\rm e}^s\!\!\!
\int\limits_{\{\ffi \geq  s\}} \!\!
\frac{
|\na\ffi|_{\cg}^{\beta-2} 
\Big(\ric(\na\ffi,\na\ffi)
+\big|\na\na \ffi\big|_{\cg}^2
 +  \, (\beta-2) \, \big| \na |\na \ffi |_{\cg}\big|_{\cg}^2\,\Big) 
}
 {{\rm e}^\ffi}
\,\,\dd\mu_{\cg} 
\end{equation}
In particular, $\Phi_\beta'$ is alway nonpositive. 
Moreover, $(\dd\Phi_\beta/\dd s)(s_0)=0$
for some $s_0\in \R$ and some $\beta\geq (n-2)/(n-1)$
if and only if $\{\ffi\geq s_0\}$ is isometric  
to the Riemannian product  $\big([s_0, \infty) \times \{\phi = s_0\}, d\rho\otimes d\rho + \tilde{g}_{\{\mid \phi = s_0\}}\big)$.
\end{theorem}
The above Theorem clearly translates in terms of $(M, g)$ and $G$, exactly as Theorem \ref{thm:M/R} was deduced from Theorem \ref{thm:main_conf}.


\subsection*{Acknowledgements}
\emph{The author are grateful to C.~Arezzo, A.~Carlotto, A.~Farina, G.~Huisken, L.~Mari, D.~Peralta-Salas, F.~Schulze and P.~Topping  for useful comments and discussions during the preparation of the paper. The authors are members of Gruppo Nazionale per
l'Analisi Matematica, la Probabilit\`a e le loro Applicazioni (GNAMPA),
which is part of the Istituto Nazionale di Alta Matematica (INdAM).
The
paper was partially completed during the authors' attendance to the program “Geometry and relativity” organized by the Erwin Schroedinger International Institute
for Mathematics and Physics (ESI).}



\bibliographystyle{plain}


\end{document}